\providecommand{\MRR}{\relax\ifhmode\unskip\space\fi MR }
\providecommand{\href}[2]{#2}

\documentclass{amsart}
\usepackage[utf8]{inputenc}

\usepackage{geometry}
\usepackage{epsdice}

\usepackage{enumitem}

\usepackage{color}
   
\usepackage{stmaryrd}
\usepackage[normalem]{ulem}

\usepackage[OT2,T1]{fontenc}

\usepackage[final]{microtype}

\usepackage{tikz}
\usepackage{tikz-cd}
\usepackage{xypic}
\usepackage{relsize}
\usepackage[bbgreekl]{mathbbol}
\usepackage{amsfonts}
\DeclareSymbolFontAlphabet{\mathbb}{AMSb} 
\DeclareMathAlphabet{\mathpzc}{OT1}{pzc}{m}{it}
\DeclareSymbolFontAlphabet{\mathbbl}{bbold}
\newcommand{\Prism}{{\mathlarger{\mathbbl{\Delta}}}}
\usepackage{color}
\usepackage{hyperref}
\hypersetup{colorlinks,linkcolor={red},citecolor={blue},urlcolor={red}} 

\usepackage[capitalise]{cleveref}

\usepackage{mathtools}
\usepackage{amsmath}
\usepackage{amssymb}
\usepackage[all,cmtip]{xy}
\usepackage{mathrsfs}

\usepackage{amsthm,amssymb}
\numberwithin{equation}{section}

\theoremstyle{plain}
\newtheorem{theorem}[equation]{Theorem}
\newtheorem{conjecture}[equation]{Conjecture}
\newtheorem{prop}[equation]{Proposition}
\newtheorem{lemma}[equation]{Lemma}

\newtheorem{cor}[equation]{Corollary}

\theoremstyle{definition}
\newtheorem{defi}[equation]{Definition}

\newtheorem{example}[equation]{Example}
\newtheorem{notation}[equation]{Notation}

\newtheorem{remark}[equation]{Remark}

\DeclareMathOperator{\Spf}{Spf}

\newcommand{\an}{{\mathrm{an}}}

\newcommand{\dR}{{\mathrm{dR}}}
\newcommand{\fkt}{\mathfrak{t}}
\newcommand{\mat}{\text{Mat}}
\newcommand{\bdr}{{\mathbb{B}_{\mathrm{dR}}^{+}}}
\newcommand{\rb}{\mathrm{Rep}_{G_K}(B_{\dR}^+)}

\newcommand{\db}{\mathrm{Rep}_{\hat{G}}(B_{\dR,L}^+)}

\DeclareMathOperator{\Spec}{Spec}
\DeclareMathOperator{\Tot}{Tot}

\renewcommand{\phi}{\varphi}
\renewcommand{\epsilon}{\varepsilon}

\DeclareMathOperator{\Vect}{Vect}

\newcommand{\Mod}{\mathrm{Mod}}

\newcommand{\GL}{{\mathrm{GL}}}

\DeclareMathOperator{\Rep}{Rep}

\newcommand{\crys}{{\mathrm{crys}}}
\newcommand{\perf}{{\mathrm{perf}}}
\newcommand{\fp}{{\mathrm{fp}}}

\DeclareMathOperator{\Gal}{Gal}

\DeclareMathOperator{\Hom}{Hom}

\DeclareMathOperator{\Eq}{Eq}

\DeclareMathOperator{\Id}{Id}

\newcommand{\Shv}{{\mathcal S\mathrm{hv}}}

\newcommand{\mr}{\mathrm}

\newcommand{\Ker}{\mr{Ker}}

\newcommand{\Ainf}{{A_{\mathrm{inf}}}}

\usepackage{tikz}

\newcommand{\LHS}{\mathrm{LHS}}
\newcommand{\RHS}{\mathrm{RHS}}

\title{De Rham prismatic crystals over $\mathcal{O}_K$}

\author{Zeyu Liu}

\address{Department of Mathematics,University of California, San Diego, 9500 Gilman Dr. La Jolla, CA 92093}

\email{zeliu@ucsd.edu}

\begin{document}

\maketitle

\begin{abstract}
We study de Rham prismatic crystals on $(\mathcal{O}_K)_{\Prism}$. We show that a de Rham crystal is controlled by a sequence of matrices $\{A_{m,1}\}_{m \geq 0}$ with $A_{0,1}$ "nilpotent". Using this, we prove that the natural functor from the category of de Rham crystals over $(\mathcal{O}_K)_{\Prism}$ to the category of nearly de Rham representations is fully faithful. The key ingredient is a Sen style decompletion theorem for $B_{\dR}^+$-representations of $G_K$.
\end{abstract}

\tableofcontents

\section{Introduction}
Once constructed by Bhatt and Scholze in \cite{BS19}, prismatic cohomology theory plays a central role in the study of $p$-adic cohomologies as it can naturally serve as a bridge connecting various classical $p$-adic cohomology theories via specialization along different directions. Motivated by such "universal" property, it is natural to expect that various coefficients in the prismatic site of a $p$-adic formal scheme could recover some coefficients in other $p$-adic cohomology theories. Recently, there are fruitful results in this direction. Let us give a quick review here.

Let $X$ be a $p$-adic formal scheme. Consider the absolute prismatic site $X_{\Prism}$. Recall that prismatic crystals with various coefficients can be defined in a universal way:

For $*\in \{\mathcal{O}_{\Prism},\overline{\mathcal{O}}_{\Prism},\overline{\mathcal{O}}_{\Prism}[\frac{1}{p}], (\mathcal{O}_{\Prism}[\frac{1}{p}])^{\wedge}_{\mathcal{I}} \}$,
\begin{defi}
An abelian sheaf $\mathcal{F}$ on $X_{\Prism}$ is called a $(*)$-crystal if for any $(A,I)\in X_{\Prism}$, $\mathcal{F}(A,I)$ is a finite projective $*(A)$-module such that for any morphism $(A,I)\rightarrow (B,J)$, there is a canonical isomorphism
\[\mathcal{F}(A,I)\otimes_{*(A)}*(B)\cong \mathcal{F}(B,J).\]
Moreover, the category of such $\mathcal{F}$ is denoted as $\Vect(X_{\Prism},*)$.
\end{defi}
In this way, we get
\begin{itemize}
    \item Prismatic crystal if $*=\mathcal{O}_{\Prism}$.
    \item Hodge-Tate (prismatic) crystal if $*=\overline{\mathcal{O}}_{\Prism}$.
    \item Rational Hodge-Tate (prismatic) crystal if $*=\overline{\mathcal{O}}_{\Prism}[\frac{1}{p}]$.
    \item De Rham (prismatic) crystal if $*=(\mathcal{O}_{\Prism}[\frac{1}{p}])^{\wedge}_{\mathcal{I}}$.
\end{itemize}
Moreover, if we add the Frobenius structure, then we will get prismatic $F$-crystals:

For $*\in \{\mathcal{O}_{\Prism}, (\mathcal{O}_{\Prism}[\frac{1}{\mathcal{I}}])^{\wedge}_{p} \}$,
\begin{defi}
A prismatic $F$-crystal in $*$-modules is a pair $(\mathcal{E},\varphi_{\mathcal{E}})$ such that
\begin{itemize}
    \item $\mathcal{E}$ is a $(*)$-crystal;
    \item $\varphi_{\mathcal{E}}$ is an identification 
    \[\varphi_{\mathcal{E}}: (\varphi^{*}\mathcal{E})[1/\mathcal{I}]\cong \mathcal{E}[1/\mathcal{I}].\]
\end{itemize}
The category of such $\mathcal{F}$ is denoted as $\Vect^{\varphi}(X_{\Prism},*)$.
\end{defi}
In this way, we get prismatic $F$-crystals if $*=\mathcal{O}_{\Prism}$ and Laurent prismatic $F$-crystals if taking $*=(\mathcal{O}_{\Prism}[\frac{1}{\mathcal{I}}])^{\wedge}_{p}$.

Heuristically those $F$-crystals were first studied as they naturally correspond to some nice integral $p$-adic Galois representations over $\mathbb{Q}_p$ of interest by the work of Bhatt and Scholze (Frobenius is a very strong condition to guarantee that we end up with something defined over $\mathbb{Q}_p$ or a finite extension $K/\mathbb{Q}_p$). On the other hand, if we consider those prismatic crystals without Frobenius, typically we could only get something defined over $\mathbb{C}_p$, or equivalently, over some large highly ramified fields (eg. $K_{p^{\infty}}$ and $K_{\infty}$) by Faltings almost purity theorem and decompletion theory.

To concretely relate various prismatic crystals to $p$-adic Galois representations, the story starts from the following key theorem due to Bhatt and Scholze, which was also proven by Du and Liu independently:
\begin{theorem}[\cite{BS21},\cite{DL21}]
\[\xymatrix{\Vect^{\varphi}(X_{\Prism},\mathcal{O}_{\Prism})\ar[d]^{\cong}\ar[r]& \Vect^{\varphi}(X_{\Prism},(\mathcal{O}_{\Prism}[\frac{1}{\mathcal{I}}])^{\wedge}_{p})\ar[d]_{T}^{\cong}
\\\Rep_{\mathbb{Z}_p}^{\crys}(G_K)\ar[r]&\Rep_{\mathbb{Z}_p}(G_K)}\]
\end{theorem}
Here $K$ is a complete discretely valued field with ring of integers $\mathcal{O}_K$ of mixed characteristic with perfect residue field $k$ and $X=\Spf \mathcal{O}_K$. $G_K$ is the absolute Galois group of $K$. $\Rep_{\mathbb{Z}_p}(G_K)$ is defined to be the category of finite free $\mathbb{Z}_p$-representations of $G_K$, while $\Rep_{\mathbb{Z}_p}^{\crys}(G_K)$ is the full subcategory consisting of crystalline $\mathbb{Z}_p$-representations of $G_K$.

From now on we fix a uniformizer $\pi$ of $\mathcal{O}_K$ with Eisenstein polynomial $E(u)\in W(k)[[u]]$. $(\mathfrak{S},E)=(W(k)[[u]], E(u))$ is the Breuil-Kisin prism.
\begin{remark}
\begin{itemize}
    \item $T$ which is called \'etale realization functor, was first proven to be an equivalence independently by Zhiyou Wu. Concretely, it is given by taking $\varphi$-invariants of the evaluation of $\mathcal{E}$ at $\Ainf$ (By work of Kedlaya and Liu, the category of \'etale $\varphi$-modules over $W(\mathbb{C}_p^\flat)$ is equivalent to $\Rep_{\mathbb{Z}_p}(G_K)$). Moreover, the crystal condition can be transferred to give a $G_K$-action (key observation of Wu). 
    
    Also, by evaluating $\mathcal{E}$ at some other prisms $\mathbf{A}_{K}^+$ (which could be viewed as certain "deperfection" of $\Ainf$), Wu gives a new proof of Fontaine's classical result that $\Rep_{\mathbb{Z}_p}(G_K)$ is equivalent to $(\varphi,\Gamma)$-modules over $\mathbf{A}_{K}$. We refer the readers to \cite{Wu21} for details.
    \item In Bhatt-Scholze's work, they show that the image of $T$ when restricted to  $\Vect^{\varphi}(X_{\Prism},\mathcal{O}_{\Prism})$ naturally enjoys crystalline property first. The genuine difficulty is to prove essential surjectivity. To do that, given a crystalline $\mathbb{Z}_p$-representation, Bhatt and Scholze first construct a crystal over some larger rings and then show the descent data is bounded to actually descend to a prismatic crystal, which relies heavily on using Beilinson fiber sequence to calculate Frobenius invariant elements in certain large rings.
    \item On the other hand, Du and Liu prove the equivalence in a different method. Their study is based on the Kummer tower. In particular, they prove an \'etale realization in this setting and then give a descent data via bounding the matrix on which Frobenius acts via, relying on the calculation of crystalline condition.
    \item Later all of the results are generalized to the relative case. Min-Wang prove that $T$ is an equivalence in the relative setting in \cite{MW21a}. Du-Liu-Moon-Shimizu and Guo-Reinecke prove the following result independently:
    \end{itemize}
\end{remark}
\begin{theorem}[\cite{GR22},\cite{DLMS22}]
    For $X$ a smooth formal scheme over $\mathcal{O}_K$
\[\Vect^{\an,\varphi}(X_{\Prism},\mathcal{O}_{\Prism})\stackrel{\cong}{\longrightarrow} \Rep_{\mathbb{Z}_p}^{\crys}(X_{\eta}).\]
\end{theorem}
Here the right handside is crystalline $\mathbb{Z}_p$-local systems on the generic fiber of $X$, while the left handside is the so called analytic prismatic $F$-crystals consisting of crystals $\mathcal{F}$ on $X_{\Prism}$ equipped with Frobenius structures whose evaluations at a prism $(A,I)\in X_{\Prism}$ are vector bundles on the analytic locus of $\Spec(A)$, we refer the readers to \cite[Remark 1.4]{GR22} for details.

On the other hand, one might ask what's the story if one drops the Frobenius structure and works with prismatic crystals. 

Towards this direction, Yu Min and Yupeng Wang study Hodge-Tate crystals on $\mathcal{O}_K$ in \cite{MW21}. Later Hui Gao relates rational Hodge-Tate crystals on $\mathcal{O}_K$ to nearly Hodge-Tate representations in \cite{Gao22}. Let us review their work here. For simplicity we assume $p>2$ from now on.

Utilizing the Breuil-Kisin prism as the chosen weakly final object in $(\mathcal{O}_K)_{\Prism}$ and by identifying $\mathfrak{S}^1/E$ with $\mathcal{O}_K\{X_1\}^{\wedge}_{\rm pd}$, which is involved in the stratification of the Breuil-Kisin prism, Min-Wang prove the following result (see \cite[Theorem 3.5]{MW21}):
\begin{theorem}
    To give a rank $l$ Hodge-Tate crystal $\mathbb{M}$, it is the same as specifying a pair $(M,\varepsilon)$, where $M$ is a finite free $\mathfrak{S}$-module $M$ of rank $l$ and the stratification $\varepsilon$ is an $\mathcal{O}_K\{X_1\}^{\wedge}_{\rm pd}$-linear isomorphism 
$$\varepsilon: M\otimes_{\mathcal{O}_K}  \mathcal{O}_K\{X_1\}^{\wedge}_{\rm pd} \stackrel{\cong}{\rightarrow} M\otimes_{\mathcal{O}_K}  \mathcal{O}_K\{X_1\}^{\wedge}_{\rm pd}. $$
Moreover, $\varepsilon$ is uniquely determined by a matrix $A_1\in M_{l\times l}(\mathcal{O}_K)$ such that $\lim_{n\to+\infty}\prod_{i=0}^n(iE'(\pi)+A_{1}) = 0$ in the sense that
\begin{equation*}
    \varepsilon(\underline e)=\underline e \sum_{n=0}^{\infty}A_n X_1^{[n]}
\end{equation*}
with $A_0=I$, $A_n=\prod_{i=0}^{n-1}(iE'(\pi)+A_{1})$ for $n\geq 1$.
\end{theorem}
Then essentially by proving that the Sen operator of the $\mathbb{C}_p$-representation of $G_K$ associated to such a (rational) Hodge-Tate crystal is $-\frac{A_1}{\beta}$ for $\beta=E^{\prime}(\pi)$, which was conjectured by Min and Wang, Gao constructs the following diagram (see \cite[Thm 1.1.3]{Gao22}):
\begin{theorem}\label{Gao theorem}
    \[\xymatrix{\operatorname{Vect}((\mathcal{O}_{K})_{\Prism}, \overline{\mathcal{O}}_{\Prism}[\frac{1}{p}])\ar[d]^{\cong}\ar[r]^{}& \operatorname{Vect}((\mathcal{O}_{K})_{\Prism}^{\perf}, \overline{\mathcal{O}}_{\Prism}[\frac{1}{p}])\ar[d]_{}^{\cong}
\\\Rep_{\mathbb{C}_p}^{nHT}(G_K)  \ar@{^{(}->}[r]&\Rep_{\mathbb{C}_p}^{}(G_K)}\]
Here $\Rep_{\mathbb{C}_p}^{nHT}(G_K)$ is defined by Gao, which is the full subcategory of $\Rep_{\mathbb{C}_p}^{}(G_K)$ consisting of those $W\in \Rep_{\mathbb{C}_p}^{}(G_K)$ with Sen weights in the subset $\mathbb{Z}+E^{\prime}(\pi)^{-1}\mathfrak{m}_{\mathcal{O}_{\overline{K}}}$.
\end{theorem}
Then Gao asked a natural question that whether we can deform this diagram by replacing the rational Hodge-Tate crystals with de Rham crystals in this picture, see \cite[Remark 1.1.4]{Gao22}. 

This is one of the motivations of our study on the de Rham prismatic crystals over $(\mathcal{O}_{K})_{\Prism}$ and we wish to give a partial answer to this question. To state our result, we need some notations:
\begin{notation}
\begin{itemize}
    \item $\Rep_{B_{\dR}^+}^{\fp}(G_K)$ is defined to be the category of finite projective $B_{\dR}^+$-modules equipped with a semi-linear $G_K$ action.
    \item $M\in \Rep_{B_{\dR}^+}^{\fp}(G_K)$ is called \textit{nearly de Rham} if $M/tM\in \Rep_{\mathbb{C}_p}^{nHT}(G_K)$, i.e. all of the Sen weights of $M/tM$ are in the subset $\mathbb{Z}+E^{\prime}(\pi)^{-1}\mathfrak{m}_{\mathcal{O}_{\overline{K}}}$. (In this paper $t$ is identified with $E([\pi^{\flat}])$, where $[\pi^{\flat}]$ is defined in the end of this section).  Denote $\Rep_{B_{\dR}^+}^{\fp,ndR}(G_K)$ to be the full subcategory of $\Rep_{B_{\dR}^+}^{\fp}(G_K)$ consisting of such $M$.
\end{itemize}
\end{notation}

Now we are ready to state our main result:
\begin{theorem}\label{main commutative diagram}
    We have a commutative diagram
    \[\xymatrix{\operatorname{Vect}((\mathcal{O}_{K})_{\Prism}, (\mathcal{O}_{\Prism}[\frac{1}{p}])_{\mathcal{I}}^{\wedge})\ar[d]^{}\ar[r]^{V}& \operatorname{Vect}((\mathcal{O}_{K})_{\Prism}^{\perf}, (\mathcal{O}_{\Prism}[\frac{1}{p}])_{\mathcal{I}}^{\wedge})\ar[d]_{}^{T}
\\\Rep_{B_{\dR}^+}^{\fp,ndR}(G_K)  \ar@{^{(}->}[r]&\Rep_{B_{\dR}^+}^{\fp}(G_K)}\]
such that $V$ is fully faithful and $T$ is an equivalence. Here $V$ is given by restricting a de Rham prismatic crystal to the perfect prismatic site and $T$ is the evaluation at the $\Ainf$ prism. We will call $T$ the \textbf{de Rham realization functor} later.
\end{theorem}
As a corollary, if we abuse the notation by still using $V$ to denote the composition of $T$ and $V$, then we have that:
\begin{cor}
$V: \operatorname{Vect}((\mathcal{O}_{K})_{\Prism}, (\mathcal{O}_{\Prism}[\frac{1}{p}])_{\mathcal{I}}^{\wedge})\longrightarrow \Rep_{B_{\dR}^+}^{\fp,ndR}(G_K)$ is fully faithful. 
\end{cor}

We now explain how to prove \cref{main commutative diagram}. 
Once the basic properties of $V$ and $T$ are established, then the image of $V$ naturally lands in $\Rep_{B_{\dR}^+}^{\fp,ndR}(G_K)$ as a de Rham crystal naturally specializes to a rational Hodge-Tate crystal (see \cref{specializing Hodge Tate} for details) and rational Hodge-Tate crystals are realized to nearly Hodge-Tate representations by \cref{Gao theorem}.

To study the behavior of $V$ and $T$, we need several preliminaries.

First we give an explicit description of de Rham crystals in terms of a stratification of the \v{C}ech nerve associated to the Breuil-Kisin prism by identifying $\bdr(\mathfrak{S}^{1})$ (resp. $\bdr(\mathfrak{S}^{1})$) with $K[[T]]$ (resp. $K\{X_1\}^{\wedge}_{\rm pd}[[T]]$) (see \cref{Equ-structure} for details), then we have the following:
\begin{theorem}[\cref{Main theorem I}]
    Let $M$ be a finite free $\bdr(\mathfrak{S})$-module and $\varepsilon$ be an $\mathcal{O}_K\{X_1\}^{\wedge}_{\rm pd}$-linear isomorphism 
$$\varepsilon: M\otimes_{K[[T]],\delta_{0}^{1}} K\{X_1\}^{\wedge}_{\rm pd}[[T]] \stackrel{\sim}{\rightarrow} M\otimes_{K[[T]],\delta_{1}^{1}} K\{X_1\}^{\wedge}_{\rm pd}[[T]].$$
  Fix a basis $\underline e$ of $M$ and write $\epsilon(\underline{e})=\underline e\cdot \sum_{m\geq0}(\sum_{n\geq 0}A_{m,n}X^{[n]})t^m$, then if $(M,\varepsilon)$ is induced from a de Rham prismatic crystal $\mathcal{M} \in \operatorname{Vect}((\mathcal{O}_{K})_{\Prism}, (\mathcal{O}_{\Prism}[\frac{1}{p}])_{\mathcal{I}}^{\wedge})$, the following holds:
    \begin{itemize}
        \item $A_{0,0}=I$ and $A_{i,0}=0$ for $i>0$.
        \item $A_{0,1} \in M_l(K)$ satisfies that $\lim_{n\to+\infty}\prod_{i=0}^n(iE'(\pi)+A_{0,1}) = 0$.
        \item $A_{m,n+1}=(\beta(n-m)+A_{0,1})A_{m,n}+\sum_{\substack{i+j=m\\i\leq m-1}}(A_{j,1}+(n-i)\theta_{1,j})A_{i,n}$ for $m,n\in \mathbb{N}^{\geq 0}$. In particular, $\{A_{m,n}\}$ is determined by $\{A_{m,1}\}_{m\geq 0}$.
    \end{itemize}
    where $\theta_{1,j}$ is defined in \cref{image of t}.
\end{theorem}
Actually we conjecture that the reverse is also true:
\begin{conjecture}[\cref{conjecture e=1}]\label{restate e=1 conjecture}
 Assume that $\{B_{m,1}\}$ is a sequence of matrices in $ M_l(K)$ such that
    \begin{itemize}
        \item $\lim_{n\to+\infty}\prod_{i=0}^n(iE'(\pi)+B_{0,1}) = 0$.
    \end{itemize}
    Then if we define a sequence of matrices
    $\{A_{m,n}\}$ via following:
\begin{itemize}
        \item $A_{0,0}=I$ and $A_{i,0}=0$ for $i>0$, $A_{m,1}=B_{m,1}$.
        \item $A_{m,n+1}=(\beta(n-m)+A_{0,1})A_{m,n}+\sum_{\substack{i+j=m\\i\leq m-1}}(A_{j,1}+(n-i)\theta_{1,j})A_{i,n}$ for $m,n\in \mathbb{N}^{\geq 0}$.
    \end{itemize}
then 
\begin{equation*}
    \epsilon(\underline{e})=\underline e\cdot \sum_{m\geq0}(\sum_{n\geq 0}A_{m,n}X^{[n]})t^m.
\end{equation*}
can serve as the stratification associated to certain prismatic de Rham crystal, i.e. the cocycle condition is satisfied.
\end{conjecture}

To show that $V$ is fully faithful, essentially it suffices to show that $H^0((\mathcal{O}_{K})_{\Prism},\mathcal{M})\cong \mathcal{M}(\Ainf)^{G_K}$. The idea is to write $\Vec{v}\in \mathcal{M}(\Ainf)$ as linear combinations of $\underline e$ with coefficients in $B_{\dR}^+\cong \mathbb{C}_p[[t]]$, then characterize the Galois invariant action. However, the main difficulty is that $B_{\dR}^+\cong \mathbb{C}_p[[t]]$ is only an abstract isomorphism, but couldn't be Galois equivariant! To solve this problem, we want to control the coefficient of $\Vec{v}$ provided that $\Vec{v}$ is Galois invariant. The key ingredient is the following Sen style decompletion theorem (see the end of this section for notations):
\begin{theorem}[\cref{combine three theorems}]
    Given $W\in \Rep_{B_{\dR}^+}^{\fp}(G_K)$ of rank $d$, define
\[D_{\mathrm{Sen},K_{\infty}[[t]]}(W)=(W^{G_L})^{\hat{G}-la,\gamma=1}.\]
Then this is a finite projective $K_{\infty}[[t]]$-module of rank $d$ such that 
\[D_{\mathrm{Sen},K_{\infty}[[t]]}(W)\otimes_{K_{\infty}[[t]]}B_{\dR}^{+}=W.\]
\end{theorem}
Here $(W^{G_L})^{\hat{G}-la,\gamma=1}$ is the subspace of $W$ consisting of those vectors which are $G_L$-invariant, $\hat{G}$-locally analytic and $\Gal(L/K_{\infty})$-invariant, see \cref{de rham analytic vectors} for details.

For those $W$ associated to a de Rham crystal, we can write $D_{\mathrm{Sen},K_{\infty}[[t]]}(W)$ concretely:
\begin{theorem}[\cref{decompletion for de Rham crystals}]\label{restate decompletion for de Rham crystals}
    Suppose $\mathcal{M}\in \operatorname{Vect}((\mathcal{O}_{K})_{\Prism}, (\mathcal{O}_{\Prism}[\frac{1}{p}])_{\mathcal{I}}^{\wedge})$. Let $M=\mathcal{M}(\mathfrak{S},(E(u)))$, and $V(\mathcal{M})=\mathcal{M}(\Ainf,(\xi))$, then
    \[D_{\mathrm{Sen},K_{\infty}[[t]]}(V(\mathcal{M}))=M\otimes_{K[[t]]} K_{\infty}[[t]],\]
here we have identified $\bdr(\mathfrak{S})$ with $K[[T]]$.
\end{theorem}
 Now if $\vec{v}\in V(\mathcal{M})$ is $G_K$ invariant, then in particular it is a $\hat{G}$-locally analytic, $\Gal(L/K_{\infty})$ invariant element, hence  $\vec{v} \in D_{\mathrm{Sen},K_{\infty}[[t]]}(V(\mathcal{M}))$. Thanks to \cref{restate decompletion for de Rham crystals}, we can assume that $\vec{v}=\underline e\sum_{m=1}^{\infty} D_mt^m \in \mathcal{M}$ with $D_m\in M_{l\times 1}(K_\infty)$, from which we can calculate the $\hat{G}$-action concretely using \cref{cocycle g action}:
 $$0=(g-\Id)(\underline e\sum_{m=1}^{\infty} D_mt^m)=\underline e (\cdot (\sum_{m\geq0}(\sum_{n\geq 0}A_{m,n}X_{1}(g)^{[n]})t^m)(\sum_{m=1}^{\infty} g(D_m)(\alpha t)^m))-\underline e\sum_{m=1}^{\infty} D_mt^m.$$
 We want to show that this implies $\{D_m\}$ are actually matrices in $M_{l\times 1}(K)$ which are exactly determined by the same condition computing $H^0((\mathcal{O}_{K})_{\Prism},\mathcal{M})$ via \v{C}ech-Alexander complex, see \cref{prismatic invariant} for details. Then another difficulty is that $X_1$ is no longer a variable and we couldn't just compare the coefficient! To overcome it, essentially using \cite[Lemma 2.4.4]{DL21}, one can show that $\theta(X_1(\tau))$ is transcendental over $K$, this (plus some other conditions) implies the right hand side is always $0$ when viewing $X_1(g)$ as a "variable". As a result, we prove that 
\begin{theorem}[\cref{fully faithful of the restriction functor}]
    The restriction functor 
    \[V: \operatorname{Vect}((\mathcal{O}_{K})_{\Prism}, (\mathcal{O}_{\Prism}[\frac{1}{p}])_{\mathcal{I}}^{\wedge})\longrightarrow \operatorname{Vect}((\mathcal{O}_{K})_{\Prism}^{\perf}, (\mathcal{O}_{\Prism}[\frac{1}{p}])_{\mathcal{I}}^{\wedge})\]
is fully faithful.
\end{theorem}   
On the other hand, we want to understand the essential image of $V$, i.e. to determine which kind of de Rham representation of $G_K$ (here we implicitly identify $\operatorname{Vect}((\mathcal{O}_{K})_{\Prism}^{\perf}, (\mathcal{O}_{\Prism}[\frac{1}{p}])_{\mathcal{I}}^{\wedge})$ with $\rb$ via $T$ in \cref{main commutative diagram}) can be associated to a de Rham prismatic crystal. To state our conjecture in this direction, we need some preliminaries. 

For $W\in \rb$, $D_{\mathrm{Sen},K_{\infty}[[t]]}(W)$ is actually equipped with a natural monodromy operator $N_{\nabla}$:
\begin{theorem}[\cref{Kummer Sen operator}]\label{Restate Kummer Sen operator}
    Given $W\in \rb$, there is a $K_{\infty}$-linear operator
    \begin{equation*}
        N_{\nabla}: D_{\mathrm{Sen},K_{\infty}[[t]]}(W)\longrightarrow  D_{\mathrm{Sen},K_{\infty}[[t]]}(W)
    \end{equation*}
such that $N_{\nabla}$ satisfies Leibniz rule and that
\begin{equation*}
    N_{\nabla}(tv)=N_{\nabla}(t)v+tN_{\nabla}(v)=E^{\prime}(u)\lambda u \cdot v+tN_{\nabla}(v).
\end{equation*}
\end{theorem}

In the special case that $W=V(\mathcal{M})$, where $\mathcal{M}\in \operatorname{Vect}((\mathcal{O}_{K})_{\Prism}, (\mathcal{O}_{\Prism}[\frac{1}{p}])_{\mathcal{I}}^{\wedge})$ is a de Rham crystal corresponding to the pair $(M,\varepsilon)$ via \cref{algebraic description of crystal} for a finite free $\bdr(\mathfrak{S})$-module $M$ and a  stratification $\varepsilon$ such that  $$\varepsilon(\underline{e})=\underline e\cdot \sum_{m\geq0}(\sum_{n\geq 0}A_{m,n}X^{[n]})t^m,$$
then the monodromy operator $N_{\nabla}$ on $D_{\mathrm{Sen},K_{\infty}[[t]]}(W)$ reflects the stratification information. More precisely,
\begin{theorem}[\cref{operator and stratification}]
    On $D_{\mathrm{Sen},K_{\infty}[[t]]}(W)$, we have that 
    \[N_{\nabla}(\underline e)=-\underline e \lambda_1 u(\sum_{m=0}^{\infty}A_{m,1}t^m).\]
    Here $\lambda_1=\frac{\lambda}{E(u)}\in K[[t]]^{\times}$.
\end{theorem}
Based on this theorem, we see that the coefficients of $N_{\nabla}$ on $D_{\mathrm{Sen},K_{\infty}[[t]]}(V(\mathcal{M}))$ need to be in $K[[t]]$, not just in $K_{\infty}[[t]]$. Motivated by this observation, we make the following definition:
\begin{defi}\label{strong}
$W\in \Rep_{B_{\dR}^+}^{\fp,ndR}(G_K)$ is called \textit{strong nearly de Rham} if there exists a $K[[t]]$-module $M$ inside $D_{\mathrm{Sen},K_{\infty}[[t]]}(W)$ such that the following holds:
\begin{itemize}
    \item $M$ is stable under $N_{\nabla}$;
    \item $M\otimes_{K[[t]]}K_{\infty}[[t]]=D_{\mathrm{Sen},K_{\infty}[[t]]}(W)$.
\end{itemize}
We use $\Rep_{B_{\dR}^+}^{\fp,SndR}(G_K)$ to denote the full subcategory of $\Rep_{B_{\dR}^+}^{\fp,ndR}(G_K)$ consisting of such $W$.
\end{defi}
\begin{remark}
Such $M$ is unique if exists. Actually, for $W\in \Rep_{B_{\dR}^+}^{\fp,SndR}(G_K)$, if both $M_1$ and $M_2$ satisfies the desired property, then without loss of generality, we can assume $M_1 \subseteq M_2$, to show the inclusion $M_1\hookrightarrow M_2$ is an equality, it suffices to check that after the faithfully flat base change $K[[t]]\rightarrow K_{\infty}[[t]]$, which follows by definition.
\end{remark}
Then we have the following result:
\begin{prop}[\cref{weakly essential surjectivity}]
Consider the composition of the de Rham realization functor and the restriction functor, which we still denoted as $V$ by abuse of notation, then we get a fully faithful functor 
\begin{equation*}
     V: \operatorname{Vect}((\mathcal{O}_{K})_{\Prism}, (\mathcal{O}_{\Prism}[\frac{1}{p}])_{\mathcal{I}}^{\wedge})\longrightarrow \Rep_{B_{\dR}^+}^{\fp,SndR}(G_K).
\end{equation*}
Moreover, assume that \cref{restate e=1 conjecture} holds, then $ V$ is essentially surjective, in particular, it induces an equivalence of categories between $\operatorname{Vect}((\mathcal{O}_{K})_{\Prism}, (\mathcal{O}_{\Prism}[\frac{1}{p}])_{\mathcal{I}}^{\wedge})$ and $\Rep_{B_{\dR}^+}^{\fp,SndR}(G_K)$.
\end{prop}

\begin{remark}
It is natural to ask whether the artificial condition in the definition of $\Rep_{B_{\dR}^+}^{\fp,SndR}(G_K)$ can be removed. In other words, is the natural embedding $\Rep_{B_{\dR}^+}^{\fp,SndR}(G_K)\xhookrightarrow{} \Rep_{B_{\dR}^+}^{\fp,ndR}(G_K)$ essentially surjective? It is unclear to the author whether this should be true.

On the other hand, in the final stage of our project, we learned that Gao-Min-Wang have also initiated a similar project and in the forthcoming paper they can actually prove a similar version of \cref{restate e=1 conjecture}, in other words, they can give a sufficient condition to make $\varepsilon$ serve as a stratification of a de Rham crystal. Moreover, based on such sufficient condition, they could prove the stronger result that $V: \operatorname{Vect}((\mathcal{O}_{K})_{\Prism}, (\mathcal{O}_{\Prism}[\frac{1}{p}])_{\mathcal{I}}^{\wedge})\longrightarrow \Rep_{B_{\dR}^+}^{\fp,ndR}(G_K)$ is essentially surjective, which essentially implies that the natural embedding $\Rep_{B_{\dR}^+}^{\fp,SndR}(G_K)\xhookrightarrow{} \Rep_{B_{\dR}^+}^{\fp,ndR}(G_K)$ should be essentially surjective.
\end{remark}

\textbf{Outline}. The paper is organized as follows. 

In section 2, after introducing stratification of de Rham crystals over $(\mathcal{O}_K)_{\Prism}$, we give some necessary conditions on a sequence of matrices $\{A_{m,n}\}$ controlling such a stratification and we give some evidence of the conjecture that such conditions are also sufficient.

In section 3, we study absolute prismatic cohomology of de Rham crystals over $(\mathcal{O}_K)_{\Prism}$. In particular, we show such cohomology is concentrated on degree $0$ and $1$.

In section 4, we introduce locally analytic vectors and use this theory to study $\rb$. In particular, we prove a Sen style decompletion theorem for de Rham representations of $G_K$. As a consequence, we construct a decompletion functor $D_{\mathrm{Sen},K_{\infty}[[t]]}(\cdot)$ from $\rb$ to $\Mod_{K_{\infty}[[t]]}^{N_{\nabla}}$, which loses no information on detecting isomorphism classes. 

In section 5, we apply the decompletion theory developed in section 4 to study de Rham representations arising from de Rham prismatic crystals. We show that $D_{\mathrm{Sen},K_{\infty}[[t]]}(V(\mathcal{M}))$ can be written explicitly. As a consequence, we prove that the restriction functor $V$ in \cref{main commutative diagram} is fully faithful. Moreover, we show that we can extract stratification information of a de Rham crystal from the monodromy operator $N_{\nabla}$ after applying the decompletion functor $D_{\mathrm{Sen},K_{\infty}[[t]]}(\cdot)$ to its associated de Rham representations. As a consequence, we "narrow" the essential image of $V$ to what so called and argue that those strong nearly de Rham representations are precisely the image of $V$ provided \cref{restate e=1 conjecture} is true.
\vspace{0.2cm}

\textbf{Notations and conventions} In this paper $\mathcal{O}_K$ is a complete discrete valuation ring of mixed characteristic with fraction field $K$ and perfect residue field $k$ of characteristic $p>2$. Fix a uniformizer $\pi$ of $\mathcal{O}_K$. Fix an algebraic closure $\overline K$ of $K$, whose $p$-adic completion is denoted as $\mathbb{C}_p$. $\{\pi_n\}$ is a chosen sequence of $p$-power roots of $\pi$ in $\overline K$ such that $\pi_0=\pi$, $\pi_{n+1}^p=\pi_n$ while $\{\xi_n\}$ is a chosen sequence of $p$-power roots of unity such that $\xi_0=1$, $\xi_{n+1}^p=\xi_n$. Let $K_{\infty}=\cup_{n=1}^{\infty} K\left(\pi_{n}\right)$ and $K_{{p}^{\infty}}=\cup_{n=1}^{\infty} K\left(\xi_{n}\right)$. Finally we use $ L=\cup_{n=1}^{\infty} K\left(\pi_{n}, \xi_{n}\right)$ to denote the union of $K_{\infty}$ and $K_{{p}^{\infty}}$. $\hat{L}$ is the $p$-adic completion of $L$ with ring of integers $\mathcal{O}_{\hat{L}}$.

$\Ainf=W(\varprojlim_{x\mapsto x^p}\mathcal{O}_{\mathbb{C}_p}/p)$ with $\theta: \Ainf \longrightarrow \mathcal{O}_{\mathbb{C}_p}$. $B_{\dR}^+=(\Ainf[\frac{1}{p}])_{I}^{\wedge}$ where $I=\Ker \theta$. Similarly we define $A_L=\Ainf(\mathcal{O}_{\hat{L}})=W(\varprojlim_{x\mapsto x^p}\mathcal{O}_{\hat{L}}/p)$ and $B_{\dR,L}^+=(A_L[\frac{1}{p}])_{I}^{\wedge}$. Denote $\pi^{\flat}$ to be $(\pi_0,\pi_1,\cdots)\in \mathcal{O}_{\hat{L}}^{\flat}=\varprojlim_{x\mapsto x^p}\mathcal{O}_{\hat{L}}/p$ and $\varepsilon$ to be $(1,\xi_1,\xi_2,\cdots)\in \mathcal{O}_{\hat{L}}^{\flat}$. $\xi=\frac{[\varepsilon]-1}{[\varepsilon]^{1/p}-1}$ is a generator of $\Ker(\theta)$. $A_{\mathrm{cris}}$ is the usual crystalline ring, i.e. the $p$-completion of the divided power envelope of $\Ainf$ with respect to $\ker(\theta)$. $B_{\mathrm{cris}}^+=A_{\mathrm{cris}}[\frac{1}{p}]$.

In this paper we will embed the Breuil-Kisin prism $(\mathfrak{S},E)=(W(k)[[u]], E(u))$ into $(\Ainf,\Ker(\theta))$ by sending $u$ to $[\pi^{\flat}]$.
Let $t=E([\pi^{\flat}])$, where $E(u)$ is the Eisenstein polynomial of our chosen uniformizer $\pi$. $\beta=E^{\prime}(\pi)$ and $\lambda=\prod_{n\geq 0}\varphi^n(\frac{E(u)}{E(0)})\in B_{\mathrm{cris}}^+$.

Next we fix some Galois groups. $G_{K_{\infty}}=\operatorname{Gal}\left(\bar{K} / K_{\infty}\right)$ $G_{K_{p^{\infty}}}=\operatorname{Gal}\left(\bar{K} / K_{p^{\infty}}\right)$ $G_{L}=\operatorname{Gal}(\bar{K} / L)$ while $\Gamma_K=\operatorname{Gal}(K_{p^{\infty}}/K)$
and $\hat{G}=\Gal(L/K)\cong \Gal(L/K_{{p}^{\infty}})\rtimes \Gamma_K \cong \mathbb{Z}_{p} \tau \rtimes \Gamma_K$, here $\tau\in \Gal(L/K_{{p}^{\infty}})$ is a topological generator such that $\tau \pi_n=\xi_{p^n}\pi_n$. Moreover, fix $\gamma$ to be a topological generator of $\Gamma_K$. Then $\gamma \tau \gamma^{-1}=\tau^{\chi(\gamma)}$, here $\chi$ is the $p$-adic cyclotomic character. For $g\in \hat{G}$, $c(g)\in \mathbb{Z}_p$ is defined such that $g(\pi^{\flat})=\varepsilon^{c(g)}\pi^{\flat}$.

$K\{X_1\}^{\wedge}_{\rm pd}$ is the $p$-completion of the free pd polynomial with one variable $X_1$ over $K$. $X_1^{[n]}=\frac{X_1^n}{n!}$ is the $n$-th divided power of $X_1$. $A\{X\}_{\delta}$ means the free $\delta$-algebra over $A$ with one variable $X$. $A\langle T\rangle$ is the $p$-completion of the polynomial ring $A[T]$.

$(\mathcal{O}_{K})_{\Prism}$ is the absolute prismatic site of $\mathcal{O}_{K}$ while $(\mathcal{O}_{K})_{\Prism}^{\perf}$ is the site of perfect prisms over $\mathcal{O}_{K}$.
\vspace{0.2cm}

\textbf{Acknowledgement} The author is grateful to his advisor Kiran Kedlaya for patiently leading him to the study of $p$-adic Hodge theory and consistent encouragement as well as helpful discussions throughout the writing of this paper. Special thanks to Tong Liu for his interest in this project and valuable discussions in the preparation for the project. I thank Hui Gao for correspondence. During the preparation of the project, the author was partially supported by NSF1802161 under Professor Kedlaya. This work will be part of the author's Ph.D. thesis.

\section{de Rham prismatic crystals}
\subsection{Stratification of de Rham prismatic crystals}
Let us restate the definition of de Rham prismatic crystals first.
\begin{defi}
A de Rham prismatic crystal on $(\mathcal{O}_K)_{\Prism}$ is a sheaf $\mathcal{M}$ of $(\mathcal{O}_{\Prism}[\frac{1}{p}])_{\mathcal{I}}^{\wedge}$-modules such that for any $(A,I)\in (\mathcal{O}_K)_{\Prism}$, $\mathcal{M}(A,I)$ is a finite projective $(A[\frac{1}{p}])_{I}^{\wedge}$-module and that for any morphism $(A,I)\rightarrow (B,J)$ in $(\mathcal{O}_K)_{\Prism}$ (hence $J=IB$), the "crystal" property is satisfied, i.e. there is a canonical isomorphism
\[\mathcal{M}((A, I)) \otimes_{(A[\frac{1}{p}])_{I}^{\wedge}} (B[\frac{1}{p}])_{I}^{\wedge} \stackrel{\cong}{\rightarrow} \mathcal{M}((B, J)).\]
We denote by $\operatorname{Vect}((\mathcal{O}_{K})_{\Prism}, (\mathcal{O}_{\Prism}[\frac{1}{p}])_{\mathcal{I}}^{\wedge})$ the category of de Rham prismatic crystals.
\end{defi}
Later given $(A,I)\in (\mathcal{O}_{K})_{\Prism}$, we will use $\bdr(A)$ to denote $(A[\frac{1}{p}])_{I}^{\wedge}$ when $I$ is clear in the context.

Given $\mathcal{M}\in \operatorname{Vect}((\mathcal{O}_{K})_{\Prism}, (\mathcal{O}_{\Prism}[\frac{1}{p}])_{\mathcal{I}}^{\wedge})$, we can specialize it to get a rational Hodge-Tate crystal:
\begin{cor}\label{specializing Hodge Tate}
Define $\mathbb{M}:=\mathcal{M}/\mathcal{I}$, then $\mathbb{M}$ is a rational Hodge-Tate crystal, i.e. $\mathbb{M}\in \operatorname{Vect}((\mathcal{O}_{K})_{\Prism},\overline{\mathcal{O}}_{\Prism}[\frac{1}{p}])$, where the later is essentially defined in \cite[Definition 3.8]{MW21}. Moreover, for any $(A,I)\in (\mathcal{O}_K)_{\Prism}$, $\mathbb{M}(A,I)=\mathcal{M}((A,I))/I$.
\end{cor}
\begin{proof}
We claim that the presheaf associating each $(A,I)\in (\mathcal{O}_K)_{\Prism}$ to $\mathcal{M}((A,I))/I$ is a sheaf. Let $(A,I)\rightarrow (B,IB)$ be a cover in $(\mathcal{O}_K)_{\Prism}$ and $(C,IC)$ be the  self product of $(B,J)$ over $(A,I)$ in $(\mathcal{O}_K)_{\Prism}$ with two structure morphisms $p_1,p_2: (B,IB)\rightarrow (C,JC)$. As $\overline{\mathcal{O}}_{\Prism}[\frac{1}{p}] \in \operatorname{Vect}((\mathcal{O}_{K})_{\Prism},\overline{\mathcal{O}}_{\Prism}[\frac{1}{p}])$ and $\mathcal{M}((A,I))/I$ is a finite projective $(A/I)[\frac{1}{p}]$-module, we have a short exact sequence
\[0\rightarrow  \mathcal{M}(A,I)/I \longrightarrow (B/IB)[\frac{1}{p}]\otimes_{(A/I)[\frac{1}{p}]} (\mathcal{M}(A,I)/I)\stackrel{p_1\otimes 1-p_2\otimes 1}{\longrightarrow}(C/IC)[\frac{1}{p}]\otimes_{(A/I)[\frac{1}{p}]} (\mathcal{M}(A,I)/I).\]
By the crystal property of $\mathcal{M}$, this can be identified with 
\[0\rightarrow \mathcal{M}(A,I)/I\rightarrow \mathcal{M}(B,IB)/IB\rightarrow \mathcal{M}(C,IC)/IC.\]
Hence we are done.
\end{proof}
\begin{remark}\label{mod n analogue}
A similar argument shows that $\mathcal{M}/\mathcal{I}^n \in \operatorname{Vect}((\mathcal{O}_{K})_{\Prism},\mathcal{O}_{\Prism}/\mathcal{I}^n[\frac{1}{p}])$ and that for any $(A,I)\in (\mathcal{O}_K)_{\Prism}$, $\mathcal{M}/\mathcal{I}^n(A,I)=\mathcal{M}((A,I))/I^n$.
\end{remark}

\begin{cor}\label{algebraic description of crystal}
The category of de Rham prismatic crystals is equivalent to the category of finite free $\bdr(\mathfrak{S})$-modules $M$ equipped with a stratification satisfying cocycle condition, i.e a $\bdr(\mathfrak{S}^{1})$-linear isomorphism 
\[\epsilon: M\otimes_{\bdr(\mathfrak{S}),\delta_{0}^{1}} \bdr(\mathfrak{S}^{1}) \stackrel{\sim}{\rightarrow} M\otimes_{\bdr(\mathfrak{S}),\delta_{1}^{1}} \bdr(\mathfrak{S}^{1}),\]
where $\delta_{i}^{1}$ is defined in \cref{Equ-structure} such that the cocycle condition is satisfied:
\[\delta_{1}^{2, *}(\epsilon)=\delta_{2}^{2, *}(\epsilon) \circ \delta_{0}^{2, *}(\epsilon).\]
\end{cor}
\begin{proof}
Clearly given a de Rham prismatic crystal $\mathcal{M}$ we can get a finite free $\bdr(\mathfrak{S})$-module by evaluating $\mathcal{M}$ at the Breuil-Kisin prism $(\mathfrak{S},(E))$. Moreover, the desired stratification $\varepsilon$ follows by the crystal property of $\mathcal{M}$.

On the other hand, starting with such a pair $(M,\varepsilon)$, we can construct a de Rham crystal as follows. For any $(A,I)\in (\mathcal{O}_K)_{\Prism}$, as $(\mathfrak{S},(E))$ is a weakly final object in $(\mathcal{O}_K)_{\Prism}$, there exists $(B,J)\in (\mathcal{O}_K)_{\Prism}$ covering $(A,I)$ (i.e. $A\rightarrow B$ is $p$-completely faithfully flat) which also lies over $(\mathfrak{S},(E))$. For example, one can take $(B,J)$ to be $(A,I)\times_{(\mathcal{O}_K)_{\Prism}} (\mathfrak{S},(E))$, which satisfies the desired property thanks to \cite[Corollary 3.14]{BS19}. Consequently we can define 
\begin{equation*}
    \mathcal{M}(A,I)=\Eq(M\otimes_{\bdr(\mathfrak{S})} \bdr(B)\rightrightarrows M\otimes_{\bdr(\mathfrak{S})}\bdr(\tilde{B})),
\end{equation*}
where $(\tilde{B},\tilde{J}):=(B,J)\times_{(A,I)} (B,J)$ is the self product of $(B,J)$ over $(A,I)$. Here to give the two arrows in the diagram we implicitly use the universal property of $\mathfrak{S}^1$ to base change $\varepsilon$ along $\mathfrak{S}^1 \rightarrow \tilde{B}$. But by the proof of \cite[Proposition 2.7]{BS21}, de Rham crystals satisfy $(p,I)$-completely faithfully flat descent, hence $\mathcal{M}(A,I)$ is a finite projective $\bdr(A)$-module, we are done.
\end{proof}

Let $\mathfrak{S}^{n}$ be the $n+1$-th self product of the Breuil-Kisin prism $(\mathfrak{S},E)=(W(k)[[u]], E(u))$, then essentially using \cite[Corollary 3.14]{BS19}, we have that 
\begin{equation*}
    \mathfrak{S}^{n}=W(k)[[u_0,\cdots,u_{n}]]\{\frac{u_0-u_1}{E(u_0)},\cdots, \frac{u_0-u_n}{E(u_0)}\}_{\delta}^{{\wedge}_{(p,E(u_0))}}.
\end{equation*}
Later we will view $\mathfrak{S}^{n}$ as an $\mathfrak{S}$-algebra via $p_0: \mathfrak{S}\rightarrow \mathfrak{S}^{n}$ defined by sending $u$ to $u_0$.

First we would like to give a concrete description of $\bdr(\mathfrak{S}^{n})$.

\begin{lemma}\label{identification de rham ring}
we can identify $\bdr(\mathfrak{S}^{n})$ with $K\{X_1,\dots, X_n\}^{\wedge}_{\rm pd}[[t]]$ by identifying $t$ with the uniformizer $E(u_0)$ in the $E(u_0)$-adic complete ring $\bdr(\mathfrak{S}^n)$ and identifying $X_i$ with $\frac{u_0-u_i}{E(u_0)}$ in $\bdr(\mathfrak{S}^n)$.
\end{lemma}
\begin{proof}
Recall $\bdr(\mathfrak{S}^n)/(E)=\mathfrak{S}^n[1/p]/(E)$, hence we can identify the former with $K\{X_1,\dots, X_n\}^{\wedge}_{\rm pd}$ thanks to \cite[Lemma 2.7]{MW21} or \cite[Corollary 4.5]{Tian21}, where $X_i$ corresponds to the image of $\frac{u_0-u_i}{E(u_0)} \in \bdr(\mathfrak{S}^n)/E$. By \cite[\href{https://stacks.math.columbia.edu/tag/0ALJ}{Tag 0ALJ}]{SP22}, $(\bdr(\mathfrak{S}^n),(E))$ is a henselian pair , hence we can embed $K$ canonically into $\bdr(\mathfrak{S}^n)$. We then construct a ring homomorphism $f: K[X_1,\cdots,X_n]\rightarrow \bdr({\mathfrak{S}})$ by sending $X_i$ to $\frac{u_0-u_i}{E(u_0)} \in \bdr(\mathfrak{S}^n)$. To show that $f$ can actually be extended to a ring homomorphism  $K\{X_1,\dots, X_n\}^{\wedge}_{\rm pd}\rightarrow \bdr(\mathfrak{S}^n)$, by writing $\bdr(\mathfrak{S}^n)$ as $\varprojlim (\bdr(\mathfrak{S}^n)/E^j)$ it suffices to show that the induced ring homomorphism $f_j: K[X_1,\cdots,X_n] \rightarrow \bdr(\mathfrak{S}^n)/E^j$ can be extended to $\tilde{f}_j: K\{X_1,\dots, X_n\}^{\wedge}_{\rm pd}\rightarrow \bdr(\mathfrak{S}^n)/E^j$ for each $j$. In other words, we need to check that the image of $\frac{X_i^t}{t!}$ in $\bdr(\mathfrak{S}^n)/E^j$ is bounded for each $t$. First notice $f_1(\frac{X_i^t}{t!})$ actually lies in $\mathfrak{S}^n/E$ by the proof of \cite[Lemma 2.7]{MW21}, hence for a general $j$, $f_j(\frac{X_i^t}{t!})=z_{i,t}+Ew_{i,t}$ for some $z_{i,t}\in \mathfrak{S}^n$, $w_{i,t}\in \bdr(\mathfrak{S}^n)$, this implies that in $\bdr(\mathfrak{S}^n)/E^j$ we have that 
\begin{equation*}
\begin{split}
    f_j(\frac{X_i^t}{t!})^k=(z_{i,t}+Ew_{i,t})^k= z_{i,t}^k+\binom{k}{1} z_{i,t}^{k-1}\cdot Ew_{i,t}+\cdots+ \binom{k}{j-1} z_{i,t}^{k-(j-1)}\cdot (Ew_{i,t})^{j-1}
\end{split}
\end{equation*}
for arbitrary $k\in \mathbb{N}$, hence $f_j(\frac{X_i^t}{t!})^k \in (\mathfrak{S}^n+w_{i,t}\mathfrak{S}^n+\cdots+w_{i,t}^{j-1}\mathfrak{S}^n)$ is bounded in $\bdr(\mathfrak{S}^n)/E^j$ for any $k\geq 0$.

Now we have extended $f$ to a ring homomorphism $K\{X_1,\dots, X_n\}^{\wedge}_{\rm pd}\rightarrow \bdr(\mathfrak{S}^n)$. Then after choosing the uniformizer $E(u_0)$ in $\bdr(\mathfrak{S}^n)$, we get a ring homomorphism $K\{X_1,\dots, X_n\}^{\wedge}_{\rm pd}[[t]]\rightarrow \bdr(\mathfrak{S}^n)$ by sending $t$ to $E(u_0)$, which is actually an isomorphism by using derived Nakayama and checking that after modulo $t$. 
\end{proof}
\begin{remark}
This construction is compatible with the identification $\bdr(\mathfrak{S}^n)/(E)\cong K\{X_1,\dots, X_n\}^{\wedge}_{\rm pd}$ after modulo $E$.
\end{remark}

\begin{example}\label{image of t}
Let us calculate the image of $\frac{E(u_1)}{E(u_0)}\in \bdr(\mathfrak{S}^{n})$ under this identification. Recall that
\[E(u_1)=E(u_0)+\sum_{n=1}^{e}\frac{E^{(n)}(u_0)}{n!}(u_1-u_0)^n,\]
hence
\begin{equation}\label{first quotient}
\begin{split}
    \frac{E(u_1)}{E(u_0)}&=1+\sum_{n=1}^{e}\frac{E^{(n)}(u_0)}{n!}(\frac{u_1-u_0}{E(u_0)})^n(E(u_0))^{n-1}\\&=1+\sum_{n=1}^{e}\frac{E^{(n)}(u_0)}{n!}(-1)^n(\frac{u_0-u_1}{E(u_0)})^n(E(u_0))^{n-1}.
\end{split}
\end{equation}
Let $E^{(n)}(u_0)=\sum_{i=0}^{\infty}\theta_{n,i}t^i\in \bdr{\mathfrak{S}}\cong K[[t]]$, here $\theta_{n,i}\in K$. Then we have that $\theta_{n,0}=E^{(n)}(\pi)$. We will use $\beta$ to denote $\theta_{1,0}=E^{\prime}(\pi)$

Then \cref{first quotient} turns into
\begin{equation*}
    \begin{split}
         \frac{E(u_1)}{E(u_0)}&=1+\sum_{n=1}^{e}(\sum_{i=0}^{\infty}\theta_{n,i}t^i)\frac{(-1)^n X_1^n t^{n-1}}{n!}
         \\&=1+\sum_{i=0}^{\infty} t^i(\sum_{n=1}^e \theta_{n,i-(n-1)}\frac{(-1)^nX_1^n}{n!})
         \\&=1-\beta X_1+\sum_{i=1}^{\infty} t^i(\sum_{n=1}^e \theta_{n,i-(n-1)}\frac{(-1)^nX_1^n}{n!}).
    \end{split}
\end{equation*}
We will denote $\alpha$ to be $\frac{E(u_1)}{E(u_0)}$ for later use.

\end{example}

\begin{remark}\label{nonvanishing}
By definition, $\theta_{n,0}=E^{(n)}(\pi)$ is non zero as $E(x)$ is the minimal polynomial of $\pi$ over $W(k)$ and the degree of $E^{(n)}(x)$ is strictly smaller than that of $E(x)$.
\end{remark}

\begin{lemma}\label{Equ-structure}
   For any $0\leq i\leq n+1$, let $\delta_{i}^{n+1}:\bdr(\mathfrak{S}^n)\to\bdr(\mathfrak{S}^{n+1})$ be the structure morphism induced by the order-preserving map \[\{0,\dots,n\}\to\{0,\dots,i-1,i+1\dots,n+1\}.\]
  Then by identifying $\bdr(\mathfrak{S}^{n})$ with $K\{X_1,\dots, X_n\}^{\wedge}_{\rm pd}[[t]]$ under \cref{identification de rham ring}, we have that
  \begin{equation}\label{Equ-structure morphism on variables}
      \delta_{i}^{n+1}(X_j) = \left\{
      \begin{array}{rcl}
           (X_{j+1}-X_1)\alpha^{-1}, & i=0  \\
           X_j, & j<i \\
           X_{j+1}, & 0<i\leq j
      \end{array}
      \right.
  \end{equation}
  \begin{equation}\label{Equ-structure morphism on uniformizer}
      \delta_{i}^{n+1}(t) = \left\{
      \begin{array}{rcl}
           \alpha t, & i=0  \\
           t, & i>0
      \end{array}
      \right.
  \end{equation}
\end{lemma}
\begin{proof}

For the image of $t$ under $\delta_{i}^{n+1}$, the result follows from the definition of $\delta_{i}^{n+1}$ and  \cref{image of t}. Also,
\begin{equation*}
    \begin{split}
        \delta_{0}^{n+1}(X_j)=\delta_{0}^{n+1}(\frac{u_0-u_j}{E(u_0)})=\frac{u_1-u_{j+1}}{E(u_1)}=\frac{(u_0-u_{j+1})-(u_0-u_1)}{E(u_0)}\cdot \frac{E(u_0)}{E(u_1)}=(X_{j+1}-X_1)\alpha^{-1}.
    \end{split}
\end{equation*}
For $0<i\leq j$, $\delta_{i}^{n+1}$ preserves $u_0$ and sends $u_j$ to $u_{j+1}$, hence $\delta_{i}^{n+1}(X_{j})=X_{j+1}$. For $i>j$, $\delta_{i}^{n+1}$ preserves both $u_0$ and $u_j$, hence it also preserves $X_j$.
\end{proof}

Let $M$ be a finite free $\bdr(\mathfrak{S})$-module endowed with a stratification $\epsilon$ with respect to the cover $\mathfrak{S}^{\bullet}$. Fix a $\bdr(\mathfrak{S})$ basis $e_1,\ldots,e_l$ of $M$ and we write
\begin{equation}\label{stratification equation}
    \epsilon(\underline{e})=\underline e\cdot \sum_{m\geq0}(\sum_{n\geq 0}A_{m,n}X_1^{[n]})t^m,
\end{equation}
where $A_{m,n}\in M_l(K)$, $A_{0,0} = I,\lim_{n\rightarrow \infty}A_{m,n}=0$ for any $m$. Here we have identified $\bdr(\mathfrak{S})$ (resp. $\bdr(\mathfrak{S}^1))$ with $K[[T]]$ (resp. $K\{X_1\}^{\wedge}_{\rm pd}[[T]]$). 

To calculate the cocycle condition, notice that 
\begin{equation}\label{stratification 0}
    \delta_{1}^{2, *}(\varepsilon)(\underline e) =\underline e\sum_{m\geq0}(\sum_{n\geq 0}A_{m,n}X_2^{[n]})t^m,
\end{equation}

and by \cref{Equ-structure}, we have that
\begin{equation}\label{stratification calculation}
    \begin{split}
        \delta_{2}^{2, *}(\epsilon) \circ \delta_{0}^{2, *}(\epsilon)(\underline e) & =  \delta_{2}^{2, *}(\varepsilon)(\underline e\sum_{p,q\geq 0} A_{p,q}(X_2-X_1)^{[q]}\alpha^{-q}\alpha^pt^p)\\
        & = \underline e(\sum_{m\geq0}(\sum_{n\geq 0}A_{m,n}X_{1}^{[n]})t^m)(\sum_{p,q\geq 0} A_{p,q}(X_2-X_1)^{[q]}\alpha^{p-q}t^p).
     \end{split}
\end{equation}
For $p\in \mathbb{Z}$, write
\begin{equation*}
    \begin{split}
        \alpha^p=\sum_{s=0}^{\infty}c_{p,s}t^s,
    \end{split}
\end{equation*}
 where $c_{p,s}=c_{p,s}(X_1)\in K\{X_1\}^{\wedge}_{\rm pd}$. In paricular, $c_{p,0}=(1-\beta X_{1})^p$.

Then \cref{stratification calculation} turns into 
\begin{equation}\label{stratification calculation II}
    \begin{split}
        \delta_{2}^{2, *}(\epsilon) \circ \delta_{0}^{2, *}(\epsilon)(\underline e) 
        & = \underline e(\sum_{m\geq0}(\sum_{n\geq 0}A_{m,n}X_{1}^{[n]})t^m)(\sum_{p,q\geq 0} A_{p,q}(X_2-X_1)^{[q]}t^p\sum_{s=0}^{\infty}c_{p-q,s}t^s)
        \\&= \underline e(\sum_{m\geq0}(\sum_{n\geq 0}A_{m,n}X_{1}^{[n]})t^m)(\sum_{i\geq 0}t^i\sum_{p=0}^{i}\sum_{q\geq 0} A_{p,q}(X_2-X_1)^{[q]}c_{p-q,i-p})
        \\&=\underline e \sum_{m\geq 0}t^m\sum_{\substack{i+j=m}}(\sum_{n\geq 0}A_{j,n}X_{1}^{[n]})(\sum_{p=0}^{i}\sum_{q\geq 0} A_{p,q}(X_2-X_1)^{[q]}c_{p-q,i-p}).
     \end{split}
\end{equation}
Then by comparing the coefficient of $t^m$ in \cref{stratification 0} and \cref{stratification calculation II}, we see that the cocycle condition is equivalent to the following:
\begin{equation*}
    \begin{split}
        \sum_{n\geq 0}A_{m,n}X_2^{[n]}&=\sum_{\substack{i+j=m}}(\sum_{s\geq 0}A_{j,s}X_{1}^{[s]})(\sum_{p=0}^{i}\sum_{q\geq 0} A_{p,q}(X_2-X_1)^{[q]}c_{p-q,i-p})
        \\&=\sum_{\substack{i+j=m}}(\sum_{s\geq 0}A_{j,s}X_{1}^{[s]})(\sum_{p=0}^{i}\sum_{k=0}^{\infty}X_2^{[k]}\sum_{v=0}^{\infty}A_{p,k+v}(-1)^vX_1^{[v]}c_{p-(k+v),i-p})
        \\&=\sum_{k=0}^{\infty}X_2^{[k]}\sum_{\substack{i+j=m}}(\sum_{s\geq 0}A_{j,s}X_{1}^{[s]})(\sum_{p=0}^{i}\sum_{v=0}^{\infty}A_{p,k+v}(-1)^vX_1^{[v]}c_{p-(k+v),i-p}).
    \end{split}
\end{equation*}
Further compare the coefficient of $X_2^{[n]}$, the cocycle relation can be reinterpreted as
\begin{equation}\label{cocycle equivalent II}
    A_{m,k}=\sum_{\substack{i+j=m}}(\sum_{s\geq 0}A_{j,s}X_{1}^{[s]})(\sum_{p=0}^{i}\sum_{v=0}^{\infty}A_{p,k+v}(-1)^vX_1^{[v]}c_{p-(k+v),i-p}).
\end{equation}

Now we can state our main result in this section.

\begin{theorem}\label{Main theorem I}
Keep notations as above. Then
\begin{itemize}
    \item If $(M,\varepsilon)$ is induced from a de Rham prismatic crystal $\mathcal{M} \in \operatorname{Vect}((\mathcal{O}_{K})_{\Prism}, (\mathcal{O}_{\Prism}[\frac{1}{p}])_{\mathcal{I}}^{\wedge})$, then the following holds:
    \begin{itemize}
        \item $A_{0,0}=I$ and $A_{i,0}=0$ for $i>0$.
        \item $A_{0,1} \in M_l(K)$ satisfies that $\lim_{n\to+\infty}\prod_{i=0}^n(iE'(\pi)+A_{0,1}) = 0$.
        \item $A_{m,n+1}=(\beta(n-m)+A_{0,1})A_{m,n}+\sum_{\substack{i+j=m\\i\leq m-1}}(A_{j,1}+(n-i)\theta_{1,j})A_{i,n}$ for $m,n\in \mathbb{N}^{\geq 0}$. In particular, $\{A_{m,n}\}$ is determined by $\{A_{m,1}\}_{m\geq 0}$.
    \end{itemize}
\end{itemize}
\end{theorem}
\begin{proof}
Recall by \cref{cocycle equivalent II}, the cocycle condition is equivalent to the following:
\begin{equation}\label{local ref}
    A_{m,k}=\sum_{\substack{i+j=m}}(\sum_{s\geq 0}A_{j,s}X_{1}^{[s]})(\sum_{p=0}^{i}\sum_{v=0}^{\infty}A_{p,k+v}(-1)^vX_1^{[v]}c_{p-(k+v),i-p}).
\end{equation}
The idea is to compare the coefficients of $X_{1}^{[0]}$ and $X_{1}^{[1]}$ on both sides. Write
\begin{equation*}
    c_{p,s}=\sum_{i=0}d_{p,s,i}X_{1}^{[i]},
\end{equation*}
where $d_{p,s,i}\in K$. An easy observation is that $d_{m,0,,0}=1$,  $d_{p,i,0}=0$ for $m\in \mathbb{Z}$ and $i>0$.

First consider the constant term in \cref{local ref}, then we get
\begin{equation}\label{calculating 0 column}
    \begin{split}
        A_{m,k}=\sum_{i+j=m}A_{j,0}(\sum_{p=0}^i A_{p,k}d_{p-k,i-p,0}).
    \end{split}
\end{equation}
Take $m=1, k=0$, this implies that $A_{1,0}=0$ as $d_{0,1,0}=0$ and $A_{0,0}=I$ by assumption.

To show $A_{m,0}=0$ for general $m>0$, we proceed by induction on $m$. Suppose $A_{k,0}=0$ for all positive integers no larger than $m-1$ ($m\geq 2$). Than \cref{calculating 0 column} turns into 
\begin{equation*}
    A_{m,0}=d_{m,0,0}A_{m,0}+A_{m,0},
\end{equation*}
from which we see $A_{m,0}$ also vanishes. Hence we finish the proof of the first part in \cref{Main theorem I}.

Next we compare the coefficients of $X_{1}$ on left and right side of \cref{local ref}:
\begin{equation*}
\begin{split}
    0&=A_{0,0}(\sum_{p=0}^{m}(A_{p,k+1}\cdot(-1)\cdot d_{p-k-1,m-p,0}+A_{p,k}d_{p-k,m-p,1}))+\sum_{i+j=m}(A_{j,1}(\sum_{p=0}^{i}d_{p-k,i-p,0}A_{p,k})
    \\&=-A_{m,k+1}+(k-m)\beta A_{m,k}+\sum_{p=0}^{m-1}(k-p)\theta_{1,m-p}A_{p,k}+\sum_{i+j=m}A_{j,1}A_{i,k}.
\end{split}
\end{equation*}
Here the first identity follows from the vanishing of $A_{r,0}$ for $r>0$, the second equality is due to the vanishing of $d_{r,i,0}$ for $i>0$ and \cref{calculate c}.

As a result,
\begin{equation}\label{inductive formula}
    A_{m,k+1}=(\beta(k-m)+A_{0,1})A_{m,k}+\sum_{\substack{i+j=m\\i\leq m-1}}(A_{j,1}+(k-i)\theta_{1,j})A_{i,k}.
\end{equation}

Take $m=0$, this implies $$A_{0,n+1}=(\beta n+A_{0,1})A_{0,n}.$$ Since $\lim_{n\rightarrow \infty}A_{0,n}=0$, $A_{0,1} \in M_l(K)$ satisfies that $\lim_{n\to+\infty}\prod_{i=0}^n(i\beta+A_{0,1}) = 0$. We finishe the proof of \cref{Main theorem I}.
\end{proof}

The following lemma is used in the proof:
\begin{lemma}\label{calculate c}
$d_{p,s,1}=-p\theta_{1,s}$ for $p\in \mathbb{Z}$, $s\geq 0$.
\end{lemma}
\begin{proof}
When $s=0$, by definition, $c_{p,0}=(1-\beta X_1)^p$, hence $d_{p,0,1}=-p\beta=-p\theta_{1,1}$. For $s\geq 1$, we first assume $p\geq 1$, then by definition, $d_{p,s,1}$ is the coefficient of $X_{1}$ in 
\[c_{p,s}=\sum_{\substack{(g_1,\cdots,g_p)\in \mathbb{N}^p\\0\leq g_i\leq s\\ g_1+\cdots+g_p=s}}\prod_{i=1}^{p}a_{g_i},\]
here $a_i=\sum_{n=1}^{e}\theta_{n,i-(n-1)}\frac{(-1)^nX_1^n}{n!}$. 

Then observe that for $g_1,g_2\geq 1$, $X_{1}^2$ is a factor of $a_{g_1}a_{g_2}$, hence such $\prod_{i=1}^{p}a_{g_i}$ does not contribute to the coefficient of $X_{1}$ in $c_{p,s}$, then the desired result follows as the coefficient of $X_{1}$ in $a_s$ is precisely $\theta_{1,s}$.
For $p\leq 1$, a similar argument leads to the result.
\end{proof}

\begin{remark}\label{remark with MW relation}
\begin{itemize}
    \item Given a de Rham prismatic crystal $\mathcal{M} \in \operatorname{Vect}((\mathcal{O}_{K})_{\Prism}, (\mathcal{O}_{\Prism}[\frac{1}{p}])_{\mathcal{I}}^{\wedge})$, then $\mathcal{M}/\mathcal{I}$ defines a Hodge Tate crystal in the sense of \cite[Definition 3.1]{MW21} by \cref{specializing Hodge Tate}, then one can see our work is compatible with Min and Wang's results in the sense that if we let $m$ be $0$, then the equation $A_{m,n+1}=(\beta(n-m)+A_{0,1})A_{m,n}+\sum_{\substack{i+j=m\\i\leq m-1}}(A_{j,1}+(n-i)\theta_{1,j})A_{i,n}$ turns into $A_{0,n+1}=A_{0,n}(\beta n+A_{0,1})$, which is precisely the condition in \cite[Theorem 3.5]{MW21}.
    \item Thanks to \cref{Main theorem I}, by induction on $m$ and applying the inductive formula $A_{m,n+1}=(\beta(n-m)+A_{0,1})A_{m,n}+\sum_{\substack{i+j=m\\i\leq m-1}}(A_{j,1}+(n-i)\theta_{1,j})A_{i,n}$ guarantees that $\lim_{n\rightarrow \infty}A_{m,n}=0$ is satisfied automatically thanks to the convergence property that $A_{0,1}$ satisfies. In other words, there is no new restriction on $A_{1,1},A_{2,1},\cdots$ to make $\lim_{n\rightarrow \infty}A_{m,n}=0$ holds for $m\geq 1$.
    \item One might wonder whether the three properties in the theorem are sufficient to define the stratification of a prismatic de Rham crystal. Inspired by the proof of \cite[Lemma 3.6]{MW21}, we believe that the answer is positive. We have the following conjecture:
\end{itemize}

\end{remark}

\begin{conjecture}\label{conjecture e=1}
Assume that $\{A_{m,n}\}$ is a sequences of matrices in $ M_l(K)$ such that the following holds:
\begin{itemize}
        \item $A_{0,0}=I$ and $A_{i,0}=0$ for $i>0$.
        \item $A_{0,1} \in M_l(K)$ satisfies that $\lim_{n\to+\infty}\prod_{i=0}^n(iE'(\pi)+A_{0,1}) = 0$.
        \item $A_{m,n+1}=(\beta(n-m)+A_{0,1})A_{m,n}+\sum_{\substack{i+j=m\\i\leq m-1}}(A_{j,1}+(n-i)\theta_{1,j})A_{i,n}$ for $m,n\in \mathbb{N}^{\geq 0}$. In particular, $\{A_{m,n}\}$ is determined by $\{A_{m,1}\}_{m\geq 0}$.
    \end{itemize}
Then 
\begin{equation*}
    \epsilon(\underline{e})=\underline e\cdot \sum_{m\geq0}(\sum_{n\geq 0}A_{m,n}X^{[n]})t^m
\end{equation*}
can serve as the stratification associated to a certain prismatic de Rham crystal.
\end{conjecture}

Nevertheless, although we couldn't prove \cref{conjecture e=1} now, we will give an explicit expression such that in the case that $A_{0,1}$ commutes with all $A_{j,1}$, \cref{conjecture e=1} can be reduced to the study of a sequence of very explicit polynomials in variables $\{A_{0,1}, A_{1,1}\}$ showing up in the following theorem:
\begin{theorem}\label{main theorem II}
Assume that $\{A_{m,n}\}$ is a sequence of matrices in $ M_l(K)$ such that the following holds (for example, thanks to \cref{Main theorem I}, this is satisfied when coming from the stratification of a prismatic de Rham crystal such that $A_{0,1}$ commutes with all $A_{j,1}$ ):
    \begin{itemize}
        \item $A_{0,0}=I$ and $A_{i,0}=0$ for $i>0$.
        \item $A_{0,1} \in M_l(K)$ satisfies that $\lim_{n\to+\infty}\prod_{i=0}^n(iE'(\pi)+A_{0,1}) = 0$.
        \item $A_{m,n+1}=(\beta(n-m)+A_{0,1})A_{m,n}+\sum_{\substack{i+j=m\\i\leq m-1}}(A_{j,1}+(n-i)\theta_{1,j})A_{i,n}$ for $m,n\in \mathbb{N}^{\geq 0}$. In particular, $\{A_{m,n}\}$ is determined by $\{A_{m,1}\}_{m\geq 0}$.
        \item $A_{0,1}$ commutes with all $A_{j,1}$.
    \end{itemize}
    then there exists a sequence of polynomials $\{h_{m,j}\}$ (concentrated on $0\leq j\leq 2m$) such that 
\begin{align*}
    \sum_{n=0}^{\infty}A_{m,n}X^{[n]}=\sum_{j=0}^{m}h_{m,j}(1-\beta X)^{m-j}X^j(1-\beta X)^{-\frac{A_{0,1}}{\beta}}+\sum_{j=m+1}^{2m}h_{m,j}A_{0,j-m}(1-\beta X)^{m-j}X^j(1-\beta X)^{-\frac{A_{0,1}}{\beta}}
\end{align*}
Moreover, $\{h_{m,j}\}$ satisfies the following properties:
\begin{itemize}
    \item $h_{0,0}=1$ and for $m\geq 1$ $h_{m,j}$ is a polynomial of degree no more than $m$ in variables $A_{0,1},\cdots,A_{m,1}$ with coefficients in $K$, where $A_{i,1}$ is given weighted degree $i$, Moreover, $h_{m,0}=0$ for $m\geq 1$.
    \item $\{h_{m,j}\}$ can be calculated inductively via 
    \begin{itemize}
        \item for $1\leq j\leq m$,
        \begin{align*}
        h_{m,j}=&f_{m,j}A_{m,1}+\sum_{\substack{f+r=m\\ 1\leq f,r\leq m-1}}(A_{r,1}-f\theta_{1,r})(\sum_{i=1}^{f}h_{f,i}g_{m,f,i}^j+\sum_{i=f+1}^{2f}h_{f,i}g_{m,f,i}^jA_{0,i-f})
         \\&+\sum_{\substack{f+r=m\\ f\leq m-1}}\theta_{1,r}(\sum_{i=0}^{f}((j-1)g_{m,f,i}^j+(1-\frac{1}{j})g_{m,f,i}^{j-1}(-(m-(j-1))\beta+A_{0,1}))h_{f,i}
        \\&+\sum_{i=f+1}^{2f}(h_{f,i}A_{0,i-f}((j-1)g_{m,f,i}^j+(1-\frac{1}{j})g_{m,f,i}^{j-1}(-(m-(j-1))\beta+A_{0,1})))).
       \end{align*}
    \item for $m+1\leq j\leq 2m$,
    
    \begin{align*}
        h_{m,j}&=\sum_{\substack{f+r=m\\ 1\leq f,r\leq m-1}}(A_{r,1}-f\theta_{1,r})(\sum_{i=f+1}^{2f}h_{f,i}g_{m,f,i}^j\frac{A_{0,i-f}}{A_{0,j-m}})
        \\&+\sum_{\substack{f+r=m\\ f\leq m-1}}\theta_{1,r}(\sum_{i=f}^{2f}h_{f,i}(\frac{A_{0,i-f}}{A_{0,j-m}}((j-1)g_{m,f,i}^j)+(1-\frac{1}{j})g_{m,f,i}^{j-1}\frac{A_{0,i-f}}{A_{0,j-1-m}}).
        \end{align*}
    \end{itemize}
    
    Here $f_{m,1}=1$, $f_{m,i}=\frac{\beta^{i-1}}{i!}(m-1)\cdots (m-(i-1))$ for $2\leq i\leq m$ and
    \begin{equation*}\label{calculating g}
      g_{m,f,i}^j= \left\{
      \begin{array}{rcl}
           \frac{\beta^{j-i-1}}{j}\frac{1}{(j-i-1)!}(m-(f+1))\cdots(m-(f+j-i-1)), & m\geq f+1 & i+1\leq j\leq m-f+i  \\
           0, & \text{otherwise}
      \end{array}
      \right.
  \end{equation*}
\end{itemize}
\end{theorem}
\begin{remark}\label{variant main theorem II}
If we define 
\begin{equation*}\label{variant h}
      \tilde{h}_{m,j}= \left\{
      \begin{array}{rcl}
           h_{m,j}, & j\leq m  \\
           h_{m,j}A_{0,j-m}, & j+1\leq m \leq m+1
      \end{array}
      \right.
  \end{equation*}
Then we have a much cleaner inductive formula for calculating $ \tilde{h}_{m,j}$ and \cref{main theorem II} can be restated as
\[\sum_{n=0}^{\infty}A_{m,n}X^{[n]}=\sum_{j=1}^{2m}\tilde{h}_{m,j}(1-\beta X)^{m-j}X^j(1-\beta X)^{-\frac{A_{0,1}}{\beta}}.\]
Later we will freely interchange between $\tilde{h}$ and $h$.
\end{remark}

In the unramified case, we have a much cleaner expression:
\begin{cor}\label{e=1 known}
Assume as in \cref{main theorem II} and moreover $e=1$, then 
\begin{align*}
    \sum_{n=0}^{\infty}A_{m,n}X^{[n]}=\sum_{j=1}^{m}h_{m,j}(1-\beta X)^{m-j}X^j(1-\beta X)^{-\frac{A_{0,1}}{\beta}}
\end{align*}
such that:
\begin{itemize}
    \item $h_{0,0}=1$ and for $m\geq 1$ $h_{m,j}$ is a homogeneous polynomial of degree $m$ in variables $A_{1,1},\cdots,A_{m,1}$, where $A_{i,1}$ is given weighted degree $i$.
    \item $\{h_{m,j}\}$ can be calculated inductively via 
    \begin{align*}
        h_{m,j}=f_{m,j}A_{m,1}+\sum_{\substack{f+r=m\\ 1\leq f,r\leq m-1}}A_{r,1}\sum_{i}h_{f,i}g_{m,f,i}^j,
    \end{align*}
    where $f_{m,1}=1$ and $g_{m,f,i}^j$ are defined in \cref{main theorem II}.
\end{itemize}
\end{cor}
\begin{proof}
In this case $\theta_{1,r}=0$ for $r>0$, then the result follows from induction on $m$ using the inductive formula for $h
_{m,j}$.
\end{proof}
\begin{example}\label{m small}
Under the assumption of \cref{main theorem II},
\begin{equation*}
    \begin{split}
        \sum_{n=0}^{\infty}A_{1,n}X^{[n]}&=(1-\beta X)^{-\frac{A_{0,1}}{\beta}}(A_{1,1}X+\frac{\theta_{1,1}}{2}A_{0,1}(1-\beta X)^{-1}X^2).
        \\\sum_{n=0}^{\infty}A_{2,n}X^{[n]}&=(1-\beta X)^{-\frac{A_{0,1}}{\beta}}(A_{2,1}X(1-\beta X)+(\frac{\beta}{2}A_{2,1}+\frac{1}{2}A_{1,1}^2+\frac{\theta_{1,2}A_{0,1}}{2})X^2+
        \\&(\frac{\theta_{1,1}A_{1,1}}{2}+\frac{\theta_{1,1}^2}{6}+\frac{\beta\theta_{1,2}}{3})A_{0,1}(1-\beta X)^{-1}X^3+\frac{\theta_{1,1}^2}{8}A_{0,2}(1-\beta X)^{-2}X^4).
    \end{split}
\end{equation*}
\end{example}

\subsection{Calculation in the commutative case}
Our goal is to prove \cref{main theorem II}. 
In this subsection we assume that $\{A_{m,n}\}$ is a sequence of matrices in $ M_l(K)$ satisfies assumptions in \cref{main theorem II}. In particular, $A_{0,1}$ commutes with all $A_{j,1}$.

Our strategy is to prove the following lemma:
\begin{lemma}\label{preliminary lemma}
There exists a sequence of polynomials $\{h_{m,j}\}$ satisfying desired properties in \cref{main theorem II} such that \[A_{m,s}=A_{0,s-m}(\sum_{j=1}^{m}(-(m-j)\beta+ A_{0,1})\cdots (-\beta+A_{0,1})s\cdots (s-j+1)h_{m,j}+\sum_{j=m+1}^{2m}s\cdots(s-j+1)h_{m,j})\]
for non negative integers $m,s$. Here we abuse the notation by requiring that $$A_{0,-j}(-j\beta+A_{0,1})\cdots (-\beta+A_{0,1})=I$$ for $1\leq j \leq m$.
\end{lemma}

\begin{proof}[Proof of \cref{main theorem II} under \cref{preliminary lemma}]
We separate the calculation of $\sum_{s=0}^{\infty}A_{m,s}X^{[s]}$ into two parts. For the sum over $1\leq j\leq m$,
\begin{equation}\label{smaller than m}
    \begin{split}
        &\sum_{j=1}^{m}\sum_{s=0}^{\infty}A_{0,s-m}(-(m-j)\beta+ A_{0,1})\cdots (-\beta+A_{0,1})s(s-1)\cdots (s-j+1)X^{[s]}h_{m,j}
        \\=&\sum_{j=1}^{m}\sum_{s=j}^{\infty}A_{0,s-m}(-(m-j)\beta+ A_{0,1})\cdots (-\beta+A_{0,1})s(s-1)\cdots (s-j+1)X^{[s]}h_{m,j}
        \\=&\sum_{j=1}^{m}X^j\sum_{s=0}^{\infty}A_{0,s-(m-j)}(-(m-j)\beta+ A_{0,1})\cdots (-\beta+A_{0,1})X^{[s]}h_{m,j}
        \\=&\sum_{j=1}^{m}X^j(1-\beta X)^{-\frac{-(m-j)\beta+A_{0,1}}{\beta}}h_{m,j}
        \\=&\sum_{j=1}^{m}X^j(1-\beta X)^{m-j}(1-\beta X)^{\frac{-A_{0,1}}{\beta}}h_{m,j},
    \end{split}
\end{equation}
here the second to last equality follows from applying \cref{exponential sum}, which is a key observation from \cite{MW21} to $A=-(m-j)\beta+A_{0,1}$ thanks to \cref{remark with MW relation}.

For $m+1\leq j \leq 2m$ part,
\begin{equation}
    \begin{split}\label{larger than m}
        \sum_{j=m+1}^{2m}\sum_{s=0}^{\infty} h_{m,j}A_{0,s-m}s\cdots(s-j+1)X^{[s]}&=\sum_{j=m+1}^{2m}X^j h_{m,j}\sum_{s=0}^{\infty}A_{0,s+j-m}X^{[s]}
        \\&=\sum_{j=m+1}^{2m}X^j h_{m,j}A_{0,j-m}(1-\beta x)^{\frac{-A_{0,1}}{\beta}-(j-m)},
    \end{split}
\end{equation}
here the last equality also follows from \cref{exponential sum}.

Combine \cref{smaller than m}, \cref{larger than m}, than thanks to \cref{preliminary lemma},
we have that
\begin{equation*}
    \sum_{n=0}^{\infty}A_{m,n}X^{[n]}=\sum_{j=1}^{m}h_{m,j}(1-\beta X)^{m-j}X^j(1-\beta X)^{-\frac{A_{0,1}}{\beta}}+\sum_{j=m+1}^{2m}h_{m,j}A_{0,j-m}(1-\beta X)^{m-j}X^j(1-\beta X)^{-\frac{A_{0,1}}{\beta}}. \qedhere
\end{equation*}
\end{proof}
\begin{lemma}\label{exponential sum}
Suppose $A \in M_l(K)$ satisfies that $\lim_{n\to+\infty}\prod_{i=0}^n(i\beta+A) = 0$. Further we define $A_0=I$, $A_{n+1}=\prod_{i=0}^{n}(i\beta+A)$ for $n\geq 0$, then for a fixed non negative integer $k$,
\[\sum_{s=0}^{\infty}A_{k+s}X^{[s]}=A_k(1-\beta X)^{-\frac{A}{\beta}-k}.\]
\end{lemma}
\begin{proof}
This is proven in \cite[Lemma 3.6]{MW21}.
\end{proof}

Now we proceed to prove \cref{preliminary lemma}. By assumption for $m,n\geq 1$ we have that
\begin{equation*}
    \begin{split}
        A_{m,n+1}=(\beta(n-m)+A_{0,1})A_{m,n}+\sum_{\substack{i+j=m\\i\leq m-1}}(A_{j,1}+(n-i)\theta_{1,j})A_{i,n}.
    \end{split}
\end{equation*}

Then by induction on $n$ we see that
\begin{equation}\label{inductive m n}
    \begin{split}
        A_{m,s}=&A_{m,1}\sum_{b=1}^{s}A_{0,s-b}\prod_{t=1}^{b-1}(\beta(s-m-t)+A_{0,1})+\sum_{\substack{f+r=m\\1\leq f, r \leq 1}}(A_{r,1}-f\theta_{1,r})\sum_{b=1}^{s-1}A_{f,s-b}\prod_{t=1}^{b-1}(\beta(s-m-t)+A_{0,1})
        \\&+\sum_{\substack{f+r=m\\ f\leq m-1}}\sum_{b=1}^{s-1}A_{f,s-b}\theta_{1,r}(s-b)\prod_{t=1}^{b-1}(\beta(s-m-t)+A_{0,1}).
    \end{split}
\end{equation}

Our strategy is to treat the three terms on the right side of \cref{inductive m n} separately.

First by assumption $A_{0,t}=\prod_{t=0}^{t-1}(t\beta+A_{0,1})$, hence that
\begin{equation}\label{inductive m term}
    \begin{split}
        A_{m,1}\sum_{b=1}^{s}A_{0,s-b}\prod_{t=1}^{b-1}(\beta(s-m-t)+A_{0,1})&=A_{m,1}A_{0,s-m}\sum_{b=0}^{s-1}(-(m-1)\beta+A_{0,1}+b\beta)\cdots(-\beta+A_{0,1}+b\beta),
    \end{split}
\end{equation}
here we abuse the notation in the following sense:
\begin{itemize}
    \item When $m=1$, $\sum_{b=0}^{n-1}(-(m-1)\beta+A_{0,1}+b\beta)\cdots(-\beta+A_{0,1}+b\beta)$ is defined to be $n$.
    \item $A_{0,-j}(-j\beta+A_{0,1})\cdots (-\beta+A_{0,1})=I$ for $1\leq j \leq t$.
\end{itemize}
To see our abuse of notation makes sense, we need the following lemma:
\begin{lemma}\label{change m calculation order}
For $m\geq 1$, $1\leq i \leq m$, there exists a constant $f_{m,i}$ which only depends on $m$ and $i$ such that
\begin{equation*}
    \sum_{c=0}^{s-1}(-(m-1)\beta+A_{0,1}+c\beta)\cdots(-\beta+A_{0,1}+c\beta)=\sum_{i=1}^{m}(-(m-i)\beta+A_{0,1})\cdots(-\beta+A_{0,1})s(s-1)\cdots(s-i+1)f_{m,i}.
\end{equation*}
Moreover, $f_{m,i}$ is given as in \cref{e=1 known}.
\end{lemma}
\begin{proof}
We prove it by induction on $m$. When $m=1$, $f_{1,1}=1$. Suppose it holds for $m\geq 1$, then consider $m+1$ case:
\begin{equation*}
    \begin{split}
            \sum_{c=0}^{s-1}\prod_{i=1}^{m}(-i\beta+A_{0,1}+c\beta)&=(-m\beta+A_{0,1})\sum_{c=0}^{s-1}\prod_{i=1}^{m-1}(-i\beta+A_{0,1}+c\beta)+\sum_{c=0}^{s-1}c\beta\prod_{i=1}^{m-1}(-i\beta+A_{0,1}+c\beta).
        \end{split}
\end{equation*}
By induction,
\begin{equation}\label{m I}
    \begin{split}
        &(-m\beta+A_{0,1})\sum_{c=0}^{s-1}\prod_{i=1}^{m-1}(-i\beta+A_{0,1}+c\beta)
        \\=&(-m\beta+A_{0,1})\sum_{i=1}^{m}f_{m,i}(\prod_{t=1}^{m-i}(-t\beta+A_{0,1}))(\prod_{j=0}^{i-1}(s-j))
        \\=&\sum_{i=1}^{m}(-(m+1-i)\beta+A_{0,1}-(i-1)\beta)\prod_{t=1}^{m-i}(-t\beta+A_{0,1})\prod_{j=0}^{i-1}(s-j)f_{m,i}
        \\=&\sum_{i=1}^{m}(\prod_{t=1}^{m+1-i}(-t\beta+A_{0,1})\prod_{j=0}^{i-1}(s-j)f_{m,i}-(i-1)\beta f_{m,i}(\prod_{t=1}^{m-i}(-t\beta+A_{0,1}))(\prod_{j=0}^{i-1}(s-j)))
        \\=&\sum_{i=1}^{m+1}\prod_{t=1}^{m+1-i}(-t\beta+A_{0,1})(f_{m,i}\prod_{j=0}^{i-1}(s-j)-(i-2)\beta f_{m,i-1}\prod_{j=0}^{i-2}(s-j)).
    \end{split}
\end{equation}
By induction and Abel's summation formula,
\begin{equation}\label{m II}
    \begin{split}
        &\sum_{c=0}^{s-1}c\beta\prod_{i=1}^{m-1}(-i\beta+A_{0,1}+c\beta)\\=&\beta(s-1)\sum_{c=0}^{s-1}\prod_{i=1}^{m-1}(-i\beta+A_{0,1}+c\beta)-\beta\sum_{c=0}^{s-2}\sum_{t=0}^{c}\prod_{i=1}^{m-1}(-i\beta+A_{0,1}+t\beta)
        \\=&\beta(s-1)\sum_{i=1}^{m}f_{m,i}(\prod_{t=1}^{m-i}(-t\beta+A_{0,1}))(\prod_{j=0}^{i-1}(s-j))-\beta\sum_{c=0}^{s-2}f_{m,i}(\prod_{t=1}^{m-i}(-t\beta+A_{0,1}))(\prod_{j=0}^{i-1}(c+1-j))
        \\=&\sum_{i=1}^{m+1}(\prod_{t=1}^{m+1-i}(-t\beta+A_{0,1}))(\beta(s-1) f_{m,i-1}\prod_{j=0}^{i-2}(s-j)-\beta f_{m,i-1}\sum_{c=0}^{s-2}\prod_{j=0}^{i-1}(c+1-j)).
    \end{split}
\end{equation}
Using that 
\[\sum_{c=0}^{s-2}\prod_{j=0}^{i-1}(c+1-j)=\frac{1}{i}\prod_{j=0}^{i-1}(s-j),\]
and combine \cref{m I} and \cref{m II}, we have that
\begin{align*}
    \begin{split}
        &\sum_{c=0}^{s-1}\prod_{i=1}^{m}(-i\beta+A_{0,1}+c\beta)
        \\=&\sum_{i=1}^{m+1}(\prod_{t=1}^{m+1-i}(-t\beta+A_{0,1}))((f_{m,i}-\frac{\beta}{i}f_{m,i-1})\prod_{j=0}^{i-1}(s-j)+(\beta(s-1) f_{m,i-1}-(i-2)\beta f_{m,i-1})\prod_{j=0}^{i-2}(s-j))
        \\=&\sum_{i=1}^{m+1}(f_{m,i}+\beta f_{m,i-1}-\frac{\beta}{i}f_{m,i-1})(\prod_{t=1}^{m+1-i}(-t\beta+A_{0,1}))(\prod_{j=0}^{i-1}(s-j)),
    \end{split}
\end{align*}
hence the lemma holds for $m+1$ provided that we define $f_{m+1,i}$ to be $f_{m,i}+\beta f_{m,i-1}-\frac{\beta}{i}f_{m,i-1}$.
Finally by induction we calculate $f_{m,i}$ explicitly: 
\begin{equation*}
    f_{m,i}=\frac{\beta^{i-1}}{i!}(m-1)\cdots (m-(i-1)).\qedhere
\end{equation*}
\end{proof}

Now we use similar strategy to calculate the second term on the right handside of \cref{inductive m n}.

For this, we need a slight generalization of \cref{change m calculation order}.
\begin{lemma}\label{change m f i order}
Suppose $m\geq f+1$, then there exist constants $g_{m,f,i}^1,g_{m,f,i}^2,\cdots,g_{m,f,i}^{m-f+i}$ such that
\begin{equation*}
    \begin{split}
         &\sum_{c=0}^{s-1}(-(m-1)\beta+A_{0,1}+c\beta)\cdots(-(f+1)\beta+A_{0,1}+c\beta)c\cdots(c-i+1)
         \\=&\sum_{j=1}^{m-f+i}(-(m-j)\beta+A_{0,1})\cdots(-(f-i+1)\beta+A_{0,1})s\cdots(s-(j-1))g_{m,f,i}^{j}.
    \end{split}
\end{equation*}
Moreover, $g_{m,f,i}^{j}$ is given as in \cref{main theorem II}.
\end{lemma}
\begin{proof}
We prove it by induction on $m$. When $m=f+1$,
\begin{equation*}
    \begin{split}
        \LHS&=\sum_{c=0}^{s-1}c\cdots(c-i+1)
        =\frac{1}{i+1}s\cdots(s-i),
    \end{split}
\end{equation*}
hence we can take $g_{f+1,f,i}^{1}=\cdots g_{f+1,f,i}^{1}=0$ and $g_{f+1,f,i}^{i+1}=\frac{1}{i+1}$ to make the identify holds.

Suppose the lemma is proven for $m\geq f+1$, then consider $m+1$ case:
\begin{equation*}
    \begin{split}
            &\sum_{c=0}^{s-1}(-m\beta+A_{0,1}+c\beta)\cdots\cdots(-(f+1)\beta+A_{0,1}+c\beta)c\cdots(c-i+1)
            \\=&(-m\beta+A_{0,1})\sum_{c=0}^{s-1}(\prod_{t=f+1}^{m-1}(-i\beta+A_{0,1}+c\beta))(\prod_{u=0}^{i-1}(c-u))+\sum_{c=0}^{s-1}c\beta(\prod_{t=f+1}^{m-1}(-i\beta+A_{0,1}+c\beta))(\prod_{u=0}^{i-1}(c-u)).
        \end{split}
\end{equation*}
By induction,
\begin{equation}\label{m f j I}
    \begin{split}
        &(-m\beta+A_{0,1})\sum_{c=0}^{s-1}(\prod_{t=f+1}^{m-1}(-i\beta+A_{0,1}+c\beta))(\prod_{u=0}^{i-1}(c-u))
        \\=&(-m\beta+A_{0,1})\sum_{j=1}^{m-f+i}(-(m-j)\beta+A_{0,1})\cdots(-(f-i+1)\beta+A_{0,1})s\cdots(s-(j-1))g_{m,f,i}^{j}
        \\=&\sum_{j=1}^{m-f+i}(-(m+1-j)\beta+A_{0,1}-(j-1)\beta)g_{m,f,i}^{j}\prod_{t=f-i+1}^{m-j}(-t\beta+A_{0,1})\prod_{u=0}^{j-1}(s-u)
        \\=&\sum_{j=1}^{m-f+i}(g_{m,f,i}^{j}(\prod_{t=f-i+1}^{m+1-j}(-t\beta+A_{0,1}))(\prod_{u=0}^{j-1}(s-u))-(j-1)\beta g_{m,f,i}^{j}(\prod_{t=f-i+1}^{m-j}(-t\beta+A_{0,1}))(\prod_{u=0}^{j-1}(s-u)))
        \\=&\sum_{j=1}^{m+1-f+i}(\prod_{t=f-i+1}^{m+1-i}(-t\beta+A_{0,1}))(g_{m,f,i}^{j}\prod_{u=0}^{j-1}(s-u)-(j-2)\beta g_{m,f,i}^{j-1}\prod_{u=0}^{j-2}(s-u)).
    \end{split}
\end{equation}
By induction and Abel's summation formula,
\begin{equation}\label{m f j II}
    \begin{split}
        &\sum_{c=0}^{s-1}c\beta(\prod_{t=f+1}^{m-1}(-i\beta+A_{0,1}+c\beta))(\prod_{u=0}^{i-1}(c-u))
        \\=&\beta(s-1)\sum_{c=0}^{s-1}(\prod_{t=f+1}^{m-1}(-i\beta+A_{0,1}+c\beta))(\prod_{u=0}^{i-1}(c-u))-\beta\sum_{c=0}^{s-2}\sum_{p=0}^{c}(\prod_{t=f+1}^{m-1}(-i\beta+A_{0,1}+p\beta))(\prod_{u=0}^{i-1}(p-u))
        \\=&\beta(s-1)\sum_{j=1}^{m-f+i}(g_{m,f,i}^{j}(\prod_{t=f-i+1}^{m-j}(-t\beta+A_{0,1}))(\prod_{u=0}^{j-1}(s-u))-\beta\sum_{c=1}^{s-1}g_{m,f,i}^{j}(\prod_{t=f-i+1}^{m-j}(-t\beta+A_{0,1}))(\prod_{j=0}^{i-1}(c-j)))
        \\=&\sum_{j=1}^{m+1-f
        +i}(\prod_{t=f-i+1}^{m+1-j}(-t\beta+A_{0,1}))(\beta(s-1) g_{m,f,i-1}^j\prod_{u=0}^{j-2}(s-u)-\beta g_{m,f,i-1}^j\sum_{c=1}^{s-1}\prod_{u=0}^{j-2}(c-u)).
    \end{split}
\end{equation}
Using that 
\[\sum_{c=1}^{s-1}\prod_{u=0}^{j-2}(c-u)=\frac{1}{j}\prod_{u=0}^{j-1}(s-u),\]
and combine \cref{m f j I} and \cref{m f j II}, we have that
\begin{align*}
    \begin{split}
        &\sum_{c=0}^{s-1}(\prod_{t=f+1}^{m}(-t\beta+A_{0,1}+c\beta))(\prod_{u=0}^{i-1}(c-u))
        \\=&\sum_{j=1}^{m+1-f+i}(\prod_{t=f-i+1}^{m+1-i}(-t\beta+A_{0,1}))((g_{m,f,i}^{j}-\frac{\beta}{j}g_{m,f,i}^{j-1})\prod_{u=0}^{j-1}(s-u)+\beta(s-1-(j-2)) g_{m,f,i}^{j-1}\prod_{u=0}^{j-2}(s-u))
        \\=&\sum_{j=1}^{m+1-f+i}(g_{m,f,i}^{j}-\frac{\beta}{j}g_{m,f,i}^{j-1}+\beta g_{m,f,i}^{j-1})(\prod_{t=f-i+1}^{m+1-i}(-t\beta+A_{0,1}))(\prod_{u=0}^{j-1}(s-u)).
    \end{split}
\end{align*}
Hence the lemma holds for $m+1$ provided that we define $g_{m+1,f,i}^{j}$ to be $g_{m+1,f,i}^{j-1}+\beta g_{m+1,f,i}^{j-1}-\frac{\beta}{j}g_{m+1,f,i}^{j-1}$.
Then a easy induction shows that
\begin{equation*}\label{calculating g II}
      g_{m,f,i}^j= \left\{
      \begin{array}{rcl}
           \frac{\beta^{j-i-i}}{j}\frac{1}{(j-i-1)!}(m-(f+1))\cdots(m-(f+j-i-1)), & m\geq f+1 & i+1\leq j\leq m-f+i  \\
           0, & \text{otherwise}
      \end{array}
      \right. \qedhere
  \end{equation*}
\end{proof}

Now we are ready to prove \cref{preliminary lemma}.

\begin{proof}[Proof of \cref{preliminary lemma}]
We proceed by induction on $m$. When $m=1$,by induction on $s$, one can show $A_{1,s}=sA_{1,1}A_{0,s-1}+A_{0,s-1}s(s-1)\frac{\theta_{1,1}}{2}$, hence we can take $h_{1,1}=A_{1,1}=f_{1,1}A_{1,1}$ and $h_{1,2}=\frac{\theta_{1,1}}{2}$. Suppose it is proven for up to $m-1$. Consider the $m$ case. We treat three terms showing up on the right handside of \cref{inductive m n} separately. 

\textbf{Step I} For a fixed pair $(f,r)$ such that $f+r=m$, $1\leq f,r\leq m-1$, 
\begin{equation}\label{calculation for second term}
    \begin{split}
        &\sum_{b=1}^{s-1}A_{f,s-b}\prod_{t=1}^{b-1}(\beta(s-m-t)+A_{0,1})
        \\=&\sum_{b=1}^{s-1}A_{0,s-b-f}(\sum_{i=1}^{f}(\prod_{t=1}^{f-i}(-t\beta+A_{0,1}))(\prod_{u=0}^{i-1}(s-b-u))h_{f,i}+\sum_{i=f+1}^{2f}h_{f,i}(\prod_{u=0}^{i-1}(s-b-u)))(\prod_{t=1}^{b-1}((s-m-t)\beta+A_{0,1}))
        \\=&\sum_{b=1}^{s-1}A_{0,s-m}(\sum_{i=1}^{f}(\prod_{t=1}^{f-i}(-t\beta+A_{0,1}))(\prod_{u=0}^{i-1}(s-b-u))h_{f,i}+\sum_{i=f+1}^{2f}h_{f,i}(\prod_{u=0}^{i-1}(s-b-u)))(\prod_{t=1}^{r-1}((s-b-f)\beta+A_{0,1}-t\beta))
        \\=&A_{0,s-m}((\sum_{i=1}^{f}((\prod_{t=1}^{f-i}(-t\beta+A_{0,1}))h_{f,i})(\sum_{c=1}^{s-1}(\prod_{u=0}^{i-1}(c-u))(\prod_{t=1}^{r-1}((c-f)\beta+A_{0,1}-t\beta))))
        \\~~~~~~&+\sum_{i=f+1}^{2f}h_{f,i}(\sum_{c=1}^{s-1}(\prod_{u=0}^{i-1}(c-u))(\prod_{t=1}^{r-1}((c-f)\beta+A_{0,1}-t\beta)))).
    \end{split}
\end{equation}
Here the first identity follows from induction, the second equality is due to $A_{0,t}=\prod_{t=0}^{t-1}(t\beta+A_{0,1})$. The last equality just change $s-b$ to $c$.

For $1\leq i \leq f$, we have that $m-f+i\leq m$ and that
\begin{equation}\label{second term I}
    \begin{split}
        &((\prod_{t=1}^{f-i}(-t\beta+A_{0,1}))h_{f,i})(\sum_{c=1}^{s-1}(\prod_{u=0}^{i-1}(c-u))(\prod_{t=1}^{r-1}((c-f)\beta+A_{0,1}-t\beta)))
        \\=&((\prod_{t=1}^{f-i}(-t\beta+A_{0,1}))h_{f,i})(\sum_{c=1}^{s-1}(\prod_{u=0}^{i-1}(c-u))(\prod_{t=f+1}^{m-1}(-t\beta+A_{0,1}+c\beta)))
        \\=&\sum_{j=1}^{m-f+i}h_{f,i}g_{m,f,i}^j(\prod_{t=1}^{m-j}(-t\beta+A_{0,1}))(\prod_{u=0}^{j-1}(s-u))
    \end{split}
\end{equation}
by applying \cref{change m f i order}.

For $f+1\leq i \leq 2f$,
\begin{equation}\label{second term II}
    \begin{split}
        &h_{f,i}(\sum_{c=1}^{s-1}(\prod_{u=0}^{i-1}(c-u))(\prod_{t=1}^{r-1}((c-f)\beta+A_{0,1}-t\beta)))
        \\=&\sum_{j=1}^{m-f+i}h_{f,i}g_{m,f,i}^j(\prod_{t=f-i+1}^{m-j}(-t\beta+A_{0,1}))(\prod_{u=0}^{j-1}(s-u))
        \\=&\sum_{j=1}^{m}h_{f,i}g_{m,f,i}^j(A_{0,i-f}\prod_{t=1}^{m-j}(-t\beta+A_{0,1}))(\prod_{u=0}^{j-1}(s-u))+\sum_{j=m+1}^{m-f+i}h_{f,i}g_{m,f,i}^j\frac{A_{0,i-f}}{A_{0,j-m}}(\prod_{u=0}^{j-1}(s-u)).
    \end{split}
\end{equation}

\textbf{Step II} Finally we treat the last term in \cref{inductive m n}. For this, we fix a pair $(f,r)$ such that $f+r=m$ and $f\leq m-1$. Repeating the calculation as in \cref{calculation for second term} leads to the following:
\begin{equation}\label{third term I}
    \begin{split}
        &\sum_{b=1}^{s-1}A_{f,s-b}(s-b)\prod_{t=1}^{b-1}(\beta(s-m-t)+A_{0,1})
        \\=&A_{0,s-m}((\sum_{i=0}^{f}((\prod_{t=1}^{f-i}(-t\beta+A_{0,1}))h_{f,i})(\sum_{c=1}^{s-1}(c\prod_{u=0}^{i-1}(c-u))(\prod_{t=1}^{r-1}((c-f)\beta+A_{0,1}-t\beta))))
        \\~~~~~~&+\sum_{i=f+1}^{2f}h_{f,i}(\sum_{c=1}^{s-1}(c\prod_{u=0}^{i-1}(c-u))(\prod_{t=1}^{r-1}((c-f)\beta+A_{0,1}-t\beta)))).
    \end{split}
\end{equation}

For $0\leq i \leq f$, arguing as in \cref{second term I}, 
\begin{equation}\label{third term II}
    \begin{split}
        &((\prod_{t=1}^{f-i}(-t\beta+A_{0,1}))h_{f,i})(\sum_{c=1}^{s-1}(u\prod_{u=0}^{i-1}(c-u))(\prod_{t=1}^{r-1}((c-f)\beta+A_{0,1}-t\beta)))
        \\=&\sum_{j=1}^{m-f+i}h_{f,i}(g_{m,f,i}^j(\prod_{t=1}^{m-j}(-t\beta+A_{0,1}))((\prod_{u=0}^{j-1}(s-u))(s-1)-\frac{1}{j+1}\prod_{u=0}^{j}(s-u))
        \\=&\sum_{j=1}^{m-f+i}h_{f,i}(g_{m,f,i}^j(\prod_{t=1}^{m-j}(-t\beta+A_{0,1}))((\prod_{u=0}^{j-1}(s-u))(s-j+j-1)-\frac{1}{j+1}\prod_{u=0}^{j}(s-u))
        \\=&\sum_{j=1}^{m-f+i+1}h_{f,i}((j-1)g_{m,f,i}^j+(1-\frac{1}{j})g_{m,f,i}^{j-1}(-(m-(j-1))\beta+A_{0,1})^{\delta_j})(\prod_{t=1}^{m-j}(-t\beta+A_{0,1}))(\prod_{u=0}^{j-1}(s-u)).
    \end{split}
\end{equation}
Here the first identity follows from \cref{change m f i order} and Abel's summation formula, $\delta_j=1$ if $j\leq m$ and $\delta_j=0$ if $j=m+1$. 

For $f+1\leq i\leq 2f$, arguing as in \cref{second term II}, 
\begin{equation}\label{third term III}
    \begin{split}
        &h_{f,i}(\sum_{c=1}^{s-1}(c\prod_{u=0}^{i-1}(c-u))(\prod_{t=1}^{r-1}((c-f)\beta+A_{0,1}-t\beta)))
        \\=&\sum_{j=1}^{m}h_{f,i}((j-1)g_{m,f,i}^j+(1-\frac{1}{j})g_{m,f,i}^{j-1}(-(m-(j-1))\beta+A_{0,1}))A_{0,i-f}\prod_{t=1}^{m-j}(-t\beta+A_{0,1}))(\prod_{u=0}^{j-1}(s-u))
        \\&+\sum_{j=m+1}^{m-f+i}h_{f,i}(\frac{A_{0,i-f}}{A_{0,j-m}}(j-1)g_{m,f,i}^j+\frac{A_{0,i-f}}{A_{0,j-m-1}}(1-\frac{1}{j})g_{m,f,i}^{j-1}(-(m-(j-1))\beta+A_{0,1}))(\prod_{u=0}^{j-1}(s-u)).
    \end{split}
\end{equation}

\textbf{Step III} Combine \cref{inductive m n},\cref{inductive m term},\cref{change m calculation order},\cref{calculation for second term}, \cref{second term I},\cref{second term II},\cref{third term I}, \cref{third term II} AND \cref{third term III}, we conclude that $h_{m,j}$ ($1\leq j\leq 2m $) can be defined as in \cref{main theorem II} to make \cref{preliminary lemma} holds for $m$.
Hence \cref{preliminary lemma} holds for $m$.
\end{proof}

\subsection{Evidence of the conjecture}
We give some evidence of \cref{conjecture e=1} and explain why \cref{conjecture e=1} can be reduced to the study of polynomials $\{h_{m,j}\}$ showing up in \cref{main theorem II} in the case that $A_{0,1}$ commutes with all $A_{m,1}$. Indeed, given a sequence of matrices $\{A_{m,n}\}$ in $\in M_l(K)$ as in \cref{conjecture e=1}, then 
\begin{equation*}
    \epsilon(\underline{e})=\underline e\cdot \sum_{m\geq0}(\sum_{n\geq 0}A_{m,n}X_1^{[n]})t^m
\end{equation*}
defines an isomorphism as $\cdot \sum_{m\geq0}(\sum_{n\geq 0}A_{m,n}X^{[n]})$ is invertible. To verify that it can be obtained from a de Rham prismatic crystal, it suffices to check that $\varepsilon$ satisfies the cocycle condition, which is obvious if we could show the following conjecture: 
\begin{conjecture}\label{find invertible functions}
Let $a=\frac{-A_{0,1}}{\beta}$ and  $\tilde{f}(T)=\sum_{i=0}^{\infty}a_it^i\in M_{l\times l}(K)[[t]]$ such that $a_0=1$ and $a_k$ is defined inductively for $k\geq 1$ using that
\begin{equation*}
    a_k=\frac{1}{k\beta}(\sum_{\substack{i+s=k\\i\leq k-1}}(d_{i+a,s,1}a_i-A_{s,1})a_i).
\end{equation*}
 Then we have the following identity:
\begin{equation}\label{invertible functions}
    \alpha^a\frac{\tilde{f}(\alpha t)}{\tilde{f}(t)}= \sum_{m\geq0}(\sum_{n\geq 0}A_{m,n}X_1^{[n]})t^m.
\end{equation}
Here $\alpha^a$ is well defined due to the convergence property of $A_{0,1}$.
\end{conjecture}
Actually once \cref{find invertible functions} can be verified, then recall $\alpha$ is defined to be $\frac{E(u_1)}{E(u_0)}$, hence
\begin{equation*}
    \delta_{1}^{2, *}(\varepsilon)(\underline e) =\underline e (\frac{E(u_2)}{E(u_0)})^a\cdot \frac{\tilde{f}(E(u_2))}{\tilde{f}(E(u_0))},
\end{equation*}
while 
\begin{equation*}
    \begin{split}
        \delta_{2}^{2, *}(\epsilon) \circ \delta_{0}^{2, *}(\epsilon)(\underline e) & =  \delta_{2}^{2, *}(\varepsilon)(\underline e(\frac{E(u_2)}{E(u_1)})^a\cdot \frac{\tilde{f}(E(u_2))}{\tilde{f}(E(u_1))})\\
        & = \underline e((\frac{E(u_1)}{E(u_0)})^a\cdot \frac{\tilde{f}(E(u_1))}{\tilde{f}(E(u_0))} (\frac{E(u_2)}{E(u_1)})^a\cdot \frac{\tilde{f}(E(u_2))}{\tilde{f}(E(u_1))})
        \\ &=\underline e (\frac{E(u_2)}{E(u_0)})^a\cdot \frac{\tilde{f}(E(u_2))}{\tilde{f}(E(u_0))}.
     \end{split}
\end{equation*}
Hence $\delta_{1}^{2, *}(\varepsilon)(\underline e)=\delta_{2}^{2, *}(\epsilon) \circ \delta_{0}^{2, *}(\epsilon)(\underline e)$, i.e. the cocycle condition is satisfied.

On the other hand, we claim that the proof of \cref{invertible functions} is purely an algebraic problem with respect to polynomials $\{h_{m,j}\}$ in \cref{main theorem II}. Actually, \cref{invertible functions} is equivalent to that 
\begin{equation}
    \alpha^a \tilde{f}(\alpha t)\tilde{f}(t)= (\sum_{m\geq0}(\sum_{n\geq 0}A_{m,n}X_1^{[n]})t^m)\tilde{f}(t),
\end{equation}
 then we calculate that
\begin{equation*}
    \begin{split}
        \LHS&=\sum_{i=0}^{\infty}a_i\alpha^{i+a}t^i=\sum_{i=0}^{\infty}a_i(\sum_{s=0}^{\infty}c_{i+a,s}t^s)t^i=\sum_{k=0}^{\infty}t^k(\sum_{i+s=k}c_{i+a,s}a_i)
        \\&=\sum_{k=0}^{\infty}t^k(\sum_{i+s=k}(\sum_{n=0}^{\infty}d_{i+a,s,n}X_1^{[n]})a_i),
    \end{split}
\end{equation*}
and that 
\begin{equation*}
    \begin{split}
        \RHS&=(\sum_{m\geq0}(\sum_{n\geq 0}A_{m,n}X_1^{[n]})t^m)\tilde{f}(t)
        \\&=\sum_{k=0}^{\infty}t^k(\sum_{m+l=k}(\sum_{n=0}^{\infty}A_{m,n}X_1^{[n]})a_l).
    \end{split}
\end{equation*}
Hence \cref{find invertible functions} is true if and only if for any $k\geq 0$,
\begin{equation}\label{theory calculation}
    \sum_{i+s=k}(\sum_{n=0}^{\infty}d_{i+a,s,n}X_1^{[n]})a_i=\sum_{m+l=k}(\sum_{n=0}^{\infty}A_{m,n}X_1^{[n]})a_l.
\end{equation}
But recall $\alpha^p$ satisfies cocycle condition (see the paragraph after \cref{find invertible functions}) and in our notation $\alpha^p=\sum_{s=0}^{\infty}c_{p,s}t^s=\sum_{s=0}^{\infty}(\sum_{k=0}^{\infty}d_{p,s,k}X_1^{[k]})t^s$, 
hence theoretically we can use \cref{main theorem II} and \cref{calculate c} to calculate both sides of \cref{theory calculation} and verify it provided that we understand the polynomial $\{h_{m,j}\}$ in \cref{main theorem II} very well!

We calculate the difference between left hand side and right hand side of \cref{theory calculation} for small degrees using $d_{p,s,1}=-p\theta_{1,s}$ thanks to \cref{calculate c} (We implicitly extend this notation to allow $p=a=\frac{-A_{0,1}}{\beta}$ thanks to convergence property of $A_{0,1}$).

When $k=0$,
\begin{equation*}
    \LHS-\RHS=\sum_{n=0}^{\infty}d_{a,0,n}X_1^{[n]}-\sum_{n=0}^{\infty}A_{0,n}X_1^{[n]}=(1-\beta X_1)^{\frac{-A_{0,1}}{\beta}}-(1-\beta X_1)^{\frac{-A_{0,1}}{\beta}}=0.
\end{equation*}
Here the second equality holds as $d_{a,0,1}=-a\beta=A_{0,1}$ and then we can apply \cref{main theorem II}.

When $k=1$, one use that $a_0=1$, $d_{a+1,0,1}=-\beta+A_{0,1}$, $d_{a,1,1}=\frac{A_{0,1}\theta_{1,1}}{\beta}$  and then use \cref{m small} to see
\begin{equation*}
\begin{split}
    \LHS&=(\sum_{n=0}^{\infty}d_{1+a,0,n}X_1^{[n]})a_1+\sum_{n=0}^{\infty}d_{a,1,n}X_1^{[n]}
    \\&=((1-\beta X_1)^{\frac{-d_{a+1,0,1}}{\beta}}a_1)+(d_{a,1,1}X_1(1-\beta X_1)^{\frac{-d_{a,0,1}}{\beta}}+\frac{\theta_{1,1}}{2}X_1^2d_{a,0,1}(1-\beta X_1)^{-1-\frac{d_{a,0,1}}{\beta}})
    \\&=(1-\beta X_1)^{\frac{-A_{0,1}}{\beta}}((1-\beta X_1)a_1+\frac{A_{0,1}\theta_{1,1}X_1}{\beta}+\frac{\theta_{1,1}}{2}A_{0,1}\frac{X_1^2}{1-\beta X_1}),
\end{split}
\end{equation*}
and that 
\begin{equation*}
\begin{split}
    \RHS&=(\sum_{n=0}^{\infty}A_{0,n}X_1^{[n]})a_1+\sum_{n=0}^{\infty}A_{1,n}X_1^{[n]}
    \\&=(1-\beta X_1)^{\frac{-A_{0,1}}{\beta}}(a_1+A_{1,1}X_1+\frac{\theta_{1,1}}{2}A_{0,1}\frac{X_1^2}{1-\beta X_1}).
\end{split}
\end{equation*}
By definition $a_1=\frac{-A_{1,1}}{\beta}+\frac{\theta_{1,1}A_{0,1}}{\beta^2}$, then one easily verifies that $\LHS=\RHS$ for $k=1$.

When $k=2$, 
\begin{equation*}
    \begin{split}
        \LHS&=(\sum_{n=0}^{\infty}d_{2+a,1,n}X_1^{[n]})a_2+(\sum_{n=0}^{\infty}d_{1+a,1,n}X_1^{[n]})a_1+\sum_{n=0}^{\infty}d_{a,2,n}X_1^{[n]},
        \\\RHS&=(\sum_{n=0}^{\infty}A_{0,n}X_1^{[n]})a_2+(\sum_{n=0}^{\infty}A_{1,n}X_1^{[n]})a_1+\sum_{n=0}^{\infty}A_{2,n}X_1^{[n]}.
    \end{split}
\end{equation*}
Using $d_{a,0,1}=-a\beta=A_{0,1}, d_{a,1,1}=\frac{A_{0,1}\theta_{1,1}}{\beta}, d_{a,2,1}=\frac{A_{0,1}\theta_{1,2}}{\beta}$, $a_1=\frac{-A_{1,1}}{\beta}+\frac{\theta_{1,1}A_{0,1}}{\beta^2}$ and \cref{m small}, we have that
\begin{equation}\label{k=2 I}
\begin{split}
    &\sum_{n=0}^{\infty}A_{2,n}X_1^{[n]}+(\sum_{n=0}^{\infty}A_{1,n}X_1^{[n]})a_1-\sum_{n=0}^{\infty}d_{a,2,n}X_1^{[n]}-(\sum_{n=0}^{\infty}d_{1+a,1,n}X_1^{[n]})a_1
    \\=&(1-\beta X_1)^{\frac{-A_{0,1}}{\beta}}(\frac{1}{2}\theta_{1,1}(A_{1,1}-\frac{A_{0,1}\theta_{1,1}}{\beta})A_{0,1}\frac{X_1^3}{1-\beta X_1}+\frac{\beta}{2}(A_{2,1}-\frac{A_{0,1}\theta_{1,2}}{\beta})X_1^2+\frac{1}{2}(A_{1,1}^2-\frac{A_{0,1}^2\theta_{1,1}^2}{\beta^2})X_1^2
    \\+&(A_{2,1}-\frac{A_{0,1}\theta_{1,2}}{\beta})(1-\beta X_1)X_1+(A_{1,1}X_1+\frac{\theta_{1,1}A_{0,1}}{2}\frac{X_1^2}{1-\beta X_1}
    \\+&\theta_{1,1}(1-\frac{A_{0,1}}{\beta})X_1(1-\beta X_1)-\frac{\theta_{1,1}}{2}(-\beta+A_{0,1})X_1^2)\frac{\theta_{1,1}A_{0,1}-\beta A_{1,1}}{\beta^2})
    \\=&(1-\beta X_1)^{\frac{-A_{0,1}}{\beta}}\frac{X_1(\beta X_1-2)}{2}(\frac{\theta_{1,2}A_{0,1}}{\beta}-A_{2,1}+(-\theta_{1,1}+\frac{\theta_{1,1}A_{0,1}}{\beta}-A_{1,1})(\frac{-A_{1,1}}{\beta}+\frac{\theta_{1,1}A_{0,1}}{\beta^2})),
\end{split}
\end{equation}
while by definition we calculate that  $a_2=\frac{1}{2\beta}(\frac{\theta_{1,2}A_{0,1}}{\beta}-A_{2,1}+(-\theta_{1,1}+\frac{\theta_{1,1}A_{0,1}}{\beta}-A_{1,1})(\frac{-A_{1,1}}{\beta}+\frac{\theta_{1,1}A_{0,1}}{\beta^2}))$, hence 
\begin{equation}\label{k=2 II}
\begin{split}
    &(\sum_{n=0}^{\infty}A_{0,n}X_1^{[n]})a_2-(\sum_{n=0}^{\infty}d_{2+a,1,n}X_1^{[n]})a_2
    \\=&(1-\beta X_1)^{\frac{-A_{0,1}}{\beta}}\frac{X_1(2-\beta X_1)}{2}(\frac{\theta_{1,2}A_{0,1}}{\beta}-A_{2,1}+(-\theta_{1,1}+\frac{\theta_{1,1}A_{0,1}}{\beta}-A_{1,1})(\frac{-A_{1,1}}{\beta}+\frac{\theta_{1,1}A_{0,1}}{\beta^2})).
\end{split}
\end{equation}
Combine \cref{k=2 I} and \cref{k=2 II}, we get the desired result $\LHS-\RHS=0$ for $k=2$.

Hopefully these calculations in the small degree give some evidence of \cref{conjecture e=1}. For larger $k$, the verification of \cref{conjecture e=1} is hard, even if we further assume that $A_{0,1}$ commutes with all $A_{m,1}$, we still need to understand $\{h_{m,j}\}$ better to prove \cref{conjecture e=1} or \cref{find invertible functions}.

\section{Absolute prismatic cohomology of de Rham crystals}
In this section we study the absolute prismatic cohomology of de Rham crystals $\mathcal{M} \in \operatorname{Vect}((\mathcal{O}_{K})_{\Prism}, (\mathcal{O}_{\Prism}[\frac{1}{p}])_{\mathcal{I}}^{\wedge})$. Our strategy is to study the cohomology of its restriction to prisms in $(\mathcal{O}_{K})_{\Prism}$ first, then use \v{C}ech-Alexander complex to calculate $R \Gamma_{\Prism}(\mathcal{O}_{K},\mathcal{M})$.

\begin{lemma}\label{vanishing of prisms}
Let $\mathcal{M} \in \operatorname{Vect}((\mathcal{O}_{K})_{\Prism}, (\mathcal{O}_{\Prism}[\frac{1}{p}])_{\mathcal{I}}^{\wedge})$, then for any $U=(A,I)\in (\mathcal{O}_{K})_{\Prism}$ and $q>0$, we have that $$H^q(U,\mathcal{M})=0,$$
and that \[H^q(U,\mathcal{M}/\mathcal{I}^n)=0.\]
Here $\mathcal{I}^n$ is the ideal sheaf on $(\mathcal{O}_{K})_{\Prism}$ by sending $(A,I)$ to $I^n\bdr(A)$.
\end{lemma}
\begin{proof}
Let $(A,I)\rightarrow (B,IB)$ be a cover in $(\mathcal{O}_{K})_{\Prism}$, denote $V=(B,IB)$  we claim that the \v{C}ech complex of $\mathcal{M}$ for this cover $R \Gamma(U^{\bullet},\mathcal{M})$ is concentrated in degree $0$. By derived Nakayama lemma, it suffices to check that $R \Gamma(U^{\bullet},\mathcal{M})\otimes_{\bdr(A)}^{\mathbb{L}}(\bdr(A)/I)$ is concentrated in degree $0$. However, $\bdr(A)/I$ has $I$-complete Tor amplitude $[-1,0]$ as a $\bdr(A)$-module, hence $-\otimes_{\bdr(A)}(\bdr(A)/I)$ commutes with totalization by \cite[Corollary 3.1.13]{KP21}, this implies that $R \Gamma(U^{\bullet},\mathcal{M})\otimes_{\bdr(A)}^{\mathbb{L}}(\bdr(A)/I)$ is the \v{C}ech complex of the rational Hodge-Tate crystal $\mathcal{M}/\mathcal{I}$ thanks to \cref{specializing Hodge Tate}, which is concentrated in degree $0$ by \cite[Lemma 3.18]{MW21}.

This implies that $H^q(U,\mathcal{M})=0$  by \cite[\href{https://stacks.math.columbia.edu/tag/01EV}{Tag 01EV}]{SP22}.

The desired results for $\mathcal{M}/\mathcal{I}^n$ can be proven similarly thanks to \cref{mod n analogue}.
\end{proof}

\begin{cor}\label{alex}
Let $\mathcal{M} \in \operatorname{Vect}((\mathcal{O}_{K})_{\Prism}, (\mathcal{O}_{\Prism}[\frac{1}{p}])_{\mathcal{I}}^{\wedge})$, then the absolute prismatic cohomology $R \Gamma_{\Prism}(\mathcal{O}_{K},\mathcal{M})$ can be calculated by the \v{C}ech-Alexander complex
\begin{equation}\label{alexander}
    \mathcal{M}(\mathfrak{S},(E))\stackrel{d_0}{\rightarrow}\mathcal{M}(\mathfrak{S}^1,(E))\stackrel{d_1}{\rightarrow}\mathcal{M}(\mathfrak{S}^2,(E))\rightarrow\cdots,
\end{equation}
where $d_n=\sum_{i=0}^{n+1}(-1)^i\delta_{i}^{n+1}$.

Similarly, $R \Gamma_{\Prism}(\mathcal{O}_{K},\mathcal{M}/\mathcal{I}^n)$ can be calculated by the \v{C}ech-Alexander complex
\begin{equation}\label{alexander mod n}
    \mathcal{M}/\mathcal{I}^n(\mathfrak{S},(E))\stackrel{d_0}{\rightarrow}\mathcal{M}/\mathcal{I}^n(\mathfrak{S}^1,(E))\stackrel{d_1}{\rightarrow}\mathcal{M}/\mathcal{I}^n(\mathfrak{S}^2,(E))\rightarrow\cdots.
\end{equation}
\end{cor}
\begin{proof}
As $(\mathfrak{S}^n,(E))$ is the \v{C}ech complex associated to the weakly final object $(\mathfrak{S},(E))$ in $\mathcal{O}_{K})_{\Prism}$, it is a standard fact that $R \Gamma_{\Prism}(\mathcal{O}_{K},\mathcal{M})$ (resp. $R \Gamma_{\Prism}(\mathcal{O}_{K},\mathcal{M}/\mathcal{I}^n)$) can be computed via $\Tot(R\Gamma (\mathfrak{S}^\bullet,\mathcal{M}))$ (resp. $\Tot(R\Gamma (\mathfrak{S}^\bullet,\mathcal{M}/\mathcal{I}^n))$), which is just \cref{alexander} (resp. \cref{alexander mod n}) thanks to \cref{vanishing of prisms}.
\end{proof}

\begin{theorem}\label{main theorem cohomology}
    Given $\mathcal{M} \in \operatorname{Vect}((\mathcal{O}_{K})_{\Prism}, (\mathcal{O}_{\Prism}[\frac{1}{p}])_{\mathcal{I}}^{\wedge})$, then $ H^i((\mathcal{O}_{K})_{\Prism},\mathcal{M})=0$ for $i>1$.
\end{theorem}
To prove this theorem, we need to consider the restricted site $(\mathcal{O}_{K})_{\Prism}^{\prime}$ introduced in \cite[Section 3.4]{MW21}, whose underlying category is the full subcategory of $(\mathcal{O}_{K})_{\Prism}$ spanned by those objects admitting maps from $(\mathfrak{S},(E))$ with coverings inherited from those in $(\mathcal{O}_{K})_{\Prism}$.

The reason to restrict to $(\mathcal{O}_{K})_{\Prism}^{\prime}$ is that for any $(A,I)\in (\mathcal{O}_{K})_{\Prism}^{\prime}$, $I=EA$ is oriented by rigidity of morphism of prisms (See \cite[Prop 1.5]{BS19}), hence it is more convenient to relate $\mathcal{M}/\mathcal{I}^{n+1}$ and $\mathcal{M}/\mathcal{I}^{n}$, for example see the proof of \cref{induct on n for cohomology}.

We warn the reader that $(\mathcal{O}_{K})_{\Prism}^{\prime}$ is not the relative prismatic site $(\mathcal{O}_{K}/(\mathfrak{S},(E)))_{\Prism}$ as for example one can see that products are calculated differently.

By abuse of notation we still use $\mathcal{M}$ to define its restriction on $(\mathcal{O}_{K})_{\Prism}^{\prime}$.
\begin{lemma}\label{restriction to}
For $\mathcal{M} \in \operatorname{Vect}((\mathcal{O}_{K})_{\Prism}, (\mathcal{O}_{\Prism}[\frac{1}{p}])_{\mathcal{I}}^{\wedge})$, the following holds:
\begin{itemize}
    \item $R\Gamma((\mathcal{O}_{K})_{\Prism}^{\prime},\mathcal{M})=R\Gamma((\mathcal{O}_{K})_{\Prism},\mathcal{M})=\text{\v{C}A}(\mathcal{M})$.
    \item $R\Gamma((\mathcal{O}_{K})_{\Prism}^{\prime},\mathcal{M}/E^n)=R\Gamma((\mathcal{O}_{K})_{\Prism},\mathcal{M}/\mathcal{I}^n)=\text{\v{C}A}(\mathcal{M}/\mathcal{I}^n)$.
\end{itemize}
here \v{C}A denotes the \v{C}ech-Alexander complex associated to the \v{C}ech Nerve of the weakly final object $(\mathfrak{S},(E))$.
\end{lemma}
\begin{proof}
For the first part, in \cref{alex}, we already prove $R\Gamma((\mathcal{O}_{K})_{\Prism},\mathcal{M})=\text{\v{C}A}(\mathcal{M})$, then a similar argument shows that $R\Gamma((\mathcal{O}_{K})_{\Prism}^{\prime},\mathcal{M})=\text{\v{C}A}(\mathcal{M})$ as $\text{\v{C}A}(\mathcal{M})$ is also the \v{C}ech Nerve of the weakly final object $(\mathfrak{S},(E))$ in $(\mathcal{O}_{K})_{\Prism}^{\prime}$.

The second part follows from a similar argument.
\end{proof}

\begin{lemma}\label{induct on n for cohomology}
$H^i((\mathcal{O}_{K})_{\Prism}^{\prime}, \mathcal{M}/\mathcal{I}^n)=0$ for any $i>1$, $n\geq 1$.
\end{lemma}
\begin{proof}
We proceed by induction on $n$. When $n=1$, as $\mathcal{M}/\mathcal{I}$ is a rational Hodge-Tate crystal by \cref{specializing Hodge Tate}, hence $H^i((\mathcal{O}_{K})_{\Prism}^{\prime}, \mathcal{M}/\mathcal{I})=0$ for any $i>1$ by \cite[Thm 3.20]{MW21} and \cite[Lem 3.26]{MW21}. 

For the induction process, consider the short exact sequence of abelian sheeaves 
\begin{equation*}
    \mathcal{M}/\mathcal{I}^{n-1}\stackrel{\times E}{\longrightarrow} \mathcal{M}/\mathcal{I}^{n} \longrightarrow \mathcal{M}/\mathcal{I},
\end{equation*}
which induces a long exact sequence 
\begin{equation*}
    \cdots \longrightarrow H^i((\mathcal{O}_{K})_{\Prism}^{\prime}, \mathcal{M}/\mathcal{I}^{n-1})\longrightarrow H^i((\mathcal{O}_{K})_{\Prism}^{\prime}, \mathcal{M}/\mathcal{I}^{n}) \longrightarrow H^i((\mathcal{O}_{K})_{\Prism}^{\prime}, \mathcal{M}/\mathcal{I})\longrightarrow \cdots.
\end{equation*}
Hence by induction $H^i((\mathcal{O}_{K})_{\Prism}^{\prime}, \mathcal{M}/\mathcal{I}^n)=0$ for any $i>1$, we are done.
\end{proof}

\begin{lemma}
For any $\mathcal{M} \in \operatorname{Vect}((\mathcal{O}_{K})_{\Prism}^{\prime}, (\mathcal{O}_{\Prism}[\frac{1}{p}])_{\mathcal{I}}^{\wedge})$, $\mathcal{M}=R\lim \mathcal{M}/\mathcal{I}^n$.
\end{lemma}
\begin{proof}
Notice that the topos $\Shv((\mathcal{O}_{K})_{\Prism}^{\prime})$ is replete and that  $\mathcal{M}/\mathcal{I}^{n+1}\rightarrow \mathcal{M}/\mathcal{I}^n$ is surjective for each $n$, hence $R\lim \mathcal{M}/\mathcal{I}^n\cong \lim \mathcal{M}/\mathcal{I}^n$ by \cite[Prop.3.1.10]{BS15}. Then the result follows as $\mathcal{M}$ is $\mathcal{I}$-adically complete. 
\end{proof}

\begin{proof}[Proof of \cref{main theorem cohomology}]
Notice that  $$R\Gamma((\mathcal{O}_{K})_{\Prism}^{\prime},\mathcal{M})=R\lim R\Gamma((\mathcal{O}_{K})_{\Prism}^{\prime},\mathcal{M}/\mathcal{I}^n)$$
as $R\Gamma$ commutes with taking derived inverse limit.
Also, the cohomology of inverse limits can be calculated via the following short exact sequence (See \cite[\href{https://stacks.math.columbia.edu/tag/07KY}{Tag 07KY}]{SP22})
\begin{equation*}
     0 \rightarrow R^{1}\lim_n H^{j-1}((\mathcal{O}_{K})_{\Prism}^{\prime},\mathcal{M}/\mathcal{I}^n) \rightarrow H^{j}((\mathcal{O}_{K})_{\Prism}^{\prime},\mathcal{M})\rightarrow \lim_n H^{j}((\mathcal{O}_{K})_{\Prism}^{\prime},\mathcal{M}/\mathcal{I}^n) \rightarrow 0.
\end{equation*}
As a result, $H^{j}((\mathcal{O}_{K})_{\Prism}^{\prime},\mathcal{M})=0$ for $j\geq 3$ by applying \cref{induct on n for cohomology}. For $j=2$, we still need to show $R^{1}\lim_n H^{1}((\mathcal{O}_{K})_{\Prism}^{\prime},\mathcal{M}/\mathcal{I}^n)=0$ to guarantee that $H^{2}((\mathcal{O}_{K})_{\Prism}^{\prime},\mathcal{M})=0$. To do this,  it suffices to show that for any $n$,
\begin{equation*}
    H^{1}((\mathcal{O}_{K})_{\Prism}^{\prime},\mathcal{M}/\mathcal{I}^{n+1})\longrightarrow H^{1}((\mathcal{O}_{K})_{\Prism}^{\prime},\mathcal{M}/\mathcal{I}^n)
\end{equation*}
is surjective. However, this can be seen via the long exact sequence associated to the short exact sequence
\begin{equation*}
    0\longrightarrow \mathcal{M}/\mathcal{I}\stackrel{\times E^n}{\longrightarrow} \mathcal{M}/\mathcal{I}^{n+1}\longrightarrow \mathcal{M}/\mathcal{I}^{n}\longrightarrow 0.
\end{equation*}
Now we conclude that $H^{j}((\mathcal{O}_{K})_{\Prism}^{\prime},\mathcal{M})=0$ for $j\geq 2$. Hence $H^{j}((\mathcal{O}_{K})_{\Prism},\mathcal{M})=0$ for $j\geq 2$ thanks to \cref{restriction to} again. 
\end{proof}

Next we would like to give a detailed study of $H^i((\mathcal{O}_{K})_{\Prism},\mathcal{M})$ based on our stratification data.

To give a rank $l$ de Rham crystal $\mathcal{M} \in \operatorname{Vect}((\mathcal{O}_{K})_{\Prism}, (\mathcal{O}_{\Prism}[\frac{1}{p}])_{\mathcal{I}}^{\wedge})$, by \cref{algebraic description of crystal}, it is equivalent to specify a finite free $\bdr(\mathfrak{S})$-module $M$ of rank $l$ equipped with a stratification $\varepsilon$ satisfying cocyle condition. 

Notice that every element $v\in M$ can be written as $v=\underline e \Vec{f}$ for some $\Vec{f}\in M_{l\times 1}(K[[T]])$, further denote $\Vec{f}=\sum_{i=0}^{\infty}B_it^i$ with $B_i\in M_{l\times 1}(K)$, then unwinding the definition of $\delta_{0}^1$ and use the notation in Section 2, wee see that
\begin{equation*}
    \begin{split}
        \delta_{0}^1(v)=\delta_{0}^1(\underline e \Vec{f})=\delta_{0}^1(\underline e (\sum_{i=0}^{\infty}B_it^i))&=\varepsilon(\underline e)(\sum_{i=0}^{\infty}B_i(\alpha t)^i)
        \\&=\underline e\cdot (\sum_{m\geq0}(\sum_{n\geq 0}A_{m,n}X^{[n]})t^m)(\sum_{i=0}^{\infty}B_i(\alpha t)^i)
        \\&=\underline e\cdot (\sum_{m\geq0}(\sum_{n\geq 0}A_{m,n}X^{[n]})t^m)(\sum_{i=0}^{\infty}B_i t^i(\sum_{s=0}^{\infty}c_{i,s}t^s))
        \\&=\underline e\cdot (\sum_{m\geq0}(\sum_{n\geq 0}A_{m,n}X^{[n]})t^m)(\sum_{j=0}^{\infty}t^j(\sum_{i=0}^{j}B_i c_{i,j-i}))
        \\&=\underline e\sum_{m=0}^{\infty}t^m(\sum_{\substack{i+j=m\\0\leq i,j\leq m}}(\sum_{s=0}^{\infty}A_{i,s}X_{1}^{[s]})(\sum_{p=0}^{j}B_p c_{p,j-p})).
    \end{split}
\end{equation*}
Hence by \cref{alex},  $v\in H^{0}_{\Prism}(\mathcal{O}_{K},\mathcal{M})$, which is equivalent to that $0=d_0(v)=\delta_{0}^1(v)-\delta_{1}^1(v)=\delta_{0}^1(v)-v$, if and only if the following holds:
\begin{equation*}
    \begin{split}
        \underline e\sum_{m=0}^{\infty}t^m(\sum_{\substack{i+j=m\\0\leq i,j\leq m}}(\sum_{s=0}^{\infty}A_{i,s}X_{1}^{[s]})(\sum_{p=0}^{j}B_p c_{p,j-p}))=\underline e\sum_{i=0}^{\infty}B_it^i.
    \end{split}
\end{equation*}
Compare the coefficients of $B_m$, we see this is equivalent to that
\begin{equation}\label{global section of crystal}
    \begin{split}
        B_m=\sum_{\substack{i+j=m\\0\leq i,j\leq m}}(\sum_{s=0}^{\infty}A_{i,s}X_{1}^{[s]})(\sum_{p=0}^{j}B_p c_{p,j-p})
    \end{split}.
\end{equation}
We explore the conditions on $B_m$ to guarantee that \cref{global section of crystal} holds.

First compare the constant term on the both sides of \cref{global section of crystal}. On the right hand side of \cref{global section of crystal}, we get 
\[A_{0,0}\sum_{p=0}^m B_p d_{p,m-p,0}=B_m c_{m,0}=B_m,\]
here we use the fact that $A_{j,0}=0$ for $j>0$ by \cref{Main theorem I} and that $d_{p,s,0}=0$ for $s>0$ by unwinding definition of $d_{p,s,j}$.

Hence we obtain nothing new when considering the constant term.

Next we compare the coefficent of $X_1$ on the both sides of \cref{global section of crystal}. We compute the coefficient of $X_1$ on the right hand side, which is
\begin{equation}
    \begin{split}
        A_{0,0}(\sum_{p=0}^{m}B_p d_{p,m-p,1})+\sum_{i+j=m} A_{i,1}B_{j}=\sum_{p=0}^m(-p\theta_{1,m-p}B_p)+\sum_{i+j=m} A_{i,1}B_{j}.
    \end{split}
\end{equation}
Here we have used \cref{calculate c}. 
This calculation implies the following:
\begin{prop}\label{prismatic invariant}
Given $v=\underline e \Vec{f}$ where $\Vec{f}=\sum_{i=0}^{\infty}B_it^i$ with $B_i\in M_{l\times 1}(K)$, then $v\in H^{0}_{\Prism}(\mathcal{O}_{K},\mathcal{M})$ implies that 
\[(A_{0,1}-m\beta)B_m=\sum_{p=0}^{m-1}(p\theta_{1,m-p}-A_{m-p,1})B_p\]
for any non-negative integers. In particular, by taking $m=0$, we have that $A_{0,1}B_0=0$
\end{prop}
\begin{proof}
By assumption, \cref{global section of crystal} holds, hence by comparing the coefficient of $X_1$ on the both sides, we have that
\[0=\sum_{p=0}^m(-p\theta_{1,m-p}B_p)+\sum_{i+j=m} A_{i,1}B_{j},\]
which implies the desired results.
\end{proof}
\begin{remark}
In \cite[Theorem 3.20]{MW21}, they show that for $w$ in the rational Hodge-Tate crystal $\mathcal{M}/\mathcal{I}$, $w\in H^0((\mathcal{O}_{K})_{\Prism},\mathcal{M}/\mathcal{I})$ if and only if $A_{0,1}w=0$, hence our results are compatible with theirs.
\end{remark}

\begin{cor}
Suppose $\mathcal{M} \in \operatorname{Vect}((\mathcal{O}_{K})_{\Prism}, (\mathcal{O}_{\Prism}[\frac{1}{p}])_{\mathcal{I}}^{\wedge})$. Further assume that that none of the Sen weights of the rational Hodge-Tate crystal $\mathcal{M}/\mathcal{I}$ are non-positive integers, then $H^0((\mathcal{O}_{K})_{\Prism},\mathcal{M})=0$.
\end{cor}
\begin{proof}
Under assumption suppose $v=\underline e \Vec{f}$ where $\Vec{f}=\sum_{i=0}^{\infty}B_it^i$ with $B_i\in M_{l\times 1}(K)$ is an element of  $H^i((\mathcal{O}_{K})_{\Prism},\mathcal{M})$, then by \cref{prismatic invariant},
\begin{equation}\label{sen weights}
    (A_{0,1}-m\beta)B_m=\sum_{p=0}^{m-1}(p\theta_{1,m-p}-A_{m-p,1})B_p.
\end{equation}
Recall that the Sen weights of $\mathcal{M}/\mathcal{I}$ are eigenvalues of the Sen operator $-\frac{A_{0,1}}{\beta}$ by \cite[Theorem 1.1.7]{Gao22}. On the other hand, by assumption none of them are non-positive integers, hence the eigenvalues of $A_{0,1}-m\beta$ are non-zero for any $m\in \mathbb{N}^{\geq 0}$. In other words, all $A_{0,1}-m\beta$ are invertible. We claim that this implies that $B_i$ all vanish. We proceed by induction on $i$.

First take $m=0$ in \cref{sen weights}, hence $A_{0,1}B_0=0$, which implies that $B_0=0$ as $A_{0,1}$ is invertible. 

For the induction process, suppose we have proven that $B_{i}=0$ for all $i\leq m-1$ ($m\geq 2$), then \cref{sen weights} guarantees that $B_{m}$ also vanishes as $A_{0,1}-m\beta$ is invertible.

As a result, $v=0$. This implies that there are no non-zero elements in  $H^i((\mathcal{O}_{K})_{\Prism},\mathcal{M})$, hence $H^i((\mathcal{O}_{K})_{\Prism},\mathcal{M})=0$.
\end{proof}
\begin{cor}
Let $\mathcal{M} \in \operatorname{Vect}((\mathcal{O}_{K})_{\Prism}, (\mathcal{O}_{\Prism}[\frac{1}{p}])_{\mathcal{I}}^{\wedge})$ be a rank $l$ de Rham crystal. Define $q$ to be the number of non-positive integers (with multiplicity) among Sen weights of the rational Hodge-Tate crystal $\mathcal{M}/\mathcal{I}$, then $H^0((\mathcal{O}_{K})_{\Prism},\mathcal{M})$ is a $K$-vector space of dimension at most $q$. In particular, $\dim_K H^0((\mathcal{O}_{K})_{\Prism},\mathcal{M})\leq l$.
\end{cor}
\begin{proof}
Clearly $H^0((\mathcal{O}_{K})_{\Prism},\mathcal{M})$ is a $K$-vector space as $d_0$ is $K$-linear. Suppose $s_1,\cdots,s_k$ are non-positive integers which show up in the set of Sen weights the rational Hodge-Tate crystal $\mathcal{M}/\mathcal{I}$ with multiplicity $\lambda_1,\cdots,\lambda_k$, then $\lambda_1+\cdots+\lambda_k=q$ by assumption. As the Sen weights of $\mathcal{M}/\mathcal{I}$ are precisely eigenvalues of the Sen operator $-\frac{A_{0,1}}{\beta}$ by \cite[Theorem 1.1.7]{Gao22}, hence $A_{0,1}-m\beta$ are invertible unless $m\in\{s_1,\cdots,s_k\}$. On the other hand, the solution space of $(A_{0,1}-s_i\beta)X=Y$ is of dimension at most $\lambda_i$ for any $X,Y\in M_{l\times 1}(K)$. By \cref{prismatic invariant}, $H^0((\mathcal{O}_{K})_{\Prism},\mathcal{M})$ is of dimension at most  $\lambda_1+\cdots+\lambda_k=q$ as a $K$-vector space.
\end{proof}

\section{Locally analytic vectors and Decompletion theorem}
\subsection{First descent step}
We aim to prove \cref{de rham almost purity} in this subsection, which could be thought as a purity for de Rham representations. This result should be familiar to experts, but we decide to write it down for completeness. Our strategy is to imitate the proof written in \cite[Section 15.2]{BC09} treating the cyclotomic tower.
\begin{lemma}\label{compatibility I}
$L_{\dR}^{+}:=(B_{\dR}^+)^{G_L}$ is a closed $L$-subalgebra of $B_{\dR}^+$ that is a complete discrete valuation ring with uniformizer $t$ and residue field $\hat{L}$. Moreover, the topological ring $L_{\dR}^{+}$ is separated and complete for its subspace topology from $B_{\dR}^+$ for which the multiplication map $t: L_{\dR}^{+}\rightarrow L_{\dR}^{+}$ defines a closed embedding.
\end{lemma}
\begin{proof}
The proof is verbatim as in \cite[Lemma 15.2.1]{BC09} by replacing $K_{p^{\infty}}$ (which is denoted as $K_{\infty}$ in loc. cit.) with $L$ and $G_{K_{p^{\infty}}}$ (which is denoted as $H_K$ in loc. cit.) with $G_L$.
\end{proof}
\begin{remark}\label{two L identification}
As a consequence, from now on we can identify $L_{\dR}^{+}$ with $B_{\dR,L}^+:=\bdr(\Ainf(\mathcal{O}_{\hat{L}}))$.
\end{remark}
Notice that both $L_{\dR}^{+}$ and $B_{\dR}^+$ are topologized discrete valuation rings equipped with the finer topologies (finer than the $t$-adic topology) to make their residue fields acquire the valuation topology rather than the discrete topology as the quotient topology. As a result, finitely generated modules over these rings admits a functorial topological structure, we show this is compatible with inverse limits:
\begin{lemma}\label{topology compatibility}
Let $N$ be a finitely generated module over the topologized discrete valuation ring $\ast \in \{L_{\dR}^{+}, B_{\dR}^+\}$ endowed with natural topology, then the linear continuous bijection $N\rightarrow N/t^iN$ is a homeomorphism. Moreover, the quotient topology on $L_{\dR}^{+}/t^n$ coincides with its subspace topology from $B_{\dR}^+/t^n$ for any $n\in \mathbb{N}$.
\end{lemma}
\begin{proof}
The proof of \cite[Lemma 15.2.2]{BC09} still works here.
\end{proof}

\begin{prop}\label{descend along perfectoid L}
Given $W\in \rb$, the $L_{\dR}^{+}$-module $W^{G_L}$ is finitely generated with a continuous $\hat{G}$-action for its natural topology as a finitely generated $L_{\dR}^{+}$-module, and the natural $B_{\dR}^+$-linear map 
\begin{equation}
    \alpha_W: B_{\dR}^+\otimes_{L_{\dR}^{+}} W^{G_L}\rightarrow W
\end{equation}
is an isomorphism. In particular, the rank and invariant factors of $W^{G_L}$ over $L_{\dR}^{+}$ coincide with those of $W$ over $B_{\dR}^+$.
\end{prop}
\begin{proof}
First we treat the case when $W$ is killed by a power of $t$ by induction on the power of $t$ killing $W$ such that $\alpha_W$ is an isomorphism and $H^1(G_L,W)$ vanishes. 
We start from the case that $tW=0$, i.e. $W\in \Rep_{G_K}(\mathbb{C}_{p})$ as the quotient topology on $\mathbb{C}_{p}\simeq B_{\dR}^+/(\xi)$ induced from that on $B_{\dR}^+$ coincides with its natural topology. Then by Faltings's almost purity theorem
\[\alpha_W: \mathbb{C}_{p}\otimes_{\hat{L}}W^{G_L} \stackrel{\simeq}{\longrightarrow} W.\]
Also, we have that $H^{1}(G_L,W)=0$ by \cite[Theorem 3.5.8]{KL16} and \cite[Theorem 6.5]{Sch13} as $\hat{L}$ is a perfectoid field.

Suppose we have proven that $\alpha_W$ is an isomorphism and $H^{1}(G_L,W)=0$ provided that $W$ is killed by $t^n$, here $n\geq 1$. Now given a $W\in \rb$ which is killed by $t^{n+1}$, we proceed by utilizing the following topologically exact sequence (as $tW$ is equipped with the subspace topology induced from $W$) in $\rb$:
\begin{equation}\label{exact W}
    0\rightarrow tW\rightarrow W\rightarrow W/tW\rightarrow 0.
\end{equation}
Observe that both $tW$ and $W/tW$ are killed by $t^n$, then by induction $H^{1}(G_L,tW)=0$, which implies the exactness of the following $\hat{G}$-equivariant sequence of $L_{\dR}^+$-modules:
\begin{equation*}
    0\rightarrow (tW)^{G_L}\rightarrow W^{G_L}\rightarrow (W/tW)^{G_L}\rightarrow 0.
\end{equation*}
As an upshot, we see that $W^{G_L}$ is finitely generated over $L_{\dR}^+$. This leads to the following $B_{\dR}^+$-linear diagram whose rows are all exact:
\[\xymatrix{0\ar[r] &B_{\dR}^+\otimes_{L_{\dR}^+}(tW)^{G_L}\ar[d]^{\alpha_{tW}}\ar[r]&B_{\dR}^+\otimes_{L_{\dR}^+}W^{G_L}\ar[d]^{\alpha_W}\ar[r]&B_{\dR}^+\otimes_{L_{\dR}^+}(W/tW)^{G_L}\ar[d]^{\alpha_{W/tW}}\ar[r]&0
\\0\ar[r]& tW \ar[r]& W \ar[r] & W/tW \ar[r] &0.
}\]
Here the top row is exact as the scalar extension of injective discrete valuation rings $L_{\dR}^+\rightarrow B_{\dR}^+$ is flat.

Then five lemma implies that $\alpha_W$ is an isomorphism as so are for $\alpha_{tW}$ and $\alpha_{W/tW}$ by induction. Moreover, thanks to \cite[Exercise 2.5.3]{BC09} again, \cref{exact W} induces the exact sequence
\begin{equation*}
    H^{1}(G_L,tW)\rightarrow H^{1}(G_L,W)\rightarrow H^{1}(G_L,W/tW).
\end{equation*}
Hence $H^{1}(G_L,W)=0$ as $H^{1}(G_L,tW)=H^{1}(G_L,W/tW)=0$ by induction.

Now we have finished the proof for those $W$ killed by a power of $t$. One of the byproduct is that the functor $W\mapsto W^{G_L}$ is exact when restricted to torsion $W\in \rb$ as $H^1(G_L, \cdot)$ vanishes on such $W$.

For non-torsion $W\in \rb$, $W=\varprojlim W/t^m$. For a fixed $m$, consider the $G_K$-equivariant right exact sequence 
\[W/t^m \stackrel{t^m}{\rightarrow} W/t^m\rightarrow W/t^n\rightarrow 0.\]
Thanks to the vanishing of $H^1$ on torsion objects, it is still right exact after taking $G_L$-invariants
\[(W/t^m)^{G_L} \stackrel{t^n}{\rightarrow} (W/t^m)^{G_L}\rightarrow (W/t^n)^{G_L}\rightarrow 0.\]
Then passing to inverse limit with respect to $m$ preserves right exactness as it forms a Mittag-Leffler system (we have shown that all the torsion terms involved are finitely generated):
\[W^{G_L}\stackrel{t^n}{\rightarrow} W^{G_L}\rightarrow (W/t^n)^{G_L}\rightarrow 0.\]
This implies that $W^{G_L}/t^n\cong (W/t^n)^{G_L}$ for all $n\geq 1$.

In particular, take $n=1$, we have that $W^{G_L}/t\cong (W/t)^{G_L}$ is a finitely generated $\hat{L}$-vector space based on our study of $t$-torsion case. But $W^{G_L}$ is a closed $L_{\dR}^+$-submodule of a finitely generated $B_{\dR}^+$-submodule$W$, hence $W^{G_L}$ is $t$-adically complete and separated. Then Nakayama's lemma implies that $W^{G_L}$ is finitely generated over $L_{\dR}^+$. 

Finally consider the natural map 
\[\alpha_W: B_{\dR}^+\otimes_{L_{\dR}^{+}} W^{G_L}\rightarrow W.\]
Both sides are finitely generated over $B_{\dR}^+$, hence $t$-adically complete and separated, by derived Nakayama lemma, to show $\alpha_W$ is an isomorphism, it suffices to show $\alpha_W/t$ (derived sense)  is an isomorphism. As $W$ is non-torsion, $\alpha_W/t$ is the usual quotient, then the result follows as $W^{G_L}/t\cong (W/t)^{G_L}$ and we have already shown $\alpha_{W/t}$ is an isomorphism.
\end{proof}
\begin{cor}\label{de rham almost purity}
The functor 
\begin{equation*}
    \begin{split}
        \rb &\longrightarrow \Rep_{\hat{G}}(B_{\dR,L}^+)
        \\W&\longmapsto W^{G_L}
    \end{split}
\end{equation*}
is well defined and induces an equivalence of categories. A quasi-inverse can be chosen as $X\mapsto B_{\dR}^+\otimes_{B_{\dR,L}^+}X$.
\end{cor}
\begin{proof}
Let us first explain why it is well defined. By \cref{two L identification} and \cref{descend along perfectoid L}, $W^{G_L}$ is a finitely generated $B_{\dR,L}^+$-module and 
\begin{equation}
    \alpha_W: B_{\dR}^+\otimes_{B_{\dR,L}^+} W^{G_L}\rightarrow W
\end{equation}
is an isomorphism. Passing to the case of cyclic modules using \cref{topology compatibility} and then apply \cref{compatibility I} as well, the natural continuous injective map $M\rightarrow B_{\dR}^+\otimes_{B_{\dR,L}^+}M$ is a homeomorphism onto its image for any finitely generated $B_{\dR,L}^+$-module $M$. Apply it to $M=W^{G_L}$, we see that the natural topology on $W^{G_L}$ as a finitely generated $B_{\dR,L}^+$-module coincides with the subspace topology induced from $W$ via $\alpha_W$. Then the continuity of the $\hat{G}$-action on $W^{G_L}$ follows from the continuity of the $G_K$-action on $W$, hence $W^{G_L}\in \Rep_{\hat{G}}(B_{\dR,L}^+)$. 

On the other hand, given $X\in \Rep_{\hat{G}}(B_{\dR,L}^+)$, then $V:=B_{\dR}^+\otimes_{B_{\dR,L}^+}X$ is equipped with a continuous $G_K$-action with its natural topology as a finitely generated $B_{\dR}^+$-module as $G_K$ acts continuously on both $B_{\dR}^+$ and $X$. As a result, $V\in\rb$. Then one can check $V^{G_L}=X$ and $V(W^{G_L})\cong W$ via $\alpha_W$ for $W\in \rb$. Hence the functors are quasi-inverses.
\end{proof}

\subsection{Sen theory over $B_{\dR,L}^+$: decompletion along Kummer tower}
\begin{notation}\label{notation for t}
Recall that we embed the Breuil-Kisin prism $(W(k)[[u]], E(u)) $ into $(\Ainf_L, \Ker(\theta))$ by sending $u$ to $[\pi^{\flat}]$. Hence the induced embedding $\bdr(\mathfrak{S})\xhookrightarrow{} B_{\dR,L}^+$ sends $t=E(u)$ to $E([\pi^{\flat}])$, which will be our chosen uniformizer in the complete discrete ring $B_{\dR,L}^+$. Define $B_{\dR,k,L}^+$ to be $B_{\dR,L}^+/t^k$, which is just $\Ainf(\mathcal{O}_{\hat{L}})[\frac{1}{p}]/t^k)$. And we use $\theta_k$ to denote the specialzation map $B_{\dR,L}^+\rightarrow B_{\dR,k,L}^+$. We will abuse notation by also writing $\theta_k$ for $\Ainf(\mathcal{O}_{\hat{L}})[\frac{1}{p}]\rightarrow \Ainf(\mathcal{O}_{\hat{L}})[\frac{1}{p}]/(t^k)$. 
When $k=1$, we will denote $\theta_k$ as $\theta$.
\end{notation}
    We give a quick review of locally analytic vectors. We refer the readers to \cite[Section 2]{BC16} \cite[Section 2]{Ber16} for details.
Recall given a $p$-adic Lie group $G$, and let $V$ be a $\mathbb{Q}_p$-Banach representation of $G$ with norm $\|\cdot \|$.
Let $H$ be an open subgroup of $G$ such that there exist coordinates $c_1,\ldots,c_d : H \to \mathbb{Z}_p$ giving rise to an analytic bijection $\mathbf{c} : H \to \mathbb{Z}_p^d$. Then 
\begin{defi}\label{de rham analytic}
\begin{itemize}
    \item An element $w \in V$ is \emph{$H$-analytic} if there exists a sequence $\{w_k\}_{k \in \mathbb{N}^d}$ with $w_k \to 0$ in $W$, such that
    \begin{equation*}\label{local analytic action}
        g(w) = \sum_{k \in \mathbb{N}^d} \mathbf{c}(g)^k w_k, \quad \forall g \in H.
    \end{equation*}
    The space of such vectors is denoted as $V^{H-an}$.
    \item A vector $w \in V$ is \emph{locally analytic} if there exists an open subgroup $H$ as above such that $w \in V^{H-an}$. The space of such vectors is denoted as $V^{G-la}$.
\end{itemize}
\end{defi}
\begin{defi}[Monodromy operators]\label{monodromy operators}
Still assume $G=\hat{G}$, $V$ a $\mathbb{Q}_p$-Banach representation of $\hat{G}$, one can define two monodromy operators $\nabla_{\gamma}$ and $\nabla_{\tau}$ on $V^{\hat{G}-la}$ as follows:
\begin{itemize}
    \item $\nabla_{\gamma}:=\frac{\log g}{\log\chi(g)}$ for $g\in \Gal(L/K_{\infty})$ enough close to $1$, where as usual $\chi$ is the $p$-adic cyclotomic character and $\log g$ is just the formal power series $\sum_{n=1}^{\infty} \frac{(-1)^{n-1}(g-1)^n}{n}$.
    \item $\nabla_{\tau}:=\frac{\log (\tau^{p^n})}{p^n}$ for $n$ large enough.
\end{itemize}
\end{defi}
\begin{remark}
A useful result is that for $x\in V^{\hat{G}-la}$,
\begin{equation*}
    \nabla_{\tau}(x)=\lim_{n\rightarrow \infty}\frac{\tau^{p^n}(x)-x}{p^n},
\end{equation*}
as on $V^{\hat{G}-la}$, $\tau=\sum_{i=0}^{\infty}\frac{\nabla_{\tau}^i}{i!}$.
\end{remark}
\begin{defi}\label{de rham analytic vectors}
Suppose $G=\hat{G}, W\in \Rep_{B_{\dR,L}^+}^{}(\hat{G})$, define
\begin{equation*}
    \begin{split}
        W^{\hat{G}-la}=\varprojlim (W/t^nW)^{{\hat{G}}-la},\quad W^{\hat{G}-la,\gamma=1}=\varprojlim (W/t^nW)^{{\hat{G}}-la,\gamma=1},
        \\ W^{\hat{G}-la,\tau=1}=\varprojlim (W/t^nW)^{{\hat{G}}-la,\tau=1}, \quad W^{\hat{G}-la,\nabla_{\gamma}=0}=\varprojlim (W/t^nW)^{{\hat{G}}-la,\nabla_{\gamma}=0},
    \end{split}
\end{equation*}
    and equip them with subspace topology induced from that on $W=\varprojlim X/t^m W$. Here for $V$ a $\mathbb{Q}_p$-Banach representation of $\hat{G}$, $$V^{{\hat{G}}-la,\gamma=1}:=V^{{\hat{G}}-la,\Gal(L/K_{\infty})=1}, \quad V^{{\hat{G}}-la,\tau=1}:=V^{{\hat{G}}-la,\Gal(L/K_{p^\infty})=1}.$$
\end{defi}
\vspace{0.2cm}
Also, given $V$ a $\mathbb{Q}_p$-Banach representation of $G$ with norm $\|\cdot \|$. then we define $V\{\{T\}\}_n$ to be the vector space of series $\sum_{k \geqslant 0} v_{k} T^{k}$ such that $p^{nk}v_k\rightarrow 0$ as $k$ tends to $\infty$, here $T$ is a variable. For $h\in V$ such that $\|v \|\leq p^{-n}$, $V\{\{h\}\}_n$ is defined to be the evaluation of $V\{\{T\}\}_n$ at $T=h$.

The story starts from the $t$-torsion case.

Let $\beta_n\in L$ such that  $\|\theta(\fkt)- \beta_n\|\leq p^{-n}$, then as in \cite[Section 4.2]{BC16} and \cite[Construction 3.2.5]{Gao22}, there is an increasing sequence $\{r_n\}$ such that if $m\geq \{r_n\}$, then $\|\theta(\fkt)- \beta_n\|_{\hat{G}_m}=\|\theta(\fkt)- \beta_n\|$ and $\theta(\fkt)- \beta_n\in \hat{L}^{\hat{G}_m-an}$, where $\hat{G}_m$ is topologically generated by $\tau^n$ and $\gamma^n$, $\|w\|_{\hat{G}_m}=\sup _{k \in \mathbb{N}^{d}}\left\|w_{k}\right\|$ ($w_{k}$ is defined as in \cref{local analytic action}).

Then the following result is essentially calculated by Gao and Poyeton:
\begin{lemma}\label{torsion analytic vector}
\begin{itemize}
    \item $\hat{L}^{\hat{G}-la} =\cup_{n \geq 1} K({\mu_{r(n)}, \pi_{r(n)}})\{\{ \theta(\fkt)-\beta_n \}\}_n.$
\item $\hat{L}^{\hat{G}-la, \nabla_\gamma=0} = L.$
\item $\hat{L}^{\tau-la, \gamma=1} = {K_{\infty}}.$ 
\end{itemize}
Here $\fkt=\frac{\log([\varepsilon])}{p\lambda}$.
\end{lemma}
\begin{proof}
The last two is \cite[Proposition 3.3.2]{GP21} while the first identification is \cite[Proposition 3.2.8]{Gao22}.
\end{proof}
\begin{remark}
By \cite[Lemma 2.5]{BC16}, $\hat{L}^{\hat{G}-la}$ is a field. It strictly contains $L$.
\end{remark}

\begin{theorem}\label{calculation de rham analytic}
Suppose $k\geq 1$, then the map $$\bigoplus_{i=0}^{k-1}(\cup_{n \geq 1} K\left(\mu_{r(n)}, \pi_{r(n)}\right)\left\{\left\{\theta_k(\fkt)-\beta_n \right\}\right\}_{n})\cdot t^i\rightarrow (B_{\dR,k,L}^+)^{\hat{G}-la}$$ is an isomorphism.
\end{theorem}
\begin{proof}
First notice that for a fixed $k$, $t\in (B_{\dR,k,L}^+)^{\hat{G}-la}$. Actually, when $k=1$, $t$ is $0$ in $B_{\dR,1,L}^+\cong \hat{L}$. To treat the general $k\geq 2$, it suffices to show $$t\in (B_{\dR,k,L}^+)^{\tau-la,\gamma=1}$$
thanks to \cite[Lemma 3.2.4]{GP21}. By \cite[Lemma 2.5(i)]{BC16}, $(B_{\dR,k,L}^+)^{\tau-la,\gamma=1}$ is a ring, hence it suffices to show $[\pi^\flat]\in (B_{\dR,k,L}^+)^{\tau-la,\gamma=1}$.

However, $\Gal(L/K_{\infty})$ acts trivially on $[\pi^\flat]$ by the definition of $[\pi^\flat]$ and $\tau$ acts on $[\pi^\flat]$ via
\begin{equation*}
    \begin{split}
        \tau ([\pi^\flat])=[\varepsilon][\pi^\flat]=([\varepsilon]-1)[\pi^\flat]
    \end{split}
\end{equation*}
then $(\tau-1)([\pi^\flat])=([\varepsilon]-1)[\pi^\flat] \in t^1B_{\dR,k,L}^+$. Then by induction we see that $(\tau-1)^{i}([\pi^\flat])=([\varepsilon]-1)^i[\pi^\flat] \in t^iB_{\dR,k,L}^+$. As a consequence,$$(\log \tau)^i([\pi^\flat])\in t^iB_{\dR,k,L}^+,$$
hence $[\pi^\flat]\in (B_{\dR,k,L}^+)^{\tau-la}$ thanks to \cite[Lemma 3.1.7]{GP21}. As a result, $[\pi^\flat]\in (B_{\dR,k,L}^+)^{\tau-la,\gamma=1}$.

Moreover, this implies $t^j\in (B_{\dR,k,L}^+)^{\tau-la,\gamma=1}$ by \cite[Lemma 2.5(i)]{BC16} again.

Now we proceed to prove the result by induction on $k$. When $k=1$, this is \cref{torsion analytic vector}.
Suppose the lemma is proven for up to $k-1$, then given $y\in (B_{\dR,k,L}^+)^{\hat{G}-la}$, $\theta(y)\in \hat{L}^{\hat{G}-la}$, hence there exists $n$ large enough, such that $\theta(y)=f(\theta(\fkt)-\beta_n)$, where $f(T)\in K({\mu_{r(n)}, \pi_{r(n)}})\{\{ T \}\}_n$. Thanks to \cref{the image of t}, $\theta_k(\mathfrak{t})\in (B_{\dR,k,L}^+)^{\hat{G}-la}$. 
On the other hand, by \cref{check convergence after modulo}, $f(\theta_k(\fkt)-\beta_n)$ converges in $B_{\dR,k,L}^+$, hence $f(\theta_k(\fkt)-\beta_n)\in (B_{\dR,k,L}^+)^{\hat{G}-la}$. Also, $y-f(\theta_k(\fkt)-\beta_n)\in t\cdot B_{\dR,k,L}^+$, hence there exists $z\in B_{\dR,k-1,L}^+$, such that $y-f(\theta_k(\fkt)-\beta_n)=tz$.
Moreover, $z\in (B_{\dR,k-1,L}^+)^{\hat{G}-la}$ as we have shown $t$ is locally analytic. Then the desired result for $k$ follows from induction.
\end{proof}
The following lemma is used in the proof:
\begin{lemma}\label{check convergence after modulo}
Let $E$ be a finite extension of $\mathbb{Q}_p$ and $f(T)=\sum_{k\geq 0}a_kT^k\in E[[T]]$. Suppose $x\in B_{\dR,k,L}^+$ for some $k\geq 1$, then $f(x)$ converges in $B_{\dR,k,L}^+$ if and only if $f(\theta(x))$ converges in $\mathbb{C}_p$.
\end{lemma}
\begin{proof}
This is \cite[Lemma 4.9]{BC16}, we quickly review its proof here. Recall that $B_{\dR,k,L}^+$ is a $\mathbb{Q}_p$-Banach space with unit ball $\Ainf(\mathcal{O}_{\hat{L}})$. By enlarging $E$ if necessary we can assume that $E$ contains an element whose valuation is precisely that of $\theta(x)$, hence it suffices to show if $\theta(x)\in \mathcal{O}_{\hat{L}}$, then $\{x^n\}$ is bounded in $B_{\dR,k,L}^+$. Pick $x_0\in \Ainf(\mathcal{O}_{\hat{L}})$ such that $\theta(x_0)=\theta(x)$, then $x=x_0+ty+t^kz$ for some $y\in \Ainf(\mathcal{O}_{\hat{L}})[\frac{1}{p}]$, $z\in B_{\dR,L}^+$. This implies
\begin{equation*}
    x^n=x_0^n+\binom{n}{1} x_{0}^{n-1}\cdot ty+\cdots+ \binom{n}{k-1} x_{0}^{n-(k-1)}\cdot (ty)^{k-1}+t^kz_k
\end{equation*}
for some $z_k\in B_{\dR,L}^+$. Hence $x^n\in (\Ainf(\mathcal{O}_{\hat{L}})+y\Ainf(\mathcal{O}_{\hat{L}})+\cdots+y^{k-1}\Ainf(\mathcal{O}_{\hat{L}}))+t^kB_{\dR,L}^+$ is bounded for any $n\geq 0$.
\end{proof}

\begin{remark}
In the cyclotomic case, a similar result is \cite[Theorem 4.11]{BC16}. In our case, we use $t$ instead of $\log([\varepsilon])$ in the expression as this makes the calculation of "$\gamma$-invariant" elements quite transparent:
\end{remark}
\begin{cor}\label{calculate de rham analytic II}
The following are isomorphisms:
\begin{equation*}
    \begin{split}
        \bigoplus_{i=0}^{k-1}K_{\infty}\cdot t^i&\stackrel{\cong}{\longrightarrow} (B_{\dR,k,L}^+)^{\hat{G}-la,\gamma=1},
        \\\bigoplus_{i=0}^{k-1}L\cdot t^i&\stackrel{\cong}{\longrightarrow} (B_{\dR,k,L}^+)^{\hat{G}-la,\nabla_{\gamma}=0}.
    \end{split}
\end{equation*}
\end{cor}
\begin{proof}
We can argue by induction on $k$ similarly as in \cref{calculation de rham analytic} using \cref{torsion analytic vector}.
\end{proof}
By taking inverse limits, we have that:
\begin{cor}\label{de rham analytic calculation}
\begin{equation*}
    \begin{split}
     &(B_{\dR,L}^+)^{\hat{G}-la}=(\cup_{n \geq 1} K\left(\mu_{r(n)}, \pi_{r(n)}\right)\left\{\left\{\fkt-\beta_n\right\}\right\}_{n})[[t]],
       \\ &(B_{\dR,L}^+)^{\hat{G}-la,\nabla_{\gamma}=0}= L[[t]],
        \\&(B_{\dR,L}^+)^{\hat{G}-la,\gamma=1}= K_{\infty}[[t]].
    \end{split}
\end{equation*}
\end{cor}

\begin{cor}\label{faithfully flat}
The map $K_{\infty}[[t]]\rightarrow (B_{\dR,L}^+)^{\hat{G}-la}$ is faithfully flat, here we equip $(B_{\dR,L}^+)^{\hat{G}-la}$ with a $K_{\infty}[[t]]$-module structure by sending $t$ to $E([\pi^{\flat}])$ as usual. 
\end{cor}
\begin{proof}
 Thanks to \cref{de rham analytic calculation}, the flatness follows using \cite[\href{https://stacks.math.columbia.edu/tag/00MK}{Tag 00MK}]{SP22}, which implies faithful flatness by \cite[\href{https://stacks.math.columbia.edu/tag/00HR}{Tag 00HR}]{SP22}.
\end{proof}
\begin{remark}
Given $X\in \db$, if $X$ is killed by $t^m$, then $X^{\hat{G}-la,\gamma=1}$ is a $(B_{\dR,m,L}^+)^{\hat{G}-la,\gamma=1}$-module, i.e. a $K_{\infty}[t]/t^m$-module thanks to \cref{calculate de rham analytic II}. Passing to inverse limit, we have that for general $X$, $X^{\hat{G}-la,\gamma=1}$ is a $K_{\infty}[[t]]$-module.
\end{remark}

\begin{theorem}\label{decend along locally analytic}
    For any $W\in \db$, $W^{\hat{G}-la}\in \Rep_{\hat{G}}((B_{\dR,L}^+)^{\hat{G}-la})$. Moreover, the natural map 
    \[\beta_W: B_{\dR,L}^+\otimes_{(B_{\dR,L}^+)^{\hat{G}-la}}W^{\hat{G}-la} \longrightarrow W\]
is an isomorphism.
\end{theorem}
\begin{proof}
As the map $(B_{\dR,L}^+)^{\hat{G}-la}\rightarrow (B_{\dR,L}^+)$ is faithfully flat, to show $W^{\hat{G}-la}$ is finitely generated over $(B_{\dR,L}^+)^{\hat{G}-la}$, it suffices to check this after scalar extension to $B_{\dR,L}^+$. Hence it reduces to show that $\beta_W$ is an isomorphism. For this, we notice that by \cite[Prop. 3.1.6]{GP21}
\begin{equation*}
    W^{\hat{G}-la}=(W^{\Gal(L/K_{p^\infty})})^{\Gamma_{K}-la}\otimes_{(B_{\dR,K_{p^\infty}}^+)^{\Gamma_{K}-la}} (B_{\dR,L}^+)^{\hat{G}-la}.
\end{equation*}
On the other hand, 
\begin{equation*}
    W=W^{\Gal(L/K_{p^\infty})}\otimes_{B_{\dR,K_{p^\infty}}^+} B_{\dR,L}^+.
\end{equation*}
Hence it suffices to show that 
\begin{equation*}
    (W^{\Gal(L/K_{p^\infty})})^{\Gamma_{K}-la}\otimes_{(B_{\dR,K_{p^\infty}}^+)^{\Gamma_{K}-la}} B_{\dR,K_{p^\infty}}^+ =W^{\Gal(L/K_{p^\infty})}.
\end{equation*}
But as $\Gamma_K$ is a $p$-adic Lie group of dimension $1$, $\Gamma_K$-locally analytic vectors coincide with classical $K$-finite vectors thanks to \cite[Thm. 3.2]{BC16}, hence 
\[(W^{\Gal(L/K_{p^\infty})})^{\Gamma_{K}-la}=(W^{\Gal(L/K_{p^\infty})})_f,\]
where the latter is defined in \cite[Section 3.3]{Fon04}, hence the desired result follows from \cite[Thm.3.6]{Fon04}.
\end{proof}
\begin{remark}\label{remark on mod n}
As a consequence, one can see that for $W\in \db$, $W^{\hat{G}-la}/t^n\cong (W/t^n)^{\hat{G}-la}$.
\end{remark}

Now we are ready to state our main theorem in this section, which should be viewed as a lifting of Sen's decompletion theory for Kummer tower developed by Gao in \cite{Gao22}. 
\begin{theorem}\label{Main theorem III}
    For any $W\in \db$,  $W^{\hat{G}-la,\gamma=1}$ is a finitely generated $K_{\infty}[[t]]$-module. Moreover, the natural map 
\[\alpha_W: (B_{\dR,L}^+)^{\hat{G}-la}\otimes_{K_{\infty}[[t]]}W^{\hat{G}-la,\gamma=1} \longrightarrow W^{\hat{G}-la}\]
is an isomorphism.
\end{theorem}

In the following, for $W\in \db$, we write $D_{\mathrm{Sen},K_{\infty}[[t]]}(W):=W^{\hat{G}-la,\gamma=1}$.

Our strategy is to treat the torsion case first, imitating Gao's proof in \cite[Proposition 3.2.7]{Gao22} for the $t$-torsion case. In particular, this should be viewed as a higher level $\theta_k$-specialization of \cite[Remark 6.1.7]{GP21}. Hence we need several preliminaries.

\begin{lemma}\label{the image of t}
For $k\geq 1$, under the specialization map $\theta_k: \Ainf(\mathcal{O}_{\hat{L}})[\frac{1}{p}]\rightarrow \Ainf(\mathcal{O}_{\hat{L}})[\frac{1}{p}]/(\xi^k)$, $\theta_k(\mathfrak{t})\in (B_{\dR,k,L}^+)^{\hat{G}-la}$ and $\theta_k(\mathfrak{t})^{-1}\in (B_{\dR,k,L}^+)^{\hat{G}-la}$, here $\mathfrak{t}$ is defined in \cref{torsion analytic vector}.
\end{lemma}
\begin{proof}
$\theta_k(\mathfrak{t})$ is nonzero as $\theta(\mathfrak{t})\neq 0$ by \cite[Lemma 3.2.1]{Gao22}.

To prove analyticity of $\theta_k(\mathfrak{t})$, we start from \cite[Lemma 5.1.1]{GP21}, which tells us that there exists $n=n(\fkt)$ such that $\fkt\in (\Ainf(\mathcal{O}_{\hat{L}})\langle\frac{[\pi^\flat]^{ep^n}}{p},\frac{p}{[\pi^\flat]^{ep^n}}\rangle[\frac{1}{p}])^{\hat{G}-la}$, hence implies the analyticity of $\fkt$ under 
\[\Ainf(\mathcal{O}_{\hat{L}})\langle\frac{[\pi^\flat]^{ep^n}}{p},\frac{p}{[\pi^\flat]^{ep^n}}\rangle[\frac{1}{p}]\stackrel{\varphi^{-n}}{\longrightarrow}B_{\dR,L}^+\stackrel{\theta_k}{\longrightarrow} B_{\dR,k,L}^+.\]
In other words, $\theta_k(\varphi^{-n}(\fkt))\in (B_{\dR,k,L}^+)^{\hat{G}-la}$. On the other hand, recall $\mathfrak{t}$ satisfies that $\varphi(\mathfrak{t})=\frac{pE(u)}{E(0)}\mathfrak{t}$, hence 
\[\fkt = \varphi^{-n}(\fkt) \cdot \prod_{i=1}^n \varphi^{-i}(\frac{pE(u)}{E(0)}),\]
so it suffices to check $\theta_{k}(\varphi^{-i}(\frac{pE(u)}{E(0)}))$ is locally analytic in $B_{\dR,k,L}^+$. For this, we use the same trick as in the proof of \cref{calculation de rham analytic}. Let $x=\varphi^{-i}(u)$, then $x$ is $\Gal_{L/K_{\infty}}$-invariant, and 
\[\tau^i(x)=[\varepsilon]x,\]
from which we can deduce that $(\log (\tau^i))^j(x)\in t^j B_{\dR,k,L}^+$ by induction on $j$. Then notice that $\tau^i$ is a topological generator of $\Gal_{L/K_{p^{\infty}}}$, we have that $x\in (B_{\dR,k,L}^+)^{\tau-la}$, which further implies that $x\in (B_{\dR,k,L}^+)^{\hat{G}-la}$.
Now we conclude that $\theta_k(\mathfrak{t})\in (B_{\dR,k,L}^+)^{\hat{G}-la}$. The analyticity of $\theta_k(\frac{1}{\mathfrak{t}})$ can be proven similarly.
\end{proof}

\begin{proof}[Proof of \cref{Main theorem III}]

Consider the torsion case first, i.e. suppose $W\in \db$ is killed by $t^k$ for some $k$. We proceed in two steps:
\begin{itemize}
    \item First we show $W^{\hat{G}-la,\nabla_{\gamma}=0}$ satisfies that
    \begin{equation}\label{step I isomorphism}
        W^{\hat{G}-la,\nabla_{\gamma}=0}\otimes_{(B_{\dR,k,L}^+)^{\hat{G}-la,\nabla_{\gamma}=0}} (B_{\dR,k,L}^+)^{\hat{G}-la}\stackrel{\cong}{\longrightarrow}W^{\hat{G}-la}.
    \end{equation}
    \item Secondly we show
    \begin{equation*}
        W^{\hat{G}-la,\gamma=1}\otimes_{(B_{\dR,k,L}^+)^{\hat{G}-la,\gamma=1}} (B_{\dR,k,L}^+)^{\hat{G}-la,\nabla_{\gamma}=0}\stackrel{\cong}{\longrightarrow} W^{\hat{G}-la,\nabla_{\gamma}=0}.
    \end{equation*}
\end{itemize}
Combine these two identities together we see $\alpha_W$ is an isomorphism (for $W$ killed by a power of $t$).

\textbf{Step 1} Define $\partial_{\gamma}$ to be the $\theta_k$-specialization of that in \cite[5.3.4]{GP21}, i.e. $\partial_{\gamma}=\frac{1}{\theta_k(\fkt)}\nabla_{\gamma}$. Then $\partial_{\gamma}(\theta_k(\fkt))=1$ as $\nabla_{\gamma}(\fkt)=\fkt$. Fix a minimal generating set of $\{x_1,\cdots,x_l\}$, on which $\partial_{\gamma}$ acts via $D_{\gamma}\in \mat((B_{\dR,k,L}^+)^{\hat{G}-la})$. Let $\underline e:=(x_1,\cdots,x_l)$. It suffices to show surjectivity of $M^{\nabla_{\gamma}=0} \otimes_{(B_{\dR,k,L}^+)^{\hat{G}-la,\nabla_{\gamma}=0}} (B_{\dR,k,L}^+)^{\hat{G}-la}\longrightarrow M$, which can be reduced to finding a matrix $H\in \GL_{l}((B_{\dR,k,L}^+)^{\hat{G}-la})$ such that 
\[\partial_{\gamma}(H)+D_{\gamma} H=0.\]
We claim that $H=\sum_{s\geq 0}(-1)^{n} D_{s} \frac{(\theta_k(\fkt)-\beta_{n})^{k}}{k !}$ works for some $n$ large enough, here $D_s$ is defined to be the matrix via which $\partial_{\gamma}^s$ acts on $\underline e$.
Actually,
\begin{equation*}
    \begin{split}
        \partial_{\gamma}(\underline e H)&=\partial_{\gamma}(\underline e)H+\underline e\partial_{\gamma}(H)
        \\&=\underline e \sum_{s\geq 0}(-1)^{s} D_{s+1} \frac{(\theta_k(\fkt)-\beta_n)^{s}}{s !}+\underline e\sum_{s\geq 1}(-1)^{s} D_{s} \frac{(\theta_k(\fkt)-\beta_n)^{s-1}}{(s-1) !}
        \\&=0.
    \end{split}
\end{equation*}
Also, for large enough $n$, $H$ converges as this can be checked after modulo $\theta$ thanks to \cref{check convergence after modulo}.

\textbf{Step 2} By step 1, $\underline e H$ forms a minimal generating set $\{y_1,\cdots,y_l\}$ for $(B_{\dR,k,L}^+)^{\hat{G}-la,\nabla_{\gamma}=0}$-module $W^{\hat{G}-la,\nabla_{\gamma}=0}$. Suppose $t^{m_l}$ is the proper annilator of $y_l$. Then if we view $W^{\hat{G}-la,\nabla_{\gamma}=0}$ as a $L$-vector space thanks to \cref{calculate de rham analytic II}, then $\{t^{n_l}y_l\}_{0\leq n_l \leq m_l-1}$ forms a $L$-basis. By \'etale descent as shown in \cite[Thm 3.3.1]{Gao22} , we have that $W^{\hat{G}-la,\nabla_{\gamma}=0,\gamma=1}$ is a $K_{\infty}$-space of dimension $\sum{m_l}$ and that $W^{\hat{G}-la,\nabla_{\gamma}=0,\gamma=1}\otimes_{K_{\infty}} L=W^{\hat{G}-la,\nabla_{\gamma}=0}$. Clearly $W^{\hat{G}-la,\nabla_{\gamma}=0,\gamma=1}=W^{\hat{G}-la,\gamma=1}$ is a $(B_{\dR,k,L}^+)^{\hat{G}-la,\gamma=1}$-module as $t$ is fixed by $\Gal({L/K_{\infty}})$. Thanks to \cref{calculate de rham analytic II} again, this implies that 
\begin{equation}\label{step II isomorphism}
    W^{\hat{G}-la,\gamma=1}\otimes_{(B_{\dR,k,L}^+)^{\hat{G}-la,\gamma=1}} (B_{\dR,k,L}^+)^{\hat{G}-la,\nabla_{\gamma}=0}\stackrel{\cong}{\longrightarrow} W^{\hat{G}-la,\nabla_{\gamma}=0}.
\end{equation}

\textbf{Step 3} For general $W$ (not necessarily $t$-power torsion), recall that by definition, 
\begin{equation*}
    W^{\hat{G}-la,\gamma=1}=\varprojlim (W/t^nW)^{\hat{G}-la,\gamma=1}.
\end{equation*}
Moreover, by our proof in step 1 and step 2, we see that each $(W/t^nW)^{\hat{G}-la,\gamma=1}$ is a finitely generated $K_{\infty}[[t]]/t^n$-module such that 
$$(W/t^{n+1}W)^{\hat{G}-la,\gamma=1}/t^{n}\stackrel{\cong}{\longrightarrow} (W/t^nW)^{\hat{G}-la,\gamma=1}.$$ 
Hence $W^{\hat{G}-la,\gamma=1}$ is a finitely generated $K_{\infty}[[t]]$-module and that $W^{\hat{G}-la,\gamma=1}/t^{n}\cong (W/t^nW)^{\hat{G}-la,\gamma=1}$ by \cite[\href{https://stacks.math.columbia.edu/tag/09B8}{Tag 09B8}]{SP22}. Togethor with \cref{remark on mod n}, we see that $\alpha_W$ mod $t^n$ is precisely $\alpha_{W/t^n}$, hence an isomorphism by combining \cref{step I isomorphism} and \cref{step II isomorphism}. Iterating all $n$ we conclude that $\alpha_W$ is an isomorphism.
\end{proof}
\begin{theorem}\label{combine three theorems}
Given $W\in \rb$, by abuse of notation we still define
\[D_{\mathrm{Sen},K_{\infty}[[t]]}(W):=(W^{G_L})^{\tau-la,\gamma=1}.\]
Then this is a finitely generated  $K_{\infty}[[t]]$-module such that 
\[D_{\mathrm{Sen},K_{\infty}[[t]]}(W)\otimes_{K_{\infty}[[t]]}B_{\dR}^+=W.\]
\end{theorem}
\begin{proof}
This follows from combining \cref{descend along perfectoid L}, \cref{decend along locally analytic} and \cref{Main theorem III}.
\end{proof}
\subsection{Kummer-Sen operator}
Given $W\in \rb$, We can define a monodromy operator on $D_{\mathrm{Sen},K_{\infty}[[t]]}(W)$. Actually, for each $k\in \mathbb{N}$, since vectors in $(W^{G_L}/t^k)^{\hat{G}-la,\gamma=1}=(W^{G_L}/t^k)^{\tau-la,\gamma=1}$ by \cite[Lemma 3.2.4]{GP21} are all $\tau$-locally analytic, hence we can define 
\begin{equation*}
    N_{\nabla,k}:=\frac{1}{p\theta_k(\mathfrak{t})}\nabla_{\tau}: (W^{G_L}/t^k)^{\tau-la,\gamma=1}\longrightarrow (W^{G_L}/t^k)^{\hat{G}-la}.
\end{equation*}
Clearly $N_{\nabla,m}$ are compatible with $N_{\nabla,n}$, hence by passing to inverse limits, we get
\begin{equation}\label{Kummer Sen operator equation}
    N_{\nabla}=\frac{1}{p\mathfrak{t}}\nabla_{\tau}: (W^{G_L})^{\tau-la,\gamma=1}\longrightarrow (W^{G_L})^{\hat{G}-la}.
\end{equation}
\begin{theorem}\label{Kummer Sen operator}
    \cref{Kummer Sen operator equation} induces a $K_{\infty}$-linear operator
    \begin{equation*}
        N_{\nabla}: D_{\mathrm{Sen},K_{\infty}[[t]]}(W)\longrightarrow  D_{\mathrm{Sen},K_{\infty}[[t]]}(W)
    \end{equation*}
such that $N_{\nabla}$ satisfies Leibniz rule and that
\begin{equation*}
    N_{\nabla}(tv)=N_{\nabla}(t)v+tN_{\nabla}(v)=E^{\prime}(u)\lambda u \cdot v+tN_{\nabla}(v).
\end{equation*}
We will call $N_{\nabla}$ the Kummer Sen operator.
\end{theorem}
\begin{proof}
Clearly $N_{\nabla}$ is $K_{\infty}$-linear as $\nabla_{\tau}$ vanishes on $K_{\infty}$, hence it suffices to show that $N_{\nabla}$ is a operator from $D_{\mathrm{Sen},K_{\infty}[[t]]}(W)$ to itself, i.e. $N_{\nabla}$ preserves $\Gal(L/K_{\infty})$-invariant elements. For this, notice that $\gamma \tau \gamma^{-1}=\tau^{\chi(\gamma)} $, hence $\gamma \nabla_{\tau}=\chi(\gamma) \nabla_{\tau} \gamma$. Also, as $\gamma(\mathfrak{t})=\chi(\gamma)\mathfrak{t}$, we conclude that $$\gamma N_{\nabla}=\frac{1}{p\gamma(\mathfrak{t})} \gamma \nabla_{\tau}=\frac{1}{p\chi(\gamma)\mathfrak{t}} \chi(\gamma) \nabla_{\tau} \gamma=N_{\nabla}\gamma. $$
This implies that $N_{\nabla}$ preserves $D_{\mathrm{Sen},K_{\infty}[[t]]}(W)$.
To show that $N_{\nabla}$ satisfies Leibniz rule, it suffices to check $N_{\nabla}(tv)=N_{\nabla}(t)v+tN_{\nabla}(v)$. But by definition,
\begin{equation*}
\begin{split}
    N_{\nabla}(tv)&=\lim_{n\rightarrow \infty} \frac{1}{p\mathfrak{t}}\frac{\tau^{p^n}(tv)-(tv)}{p^n}
    \\&=\lim_{n\rightarrow \infty} \frac{1}{p\mathfrak{t}} \frac{(\tau^{p^n}(t)-t)v+\tau^{p^n}(t)(\tau^{p^n}(v)-v)}{p^n}
    \\&=N_{\nabla}(t)v+\lim_{n\rightarrow \infty} \tau^{p^n}(t)\frac{(\tau^{p^n}(v)-v)}{p^n}.
\end{split}
\end{equation*}
Hence it suffices to show that $N_{\nabla}(t)=E^{\prime}(u)\lambda u$ and that $\lim_{m\rightarrow \infty} \tau^{p^m}(t)=t$.

For the former, we just need to check that
\begin{equation*}
    N_{\nabla}(u)=\frac{1}{p\mathfrak{t}}\lim_{n\rightarrow \infty}\frac{\tau^{p^n}(u)}{p^n}=\frac{1}{p\mathfrak{t}}\lim_{n\rightarrow \infty}\frac{\log(\tau^{p^n})(u)}{p^n}=
    \frac{1}{p\mathfrak{t}}\lim_{n\rightarrow \infty}\frac{(\sum_{i=1}^{\infty}\frac{(-1)^{i-1}([\varepsilon]^{p^n}-1)^i}{i})u}{p^n}=\frac{1}{p\mathfrak{t}}\log([\varepsilon])u
    =\lambda u.
\end{equation*}
For the later, it suffices to check after modulo $t^k$ for any $k\in \mathbb{N}$. On the other hand, notice that by \cref{image of t}
\begin{equation*}
    \frac{\tau^{p^m}(t)}{t}=\alpha(\tau^{p^m})=1-\beta X_1(\tau^{p^m})+\sum_{i=1}^{\infty} t^i(\sum_{n=1}^e \theta_{n,i-(n-1)}\frac{(-1)^nX_1(\tau^{p^m})^n}{n!}).
\end{equation*}
For a fixed $k$, $\theta_k(\alpha(\tau^{p^m}))$ is a finite sum and that $\lim_{m\rightarrow \infty}\theta_k(X_1(\tau^{p^m}))=0$ by \cref{image of X under theta} and the proof of \cref{analytic of unit}, hence $\lim_{m\rightarrow \infty}\theta_k(\alpha(\tau^{p^m}))=1 $, which implies the desired result.
\end{proof}
Motivated by \cref{Kummer Sen operator}, we can introduce the following definition:
\begin{notation}
Define $\Mod_{K_{\infty}[[t]]}^{N_{\nabla}}$ to be the category consisting of objects are finite generated $K_{\infty}[[t]]$-modules $M$ equipped with a $K_{\infty}$-linear self map $N_{\nabla}: M\rightarrow M$ such that $N_{\nabla}$ satisfies Leibniz rule and that $ N_{\nabla}(tv)=E^{\prime}(u)\lambda u \cdot v+tN_{\nabla}(v)$. Morphisms in $\Mod_{K_{\infty}[[t]]}^{N_{\nabla}}$ are $K_{\infty}[[t]]$-module homomorphisms commuting with $N_{\nabla}$.
\end{notation}
\begin{remark}\label{t compatibility}
By the definition of $\lambda$, we see that as an element in $K_{\infty}[[t]]$, $\lambda=E(u)\lambda_1=t\lambda_1$ for some units $\lambda_1\in K[[t]]$. Hence given $\Mod_{K_{\infty}[[t]]}^{N_{\nabla}}$, for $v\in M$, $N_{\nabla}(tv)=E^{\prime}(u)\lambda_1 u tv+tN_{\nabla}(v) \in tM$. Then by induction we also see that $N_{\nabla}(t^r v)\subseteq t^rM$.
\end{remark}

Thanks to \cref{Kummer Sen operator}, we have a natural functor:
\begin{equation*}
    \begin{split}
        D: \rb &\longrightarrow \Mod_{K_{\infty}[[t]]}^{N_{\nabla}}
        \\ W&\longmapsto (D_{\mathrm{Sen},K_{\infty}[[t]]}(W), N_{\nabla})
    \end{split}
\end{equation*}
\begin{lemma}\label{D is exact}
$D$ is exact and faithful. Moreover, $D$ preserves tensor product and internal Hom such that for $W_1, W_2\in \rb$, $f\in \Hom_{B_{\dR}^+}(W_1,W_2)$,
\begin{equation*}
    \begin{split}
        N_{\nabla, W_1\otimes W_2}&=1\otimes N_{\nabla, W_2}+N_{\nabla, W_1}\otimes 1
        \\ N_{\nabla, \Hom_{B_{\dR}^+}(W_1,W_2)}(f)&=N_{\nabla, W_2}\circ f-f\circ N_{\nabla, W_1}.
    \end{split}
\end{equation*}
\end{lemma}
\begin{proof}
To show it is exact, it suffices to show that given a short exact sequence $0\rightarrow W_1\rightarrow W_2\rightarrow W_3\rightarrow 0$ in $\rb$, it is still exact after applying $D_{\mathrm{Sen},K_{\infty}[[t]]}(\cdot)$ functor. But as $K_{\infty}[[t]]\rightarrow B_{\dR}^+$ is faithfully flat, it suffices to check the exactness of $$0\rightarrow D_{\mathrm{Sen},K_{\infty}[[t]]}(W_1)\otimes_{K_{\infty}[[t]]}B_{\dR}^+\rightarrow D_{\mathrm{Sen},K_{\infty}[[t]]}(W_2)\otimes_{K_{\infty}[[t]]}B_{\dR}^+\rightarrow D_{\mathrm{Sen},K_{\infty}[[t]]}(W_3)\otimes_{K_{\infty}[[t]]}B_{\dR}^+\rightarrow 0.$$
But this is just the exact sequence of underlying $B_{\dR}^+$-modules of the original exact sequence thanks to \cref{combine three theorems}.
The faithfulness can be proven similarly.

Finally we have a natural map $D_{\mathrm{Sen},K_{\infty}[[t]]}(W_1)\otimes D_{\mathrm{Sen},K_{\infty}[[t]]}(W_2)\rightarrow D_{\mathrm{Sen},K_{\infty}[[t]]}(W_1\otimes W_2)$, to show that it is an equality, it suffices to check that after the faithfully flat base change $K_{\infty}[[t]]\rightarrow B_{\dR}^+$, which follows from \cref{combine three theorems} again. The desired property for $N_{\nabla}$ can be calculated directly using $N_{\nabla}=\frac{1}{p\mathfrak{t}}\nabla_{\tau}$.

For internal Hom, the result follows from the tensor product case as $\Hom_{B_{\dR}^+}(W_1,W_2)=W_1^{\vee}\otimes_{B_{\dR}^+} W_2$ in $\rb$.
\end{proof}
\begin{remark}
However, $D$ is neither fully faithful nor essentially surjective. For the former, just notice that morphisms in $\Mod_{K_{\infty}[[t]]}^{N_{\nabla}}$ are $K_{\infty}$-linear, while morphisms in $\rb$ are only $K$-linear. 
However, we would like to show that $D$ at least preserves the information about isomorphism classes in the remaining of this section.
\end{remark}

Let us consider the $t$-torsion case first.  
\begin{prop}\label{kernel of monodromy}
For $W\in \Rep_{\mathbb{C}_p}(G_K)$, the $K_{\infty}$-linear operator $N_{\nabla}$ on the finite dimensional $K_{\infty}$-vector space $D_{\mathrm{Sen},K_{\infty}}(W)$ satisfies the following:
\begin{itemize}
     \item The kernel $\Ker(N_{\nabla})$ consists exactly of those $x\in W^{G_L}$ which are further fixed by $\Gal(L/K_{\infty})$ and whose $\hat{G}$-orbit is finite. 
    \item $\Ker(N_{\nabla})$ is equal to $W^{G_K}\otimes_K K_{\infty}$. In particular, $N_{\nabla}$ is an isomorphism if and only if $(W)^{G_K}=0$, $N_{\nabla}=0$ if and only if $W^{G_L}$ has discrete $\hat{G}$-action. Also, $\dim_{K} W^{G_K}\leq \dim_{K_{\infty}} D_{\mathrm{Sen},K_{\infty}}(W)$.
\end{itemize}
\end{prop}
\begin{proof}
First given $x\in \Ker(N_{\nabla})$, by definition, we see that $x\in W^{G_L}$ is fixed by $\Gal(L/K_{\infty})$ and that $\log(\tau)(x)=0$, the later implies that for $n$ large enough, $\tau^{p^n}(x)=\exp(p^n\log(\tau))(x)=x$, hence the $\hat{G}$-orbit of $x$ is finite. On the other hand, if $x\in W^{G_L}$ is $\Gal(L/K_{\infty})$-invariant and that the $\hat{G}$-orbit of $x$ is finite, then for $n$ large enough, $\tau^{p^n}(x)=0$, hence $x\in (W^{G_L})^{\hat{G}-la}$ and that $\log(\tau)(x)=\lim_{n\rightarrow \infty}\frac{\tau^{p^n}(x)-x}{p^n}=0$.

Finally we have a natural map $K_{\infty}\otimes (W)^{G_K}\rightarrow \Ker(N_{\nabla})$, to show that it is an isomorphism, it suffices to check that after the faithfully flat field extension $K_{\infty}\rightarrow \hat{L}^{\hat{G}-la}$. Notice that due to the work of Gao (see \cite[Theorem 3.3.1, Theorem 3.3.2]{Gao22}), on $D_{\mathrm{Sen},K_{\infty}}(W)\otimes_{K_{\infty}}\hat{L}^{\hat{G}-la}=D_{\mathrm{Sen},K_{p^\infty}}(W)\otimes_{K_{p^\infty}}\hat{L}^{\hat{G}-la}$, there exist two linearly dependent monodromy operators, namely the $\hat{L}^{\hat{G}-la}$-linear extension of $N_{\nabla}$ on $D_{\mathrm{Sen},K_{\infty}}(W)$ (which will be denoted as $N_{\nabla}\otimes \Id$) and the $\hat{L}^{\hat{G}-la}$-linear extension of the classical Sen operator $\nabla_{\gamma}$ on $D_{\mathrm{Sen},K_{p^\infty}}(W)$ (which will be denoted as $\nabla_{\gamma}\otimes \Id$), and the exact relation between these two monodromy operator is that 
$N_{\nabla}\otimes \Id=\theta(u\lambda^{\prime}) (\nabla_{\gamma}\otimes \Id)$, where $\theta(u\lambda^{\prime})$ is a unit in $K$. As a result,
\begin{equation}\label{classical II}
    \Ker(N_{\nabla})\otimes \hat{L}^{\hat{G}-la}=\Ker(N_{\nabla}\otimes \Id)=\Ker(\nabla_{\gamma}\otimes \Id)=\Ker(\nabla_{\gamma})\otimes \hat{L}^{\hat{G}-la}.
\end{equation}
Here the first and last equality holds as any field extension is faithfully flat.

But by a similar argument as in the first paragraph, we see that $\Ker(\nabla_{\gamma})$ is precisely the $K_{p^{\infty}}$-subspace of $\Gamma_K$-discrete vectors inside $D_{\mathrm{Sen},K_{p^\infty}}(W)$, from which a classical Galois descent argument verifies that $\Ker(\nabla_{\gamma})=K_{p^{\infty}}\otimes W^{G_K}$, hence that
\begin{equation}\label{classical I}
    \Ker(\nabla_{\gamma})\otimes \hat{L}^{\hat{G}-la}=W^{G_K}\otimes \hat{L}^{\hat{G}-la}.
\end{equation}
Combining \cref{classical II} and \cref{classical I}, we see that $\Ker(N_{\nabla})\otimes \hat{L}^{\hat{G}-la}=W^{G_K}\otimes \hat{L}^{\hat{G}-la}$, hence the natural map $K_{\infty}\otimes (W)^{G_K}\rightarrow \Ker(N_{\nabla})$ is an equality. The other statements then follow.
\end{proof}

\begin{prop}\label{torsion isomorphic class detection}
For $W_1, W_2\in \Rep_{\mathbb{C}_p}(G_K)$, the natural map \begin{equation*}
    K_{\infty}\otimes_K \Hom_{\Rep_{\mathbb{C}_p}(G_K)}(W_1,W_2)\longrightarrow \Hom_{\Mod_{K_{\infty}[[t]]}^{N_{\nabla}}}((D_{\mathrm{Sen},K_{\infty}}(W_1), N_{\nabla}), (D_{\mathrm{Sen},K_{\infty}}(W_2), N_{\nabla}))
\end{equation*}
is an isomorphism.
\end{prop}
\begin{proof}
Let $W=\Hom_{\mathbb{C}_p}(W_1,W_2)$, then thanks to \cref{D is exact}, the desired identity is precisely $K_{\infty}\otimes_K W^{G_K}=\Ker(N_{\nabla,W})$,
which is precisely \cref{kernel of monodromy}.
\end{proof}

\begin{cor}\label{detect isomorphism torsion}
Two objects $W_1, W_2\in \Rep_{\mathbb{C}_p}(G_K)$ are isomorphic if and only if $(D_{\mathrm{Sen},K_{\infty}}(W_1), N_{\nabla})$ and $(D_{\mathrm{Sen},K_{\infty}}(W_2), N_{\nabla})$ are isomorphic as objects in $\Mod_{K_{\infty}[[t]]}^{N_{\nabla}}$.
\end{cor}
\begin{proof}
We imitate the proof for a similar result in the cyclotomic tower setting given in \cite[Proposition 2.6]{Fon04}. The only if direction follows from the functoriality of the functor $D_{\mathrm{Sen},K_{\infty}}(\cdot)$. For the other direction, assume that $(D_{\mathrm{Sen},K_{\infty}}(W_1), N_{\nabla})$ and $(D_{\mathrm{Sen},K_{\infty}}(W_2), N_{\nabla})$ are isomorphic as objects in $\Mod_{K_{\infty}[[t]]}^{N_{\nabla}}$. In particular, this implies that $\dim_{\mathbb{C}_p}W_1=\dim_{\mathbb{C}_p}W_2$ thanks to \cref{combine three theorems}.
let $W=\Hom_{\mathbb{C}_p}(W_1,W_2)\in \Rep_{\mathbb{C}_p}(G_K)$, then our assumption precisely says that $\Ker(N_{\nabla,W})$ contains an element of $W$ that is a linear isomorphism by \cref{D is exact}. We wish to further prove that $D^{G_K}$ contains an element that is a linear isomorphism, which exactly means that $W_1$ and $W_2$ are isomorphic as objects in $\Rep_{\mathbb{C}_p}(G_K)$. For this, one can apply \cite[Lemma 2.7]{Fon04} directly to conclude as $K_{\infty}/K$ is an algebraic extension with $K$ an infinite field.
\end{proof}

Now we are ready to treat the general case, which is motivated by the cyclotomic analogue given in \cite[Proposition 3.8]{Fon04}. We warn the reader that in the Kummer tower case we need to start with $W\in \rb$ as $D_{\mathrm{Sen},K_{\infty}[[t]]}(W)$ is not $G_K$-stable inside $W$.
\begin{prop}\label{general isomorphism class}
For any $W\in \rb$, the $K_{\infty}$-vector space $\Ker(N_{\nabla})$ (a subspace of $D_{\mathrm{Sen},K_{\infty}[[t]]}(W)$) is finite dimensional satisfying that 
\begin{equation}\label{general isomorphism class zero}
    W^{G_K}\otimes_K K_{\infty}=\Ker(N_{\nabla}).
\end{equation}
In particular, $\dim_K W^{G_K}$ is finite. Moreover, the following holds:
\begin{itemize}
    \item[i)] If $W\in \Rep_{B_{\dR}^+}^{\fp}(G_K)$, then $\dim_K(W^{G_K})\leq \dim_{K_{\infty}}(D_{\mathrm{Sen},K_{\infty}[[t]]}(W)/tD_{\mathrm{Sen},K_{\infty}[[t]]}(W))$.
    \item[ii)] For $W_1, W_2 \in \rb$,
    \begin{equation*}
        K_{\infty}\otimes_K \Hom_{\rb}(W_1,W_2)\stackrel{\cong}{\longrightarrow} \Hom_{\Mod_{K_{\infty}[[t]]}^{N_{\nabla}}}((D_{\mathrm{Sen},K_{\infty}[[t]]}(W_1), N_{\nabla}), (D_{\mathrm{Sen},K_{\infty}[[t]]}(W_2), N_{\nabla})).
    \end{equation*}
    Also, if we further assume that $W_1, W_2 \in \Rep_{B_{\dR}^+}^{\fp}(G_K)$, then 
    \begin{equation*}
        \dim_{K}(\Hom_{\rb}(W_1,W_2)\leq \dim_{K_{\infty}}(D_{\mathrm{Sen},K_{\infty}[[t]]}(W_1)/t)\dim_{K_{\infty}}(D_{\mathrm{Sen},K_{\infty}[[t]]}(W_2)/t).
    \end{equation*}
\end{itemize}
\end{prop}
To prove this proposition, we need the following lemma:
\begin{lemma}\label{finite L dimension}
For any $W\in \rb$, $(W^{G_L})^{\hat{G}-la, N_{\nabla}=0, \nabla_{\gamma}=0}$ is a finite dimensional $L$-vector space.
\end{lemma}
\begin{proof}
The proof of \cref{Main theorem III} implicitly implies that $(W^{G_L})^{\hat{G}-la, \nabla_{\gamma}=0}$ is a finite generated $L[[t]]$-module. If $W$ is $t$-power torsion, then there is nothing to prove as $(W^{G_L})^{\hat{G}-la, N_{\nabla}=0, \nabla_{\gamma}=0}$ is a $L$-subspace of $(W^{G_L})^{\hat{G}-la, \nabla_{\gamma}=0}$. From now now we assume that $W\in \Rep_{B_{\dR}^+}^{\fp}(G_K)$, hence $(W^{G_L})^{\hat{G}-la, \nabla_{\gamma}=0}$ is a finite free $L[[t]]$-module with the same rank of $W$ by the proof of  \cref{Main theorem III}. We aim to show that $\dim_L((W^{G_L})^{\hat{G}-la, N_{\nabla}=0, \nabla_{\gamma}=0})$ is bounded by that rank. Argue exactly as that in \cref{Kummer Sen operator}, we have that $N_{\nabla}$ maps $(W^{G_L})^{\hat{G}-la, \nabla_{\gamma}=0}$ to itself, in particular, $((W^{G_L})^{\hat{G}-la, \nabla_{\gamma}=0}, N_{\nabla})$ defines an object in $\Mod_{L[[t]]}^{N_{\nabla}}$. Hence it suffices to show the following:
\begin{itemize}
    \item[$(\ast\ast)$]  Suppose $(M,N_{\nabla})\in \Mod_{L[[t]]}^{N_{\nabla}}$ satisfies that $M$ is a finite free $L[[t]]$-module, then the natural map 
    \begin{equation*}
        L[[t]]\otimes_{L} M^{N_{\nabla}=0}\longrightarrow M
    \end{equation*}
    is injective.
\end{itemize}
To prove $(\ast\ast)$, we argue by contradiction, otherwise we can pick a nonzero element $x$ in the kernel with a minimal length expression $\sum f_i\otimes v_i$ in elementary tensors. In particular, all $f_i$ are nonzero and that $v_i$ are linearly independent over $L$. Suppose $f_i=t^{n_i}g_i$ such that $g_i(0)\neq 0$, i.e. $g(i)$ is a unit in $L[[t]]$. Then by dividing $t^{\min_{i}\{n_i\}}$ and multiplying a suitable unit, we can assume without of loss generality that $f_1=1$. Applying $N_{\nabla}$ to $0=\sum f_i v_i$ and using Leibniz rule, we have that
\begin{equation*}
    0=\sum (N_{\nabla}(f_i)v_i+f_i N_{\nabla}(v_i))=\sum N_{\nabla}(f_i)v_i
\end{equation*}
as $N_{\nabla}(v_i)=0$ by assumption. Notice that $N_{\nabla}(f_1)=0$, hence by minimal length assumption we must have that $N_{\nabla}(f_i)=0$ for arbitrary $i$. Using \cref{t compatibility} and that $N_{\nabla}(t)=E^{\prime}(u)\lambda_1 u t$ with $E^{\prime}(u)\lambda_1 u$ being a unit, one can easily verify that for $g\in L[[t]]$, $N_{\nabla}(g)=0$ if and only if $g\in L$, from which we conclude that $f_i\in K, \forall i$. This implies that $v_1=-\sum_{i>1}f_iv_i$, which gives a nontrivial linear dependence relation over $L$, a contradiction.
\end{proof}

Now we are ready to prove \cref{general isomorphism class}.
\begin{proof}[Proof of \cref{general isomorphism class}]
Clearly there is a natural map $W^{G_K}\otimes_K K_{\infty}\rightarrow \Ker(N_{\nabla})$, to show this an equality, if suffices to check that after scalar extension to $L$.

By \cref{finite L dimension}, $(W^{G_L})^{\hat{G}-la, N_{\nabla}=0, \nabla_{\gamma}=0}$ is a finite dimensional $L$-vector space. Also, argue as in the first paragraph in the proof of \cref{kernel of monodromy}, $(W^{G_L})^{\hat{G}-la, N_{\nabla}=0, \nabla_{\gamma}=0}$ consists precisely those elements in $W^{G_L}$ whose $\hat{G}$-orbit is finite. Then we can apply classical Galois descent to $L/K_{\infty}$ and $L/K_{p^\infty}$ separately to conclude that
\begin{equation}\label{general isomorphism class I}
    (W^{G_L})^{\hat{G}-la, N_{\nabla}=0, \nabla_{\gamma}=0}=(W^{G_L})^{\hat{G}-la, N_{\nabla}=0, \gamma=1}\otimes_{K_{\infty}} L=(W^{G_L})^{\hat{G}-la, \nabla_{\gamma}=0, \tau=1}\otimes_{K_{p^\infty}} L.
\end{equation}

Notice that $(W^{G_L})^{\hat{G}-la, \nabla_{\gamma}=0, \tau=1}=(W^{G_{K_{p^{\infty}}}})^{\Gamma_K-la,\nabla_{\gamma}=0}$. However, $\Gamma_K$ is a $p$-adic Lie group of dimension $1$, $\Gamma_K$-locally analytic vectors coincides with classical $K$-finite vectors thanks to \cite[Thm. 3.2]{BC16}, hence 
\[(W^{G_{K_{\infty}}})^{\Gamma_{K}-la}=(W^{G_{K_{\infty}}})_f,\]
where the latter denotes the subspace of $W^{G_{K_{\infty}}}$ consisting of those vectors whose $\Gamma_{K}$ orbits are finite,  see \cite[Section 3.3]{Fon04} for details. In particular, by \cite[Theorem 3.6]{Fon04}, $(W^{G_K})^{\Gamma_{K}-la}$ fits into the assumption in \cite[Proposition 3.8]{Fon04}, then the proof of \cite[Proposition 3.8]{Fon04} shows that \begin{equation}\label{general isomorphism class II}
    (W^{G_{K_{p^{\infty}}}})^{\Gamma_K-la,\nabla_{\gamma}=0}=W^{G_K}\otimes_{K} K_{p^{\infty}}.
\end{equation}
Combine \cref{general isomorphism class I} and \cref{general isomorphism class I}, we have that 
\begin{equation}\label{general isomorphism class III}
   \Ker(N_{\nabla})\otimes_{K_{\infty}} L= (W^{G_L})^{\hat{G}-la, N_{\nabla}=0, \gamma=1}\otimes_{K_{\infty}} L=(W^{G_L})^{\hat{G}-la, N_{\nabla}=0, \nabla_{\gamma}=0}=W^{G_K}\otimes_{K} L.
\end{equation}
Hence $W^{G_K}\otimes_K K_{\infty}\rightarrow \Ker(N_{\nabla})$ is an equality and that $\Ker(N_{\nabla})$ is a finite dimensional $K_{\infty}$-vector space as $(W^{G_L})^{\hat{G}-la, N_{\nabla}=0, \nabla_{\gamma}=0}$ is of finite dimension over $L$ by \cref{finite L dimension}.  

For $i)$, just notice that by \cref{general isomorphism class III}, $\dim_K W^{G_K}=\dim_L (W^{G_L})^{\hat{G}-la, N_{\nabla}=0, \nabla_{\gamma}=0} $, hence the proof of \cref{finite L dimension} implies that it is bounded by the rank of $D_{\mathrm{Sen},K_{\infty}[[t]]}(W)$, which is also $\dim_{K_{\infty}}(D_{\mathrm{Sen},K_{\infty}[[t]]}(W)/tD_{\mathrm{Sen},K_{\infty}[[t]]}(W))$.

For $ii)$, we proceed as in \cref{torsion isomorphic class detection}. Consider $W=\Hom_{B_{\dR}^+}(W_1,W_2)\in \rb$. By the proof of \cref{D is exact}, we see that $$\Hom_{K_{\infty}[[t]]}(D_{\mathrm{Sen},K_{\infty}[[t]]}(W_1),D_{\mathrm{Sen},K_{\infty}[[t]]}(W_2))=D_{\mathrm{Sen},K_{\infty}[[t]]}(W),$$
and that 
\begin{equation*}
    N_{\nabla, W}(f)=N_{\nabla, W_2}\circ f-f\circ N_{\nabla, W_1}.
\end{equation*}
Hence given $f\in D_{\mathrm{Sen},K_{\infty}[[t]]}(W)$, $f$ defines a morphism in $\Mod_{K_{\infty}[[t]]}^{N_{\nabla}}$ if and only if $f\in \Ker(N_{\nabla, W})$. In other words, 
\begin{equation*}
    \Hom_{\Mod_{K_{\infty}[[t]]}^{N_{\nabla}}}((D_{\mathrm{Sen},K_{\infty}}(W_1), N_{\nabla}), (D_{\mathrm{Sen},K_{\infty}}(W_2), N_{\nabla}))=\Ker(N_{\nabla, W}).
\end{equation*}
Then we can apply \cref{general isomorphism class zero} to conclude that
\begin{equation}\label{general isomorphism class V}
    K_{\infty}\otimes_K \Hom_{\rb}(W_1,W_2)\stackrel{\cong}{\longrightarrow} \Hom_{\Mod_{K_{\infty}[[t]]}^{N_{\nabla}}}((D_{\mathrm{Sen},K_{\infty}}(W_1), N_{\nabla}), (D_{\mathrm{Sen},K_{\infty}}(W_2), N_{\nabla})),
\end{equation}
from which we see that for $W_1, W_2 \in \Rep_{B_{\dR}^+}^{\fp}(G_K)$, 
    \begin{equation*}
    \begin{split}
        \dim_{K}(\Hom_{\rb}(W_1,W_2)&\leq 
        \dim_{K_{\infty}[[t]]}(D_{\mathrm{Sen},K_{\infty}[[t]]}(W_1))\dim_{K_{\infty}}(D_{\mathrm{Sen},K_{\infty}[[t]]}(W_2))
        \\&= \dim_{K_{\infty}}(D_{\mathrm{Sen},K_{\infty}[[t]]}(W_1)/t)\dim_{K_{\infty}}(D_{\mathrm{Sen},K_{\infty}[[t]]}(W_2)/t).
    \end{split}
    \end{equation*}\qedhere
\end{proof}
\begin{cor}\label{detect isomorphic class}
Two objects $W_1, W_2\in \rb$ are isomorphic if and only if $(D_{\mathrm{Sen},K_{\infty}[[t]]}(W_1), N_{\nabla})$ and $(D_{\mathrm{Sen},K_{\infty}[[t]]}(W_2), N_{\nabla})$ are isomorphic as objects in $\Mod_{K_{\infty}[[t]]}^{N_{\nabla}}$.
\end{cor}
\begin{proof}
The "only if" is obvious by functoriality of $D_{\mathrm{Sen},K_{\infty}[[t]]}(\cdot)$. To show the other direction, we proceed as in \cite[Proposition 3.8]{Fon04}. Assume that there is an isomorphism $f_0$ in  $$\Hom_{\Mod_{K_{\infty}[[t]]}^{N_{\nabla}}}((D_{\mathrm{Sen},K_{\infty}[[t]]}(W_1), N_{\nabla}),(D_{\mathrm{Sen},K_{\infty}[[t]]}(W_2), N_{\nabla})),$$ and we wish to construct a $\Gamma_K$-equivariant isomorphism from $W_1$ to $W_2$ (see \cref{general isomorphism class V}). Let $\{x_1,\cdots,x_d\}$ and $\{y_1,\cdots,y_{d^{\prime}}\}$ be minimal $K_{\infty}[[t]]$-module generating sets for $D_{\mathrm{Sen},K_{\infty}[[t]]}(W_1)$ and $D_{\mathrm{Sen},K_{\infty}[[t]]}(W_2)$ separately. By assumption we naturally have that $D_{\mathrm{Sen},K_{\infty}[[t]]}(W_1)$ and $D_{\mathrm{Sen},K_{\infty}[[t]]}(W_2)$ are isomorphic as $K_{\infty}[[t]]$-modules via $f_0$, in particular, $d=d^{\prime}$.

Notice that $\Hom_{\rb}(W_1,W_2)$ is a finite dimensional $K$-vector space by \cref{general isomorphism class}, hence we can choose a $K$-basis $\{f_1,\cdots,f_n\}$ of it. Let 
\begin{equation*}
    \overline{f}_i\in \Hom_{K_{\infty}}(D_{\mathrm{Sen},K_{\infty}[[t]]}(W_1)/tD_{\mathrm{Sen},K_{\infty}[[t]]}(W_1),D_{\mathrm{Sen},K_{\infty}[[t]]}(W_2)/tD_{\mathrm{Sen},K_{\infty}[[t]]}(W_2))
\end{equation*}
be the reduction of $f_i$ modulo $t$. Concretely, each $\overline{f}_i$ can be described by a $d\times d$ matrix over $K_{\infty}$ utilizing the bases $\{\overline{x}_j\}$ and $\{\overline{x}_j\}$, which are just reductions of $\{x_j\}$ and $\{y_j\}$ modulo $t$. Now apply \cref{general isomorphism class} ii), we see that $\{f_1,\cdots,f_n\}$ forms a $K_{\infty}$-basis for 
\begin{equation*}
    \Hom_{\Mod_{K_{\infty}[[t]]}^{N_{\nabla}}}((D_{\mathrm{Sen},K_{\infty}[[t]]}(W_1), N_{\nabla}),(D_{\mathrm{Sen},K_{\infty}[[t]]}(W_2), N_{\nabla})),
\end{equation*}
which contains an isomorphism by our assumption. As a result, there exist $\lambda_1,\cdots,\lambda_n\in K_{\infty}$ such that $\det(\lambda_1\overline{f}_1+\cdots+\lambda_n\overline{f}_n)\neq 0$, which implies that the polynomial $\det(X_1\overline{f}_1+\cdots+X_n\overline{f}_n)\in K_{\infty}[X_1,\cdots,X_n]$ is non-zero. On the other hand, as $K$ is an infinite field, we can pick up $\mu_1,\cdots,\mu_n\in K$ such that $\det(\overline{f})\neq 0$ for $f=\mu_1f_1+\cdots \mu_nf_n$. Then $f$ defines a $K_{\infty}[[t]]$-linear isomorphism $(D_{\mathrm{Sen},K_{\infty}[[t]]}(W_1)\rightarrow (D_{\mathrm{Sen},K_{\infty}[[t]]}(W_2)$ as its reduction modulo $t$ is an isomorphism for determinant reasons.
But now $f\in \Hom_{\rb}(W_1,W_2)$, hence $f$ is an isomorphism in $\rb$, we are done.
\end{proof}

\section{Representations associated to de Rham prismatic crystals}
\subsection{De Rham crystals and $B_{\dR}^+$-representations}
We define $A_{L}$ to be $\Ainf(\mathcal{O}_{\hat{L}})$ and $A_{L,\perf}^{i}$ to be the $i$-th self products of $A_{L}$ in $(\mathcal{O}_K)_{\Prism}^{\perf}$. 

We denote by $\operatorname{Vect}((\mathcal{O}_{K})_{\Prism}^{\perf}, (\mathcal{O}_{\Prism}[\frac{1}{p}])_{\mathcal{I}}^{\wedge})$ the category of de Rham prismatic crystals on $(\mathcal{O}_K)_{\Prism}^{\perf}$.
\begin{cor}
The category of de Rham prismatic crystals on $(\mathcal{O}_K)_{\Prism}^{\perf}$ is equivalent to the category of finite free $B_{\dR,L}^+$-modules $\mathcal{M}$ on which there is stratification satisfying cocycle condition.
\end{cor}
\begin{proof}
As $(A_L,E)$ is a weakly final object in $(\mathcal{O}_K)_{\Prism}^{\perf}$, one can proceed as in \cref{algebraic description of crystal}.
\end{proof}
\begin{prop}
There is a canonical isomorphism of cosimplicial rings
\[(\mathcal{O}_{\Prism}[\frac{1}{p}])_{\mathcal{I}}^{\wedge}(A_{L,\perf}^{\bullet})\cong C(\hat{G}^{\bullet},B_{\dR,L}^+).\]
\end{prop}
\begin{proof}
Unwinding definition, it suffices to show that 
\[\bdr(A_{L,\perf}^{\bullet})\cong C(\hat{G}^{\bullet},B_{\dR,L}^+).\]
Recall that $\bdr(A_{L,\perf}^{i})$ is equipped with the inverse limit topology induced from the quotient topology on each $\bdr(A_{L,\perf}^{i})/(\xi^j)=A_{L,\perf}^{i}[\frac{1}{p}]/(\xi^j)$, hence it suffices to show that
\[A_{L,\perf}^{i}[\frac{1}{p}]/(\xi)\cong C(\hat{G}^{i},\hat{L}).\]
Here we identify $C(\hat{G}^{\bullet},B_{\dR,L}^+)/(\xi)$ with $C(\hat{G}^{\bullet},\hat{L})$ as multiplication map by $\xi$ on $B_{\dR,L}^+$ is a closed embedding with respect to the topology on $B_{\dR,L}^+$.

But this is already proven in the proof of \cite[Proposition 3.10]{MW21}.
\end{proof}
As a consequence, by Galois descent, we have the following result:
\begin{theorem}\label{de Rham realization}
The category of de Rham prismatic crystals on $(\mathcal{O}_K)_{\Prism}^{\perf}$ is equivalent to the category of $B_{\dR,L}^+$-representation of $\hat{G}$.
\end{theorem}

By either using $\Ainf$ instead of $A_L$ in the proof of the previous theorem or utilizing \cref{de rham almost purity}, we have that:
\begin{theorem}\label{various representation equivalence}
We have the following commutative diagrams in which all the arrows induce equivalences of categories
    \[\begin{tikzcd}[cramped, sep=large]
\operatorname{Vect}((\mathcal{O}_{K})_{\Prism}^{\perf}, (\mathcal{O}_{\Prism}[\frac{1}{p}])_{\mathcal{I}}^{\wedge}) \arrow[rd, "T"] \arrow[r, "T_L"] & \Rep_{\hat{G}}(B_{\dR,L}^+) \arrow[d, "-\otimes_{B_{\dR,L}^+}B_{\dR}^+ "]\\
& \rb
\end{tikzcd}\]
Here given $\mathcal{M}\in \operatorname{Vect}((\mathcal{O}_{K})_{\Prism}^{\perf}, (\mathcal{O}_{\Prism}[\frac{1}{p}])_{\mathcal{I}}^{\wedge})$, $T(\mathcal{M})=\mathcal{M}(\Ainf)$ and $T(\mathcal{M})=\mathcal{M}(A_L)$.
\end{theorem}

\begin{example}
Identify $B_{\dR,L}^+$ (resp. $\bdr(A_{L,\perf}^{1})$) with $C(\hat{G}^{0},B_{\dR,L}^+)$ (resp. $C(\hat{G}^{1},B_{\dR,L}^+)$), then for $x\in B_{\dR,L}^+$, $\delta_1^1(x)(g)=x$, $\delta_0^1(x)(g)=g(x)$. In particular, as $u_{0}(g)=\delta_1^1(u_0)(g)=[\pi^\flat],\quad
        u_{1}(g)=\delta_0^1(u_0)(g)=g[\pi^\flat]=[\varepsilon]^{c(g)}[\pi^\flat]$, we have that
\begin{equation*}
    X_1(g)=\frac{u_0-u_1}{E(u_0)}(g)=\frac{(1-[\varepsilon]^{c(g)})[\pi^{\flat}]}{E([\pi^{\flat}])}.
\end{equation*}
\end{example}

\subsection{Full faithfulness of the restriction functor} 
Consider the natural restriction functor 
\begin{equation*}
    V: \operatorname{Vect}((\mathcal{O}_{K})_{\Prism}, (\mathcal{O}_{\Prism}[\frac{1}{p}])_{\mathcal{I}}^{\wedge})\rightarrow \operatorname{Vect}((\mathcal{O}_{K})_{\Prism}^{\perf}, (\mathcal{O}_{\Prism}[\frac{1}{p}])_{\mathcal{I}}^{\wedge})\cong \Rep_{\hat{G}}(B_{\dR,L}^+)
\end{equation*}
by restricting a de Rham crystal to the the perfect prismatic site.

In this section, we aim to prove $V$ is fully faithful. 

To do this, we need several preliminaries.
First we would like to give an explicit description of the $\hat{G}$ representation $V(\mathcal{M})$ given $\mathcal{M} \in \operatorname{Vect}((\mathcal{O}_{K})_{\Prism}, (\mathcal{O}_{\Prism}[\frac{1}{p}])_{\mathcal{I}}^{\wedge})$. 

\begin{prop}\label{cocycle g action}
Let $\mathcal{M} \in \operatorname{Vect}((\mathcal{O}_{K})_{\Prism}, (\mathcal{O}_{\Prism}[\frac{1}{p}])_{\mathcal{I}}^{\wedge})$ be a de Rham prismatic crystal associated to a pair $(M,\varepsilon)$ as in \cref{algebraic description of crystal}. Choose a basis $\underline e$ in $M$ such that
\begin{equation*}
    \epsilon(\underline{e})=\underline e\cdot \sum_{m\geq0}(\sum_{n\geq 0}A_{m,n}X_1^{[n]})t^m.
\end{equation*}
 Then for $\Vec{v}=\underline e B$ a vector in  $V(\mathcal{M})=M\otimes_{\bdr(\mathfrak{S})}B_{\dR,L}^+$ with $B\in M_{l\times 1}(B_{\dR,L}^+)$, then for $g\in \hat{G}$, 
 \begin{equation*}
     g(\Vec{v})=\underline e U(g)g(B) \quad \text{where}\quad U(g)=\sum_{n\geq 0}A_{m,n}X_1(g)^{[n]})t^m.
 \end{equation*}
\end{prop}
\begin{proof}
This follows from the way that $\hat{G}$-action is constructed from stratification. 
\end{proof}

\begin{lemma}\label{image of X under theta}
$\theta_k(X(g))=\theta_{k}(\frac{[\pi^{\flat}](1-[\varepsilon])}{E([\pi^{\flat}])})(\sum_{i=1}^{k}\frac{(-1)^{i-1}c(g)\cdots (c(g)-i+1)}{i!}(1-[\varepsilon])^{i-1})$.
\end{lemma}
\begin{proof}
By definition, $X(g)=\frac{[\pi^{\flat}](1-[\varepsilon])}{E([\pi^{\flat}])}\cdot \frac{1-[\varepsilon]^{c(g)}}{1-[\varepsilon]}$. Let $Y(g):=\frac{1-[\varepsilon]^{c(g)}}{1-[\varepsilon]}$, hence it suffices to show
\begin{equation}\label{calculating Y}
    \theta_k(Y(g))=\sum_{i=1}^{k}\frac{(-1)^{i-1}c(g)\cdots (c(g)-i+1)}{i!}(1-[\varepsilon])^{i-1}.
\end{equation}
We prove it by induction on $k$. When $k=1$, $\theta_{1}(Y(g))=c(g)$ as $\theta_{1}([\varepsilon])=1$. Suppose \cref{calculating Y} is proven for up to $k$, then 
\begin{equation*}
    \begin{split}
        Y(g)&=(Y(g)-\sum_{i=0}^{k-1}\frac{(-1)^{i-1}c(g)\cdots (c(g)-i+1)}{i!}(1-[\varepsilon])^i)+\sum_{i=1}^{k}\frac{(-1)^{i-1}c(g)\cdots (c(g)-i+1)}{i!}(1-[\varepsilon])^{i-1}
        \\&=\frac{(Y(g)-\sum_{i=1}^{k}\frac{(-1)^{i-1}c(g)\cdots (c(g)-i+1)}{i!}(1-[\varepsilon])^{i-1})}{(1-[\varepsilon])^k}\cdot(1-[\varepsilon])^k+\sum_{i=1}^{k}\frac{(-1)^{i-1}c(g)\cdots (c(g)-i+1)}{i!}(1-[\varepsilon])^{i-1}.
    \end{split}
\end{equation*}
Then the desired result follows as one can use L'Hôpital's rule to calculate that 
\[\theta_{1}(\frac{(Y(g)-\sum_{i=1}^{k}\frac{(-1)^{i-1}c(g)\cdots (c(g)-i+1)}{i!}(1-[\varepsilon])^{i-1})}{(1-[\varepsilon])^k})=\frac{(-1)^{k}c(g)\cdots (c(g)-k)}{(k+1)!}.\]
\end{proof}
\begin{prop}\label{analytic of unit}
Suppose $\mathcal{M} \in \operatorname{Vect}((\mathcal{O}_{K})_{\Prism}, (\mathcal{O}_{\Prism}[\frac{1}{p}])_{\mathcal{I}}^{\wedge})$, then $\underline e \subseteq V(\mathcal{M})^{\hat{G}-la,\gamma=1}.$
\end{prop}
\begin{proof}
Recall that for $g\in \hat{G}$, $g\cdot \underline e=\underline e U(g)$ and $U(g)=\sum_{m\geq 0}(\sum_{n\geq 0}A_{m,n}X(g)^{[n]})t^m$. By \cref{de rham analytic}, it suffices to show $\underline e \subseteq (V(\mathcal{M})/(t^k))^{\hat{G}-la,\gamma=1}$ for arbitrary $k$. Fix such a $k$, thanks to \cref{image of X under theta}, it suffices to show for $p\leq k-1$,
\begin{equation}\label{coefficient of c(g)}
    \begin{split}
        \lim_{s\rightarrow \infty} A_{p,s}\frac{\theta_{k}(\frac{[\pi^{\flat}](1-[\varepsilon])}{E([\pi^{\flat}])})^s}{s!}=0.
    \end{split}
\end{equation}
This can be checked after modulo $t$ by \cref{check convergence after modulo}. Then notice that 
\begin{equation*}
    \begin{split}
        \theta_{1}(\frac{[\pi^{\flat}](1-[\varepsilon])}{E([\pi^{\flat}])})=\pi(1-\xi_p)\theta_{1}(\frac{\xi}{E([\pi^{\flat}])}),
    \end{split}
\end{equation*}
and that $v_{p}((1-\xi_p)^s)=\frac{s}{p-1}>\frac{s-1}{p-1}\geq v_p(s!)$, $\lim_{s\rightarrow \infty}A_{p,s}=0$ by \cref{remark with MW relation}, hence \cref{coefficient of c(g)} holds, we win.
\end{proof}
\begin{theorem}\label{decompletion for de Rham crystals}
    Suppose $\mathcal{M}\in \operatorname{Vect}((\mathcal{O}_{K})_{\Prism}, (\mathcal{O}_{\Prism}[\frac{1}{p}])_{\mathcal{I}}^{\wedge})$ such that the associated stratification $\varepsilon$ is given by a sequence of commutative matrices. Let $M=\mathcal{M}(\mathfrak{S},(E(u)))$, and $V(\mathcal{M})=\mathcal{M}(\Ainf,(\xi))$, then
    \[D_{\mathrm{Sen},K_{\infty}[[t]]}(V(\mathcal{M}))=M\otimes_{K[[t]]} K_{\infty}[[t]],\]
here we have identified $\bdr(\mathfrak{S})$ with $K[[T]]$.
\end{theorem}
\begin{proof}
First we check that $M\subseteq V(M)^{\hat{G}-la,\gamma=1}$. By our choice of embedding $(\mathfrak{S},(E))$ into $(\Ainf,(\xi))$ via sending $u$ to $[\pi^\flat]$, $M$ is invariant under $\Gal(L/K_{\infty})$-action. Hence it suffices to show that $M\subseteq V(M)^{\tau-la}$, which follows from \cref{analytic of unit}. By extending scalars along $K[[t]]\rightarrow (B_{\dR,L}^+)^{\hat{G}-la,\gamma=1}$, which is just $K_{\infty}[[t]]$ by \cref{de rham analytic calculation}, we get a natural map $$M\otimes_{K[[t]]} K_{\infty}[[t]]\longrightarrow D_{\mathrm{Sen},K_{\infty}[[t]]}(V(\mathcal{M})).$$
To show that it is an identity, it suffices to check that after the faithfully flat base change $K_{\infty}[[t]]\rightarrow B_{\dR}^+$, then the desired result follows from \cref{combine three theorems}.
\end{proof}
\begin{theorem}\label{fully faithful of the restriction functor}
    The restriction functor 
    \[V: \operatorname{Vect}((\mathcal{O}_{K})_{\Prism}, (\mathcal{O}_{\Prism}[\frac{1}{p}])_{\mathcal{I}}^{\wedge})\longrightarrow \operatorname{Vect}((\mathcal{O}_{K})_{\Prism}^{\perf}, (\mathcal{O}_{\Prism}[\frac{1}{p}])_{\mathcal{I}}^{\wedge})\]
is fully faithful.
\end{theorem}
\begin{proof}
It suffices to show that For $\mathcal{M}\in \operatorname{Vect}((\mathcal{O}_{K})_{\Prism}, (\mathcal{O}_{\Prism}[\frac{1}{p}])_{\mathcal{I}}^{\wedge})$, $H^0((\mathcal{O}_{K})_{\Prism},\mathcal{M})\cong \mathcal{M}(\Ainf)^{G_K}$. Suppose $\Vec{v}\in \mathcal{M}(\Ainf)^{G_K}$, then $\Vec{v}\in M\otimes_{K[[t]]} K_{\infty}[[t]]$ thanks to \cref{decompletion for de Rham crystals}. Hence it suffices to show that given
$\vec{v}=\underline e\sum_{m=0}^{\infty} D_mt^m \in \mathcal{M}(\Ainf)$ with $D_m\in M_{l\times 1}(K_\infty)$, $\Vec{v}\in \mathcal{M}(\Ainf)^{G_K}$ if and only if $D_{m}\in M_{l\times 1}(K)$ and that $D_m$ satisfies \cref{global section of crystal}, which is exactly the condition that $$\vec{v}\in H^0((\mathcal{O}_{K})_{\Prism},\mathcal{M}).$$
For this, first we calculate the $\hat{G}$-action on $\vec{v}$. Let $g\in \hat{G}$
\begin{equation*}
    \begin{split}
        g(\Vec{v})=g(\underline e (\sum_{i=0}^{\infty}D_it^i))&=\underline e(U(g)(\sum_{i=0}^{\infty}g(D_i)(\alpha(g) t)^i)
        \\&=\underline e\cdot (\sum_{m\geq0}(\sum_{n\geq 0}A_{m,n}X_1(g)^{[n]})t^m)(\sum_{i=0}^{\infty}g(D_i)(\alpha(g) t)^i)
        \\&=\underline e\cdot (\sum_{m\geq0}(\sum_{n\geq 0}A_{m,n}X_1(g)^{[n]})t^m)(\sum_{i=0}^{\infty}g(D_i) t^i(\sum_{s=0}^{\infty}c_{i,s}(X_1(g))t^s))
        \\&=\underline e\cdot (\sum_{m\geq0}(\sum_{n\geq 0}A_{m,n}X_1(g)^{[n]})t^m)(\sum_{j=0}^{\infty}t^j(\sum_{i=0}^{j}g(D_i) c_{i,j-i}(X_1(g))))
        \\&=\underline e\sum_{m=0}^{\infty}t^m(\sum_{\substack{i+j=m\\0\leq i,j\leq m}}(\sum_{s=0}^{\infty}A_{i,s}X_{1}(g)^{[s]})(\sum_{p=0}^{j}g(D_p) c_{p,j-p}(X_{1}(g)))).
    \end{split}
\end{equation*}
This implies that $\gamma \Vec{v}=\Vec{v}$ as $X_1(\gamma)=0$ and that $D_m\in M_{l\times 1}(K_\infty)$ is fixed by $\gamma$. Hence $\Vec{v}$ is $\hat{G}$-invariant if and only if $\Vec{v}$ is $\tau$-invariant, i.e.
\begin{equation}\label{g action}
    \underline e\sum_{m=0}^{\infty} D_mt^m=\underline e\sum_{m=0}^{\infty}t^m(\sum_{\substack{i+j=m\\0\leq i,j\leq m}}(\sum_{s=0}^{\infty}A_{i,s}X_{1}(\tau)^{[s]})(\sum_{p=0}^{j}\tau(D_p) c_{p,j-p}(X_{1}(\tau)))).
\end{equation}
By the $t$-adic completeness of the finite free $K_{\infty}[[t]]$-module $M\otimes_{K[[t]]} K_{\infty}[[t]]$, \cref{g action} is equivalent to that for $\forall k\geq 1$, 
\begin{equation}\label{g action mod k}
    \underline e\sum_{m=0}^{k-1} D_mt^m=\underline e\sum_{m=0}^{k-1}t^m(\sum_{\substack{i+j=m\\0\leq i,j\leq m}}(\sum_{s=0}^{\infty}A_{i,s}\theta_k(X_{1}(\tau))^{[s]})(\sum_{p=0}^{j}\tau(D_p) c_{p,j-p}(\theta_k(X_{1}(\tau))))) \text{~~~~~~after modulo~} t^k.
\end{equation}
For $k=1$, \cref{g action mod k} turns into
\begin{equation}\label{k=1}
    D_0=(\sum_{s=0}^{\infty}A_{0,s}\theta_1(X_{1}(\tau))^{[s]})\tau(D_0).
\end{equation}
We claim that this is equivalent to that
\begin{itemize}
    \item[$(\ast)$]  $f(Z):=(\sum_{s=0}^{\infty}A_{0,s}Z^{[s]})\tau(D_0)-D_0$ is always $0$, where $Z$ is a free variable. In particular, $\tau(D_0)=D_0$, hence $D_0\in  M_{l\times 1}(K)$.
\end{itemize}
Clearly $\ast$ implies \cref{k=1}. On the other hand, given \cref{k=1}, we can always find a finite extension $\tilde{K}$ over $K$ such that $(1-\xi_p)\in \tilde{K}$, $D_0\in  M_{l\times 1}(\tilde{K})$ and $\tau(D_0)\in  M_{l\times 1}(\tilde{K})$. By \cref{image of X under theta}, $$\theta_1(X(\tau))=\theta_{1}(\frac{[\pi^{\flat}](1-[\varepsilon])}{E([\pi^{\flat}])})c(\tau)=\pi\theta_{1}(\frac{(1-[\varepsilon])}{E([\pi^{\flat}])})=\pi(1-\xi_p)\theta_{1}(\frac{\xi}{E([\pi^{\flat}])}).$$ 
    This implies that $\omega_0:=\theta_{1}(\frac{\xi}{E([\pi^{\flat}])})$ is a root of $f(\pi(1-\xi_p)Z)$. Notice that $v_p((1-\xi_p)^s)=\frac{s}{p-1}>v_p(s!)$ and that $\lim_{s\rightarrow\infty} A_{0,s}=0$, hence $f(\pi(1-\xi_p)Z)\in M_{n}(\mathcal{O}_{\tilde{K}})\langle Z\rangle$   after multiplying a scalar. By Weierstrass preparation theorem, $\omega_0$ is algebraic over $\tilde{K}$ (hence also algebraic over $K$) unless $f=0$. However, \cite[Lem 2.4.4]{DL21} shows that $\omega_0$ is not algebraic over $K$, hence $f=0$.
    
    Suppose we have shown that 
    \begin{equation}\label{global invariant of crystal}
    \begin{split}
        D_m=\sum_{\substack{i+j=m\\0\leq i,j\leq m}}(\sum_{s=0}^{\infty}A_{i,s}Z^{[s]})(\sum_{p=0}^{j}\tau(D_p) c_{p,j-p}(Z))
    \end{split}
\end{equation}
for $m\leq n-1$. Then take $k=n+1$, \cref{g action mod k} turns into
\begin{equation*}
    \underline e\sum_{m=0}^{n} D_mt^m=\underline e\sum_{m=0}^{n}t^m(\sum_{\substack{i+j=m\\0\leq i,j\leq m}}(\sum_{s=0}^{\infty}A_{i,s}\theta_{n+1}(X_{1}(\tau))^{[s]})(\sum_{p=0}^{j}\tau(D_p) c_{p,j-p}(\theta_{n+1}(X_{1}(\tau)))))\text{~~~~~~after modulo~} t^{n+1}.
\end{equation*}
However, by induction (\cref{global invariant of crystal}), 
\begin{equation*}
    \underline e\sum_{m=0}^{n-1} D_mt^m=\underline e\sum_{m=0}^{n-1}t^m(\sum_{\substack{i+j=m\\0\leq i,j\leq m}}(\sum_{s=0}^{\infty}A_{i,s}\theta_{n+1}(X_{1}(\tau))^{[s]})(\sum_{p=0}^{j}\tau(D_p) c_{p,j-p}(\theta_{n+1}(X_{1}(\tau))))).
\end{equation*}
This implies that 
\begin{equation*}
     D_n=\sum_{\substack{i+j=n\\0\leq i,j\leq n}}(\sum_{s=0}^{\infty}A_{i,s}\theta_1(X_{1}(\tau))^{[s]})(\sum_{p=0}^{j}\tau(D_p) c_{p,j-p}(\theta_1(X_{1}(\tau)))).
\end{equation*}
Arguing exactly as the case $k=1$ (this makes sense as $c_{p,j-p}(Z)\in K[Z]$), we conclude that 
\begin{itemize}
    \item[$(\ast^n)$]  $f_n(Z):=\sum_{\substack{i+j=n\\0\leq i,j\leq n}}(\sum_{s=0}^{\infty}A_{i,s}Z^{[s]})(\sum_{p=0}^{j}\tau(D_p) c_{p,j-p}(Z)-D_n$ is always $0$, where $Z$ is a free variable. In particular, $\tau(D_n)=D_n$, hence $D_n\in  M_{l\times 1}(K)$.
\end{itemize}
Now by inudction we conclude that for all $m$, $D_m\in  M_{l\times 1}(K)$ and satisfies that \cref{global section of crystal}, hence $\Vec{v}\in H^0((\mathcal{O}_{K})_{\Prism},\mathcal{M})$, we are done.

\end{proof}

\subsection{Essential image of the restriction functor}
Recall that a de Rham crystal $\mathcal{M}\in \operatorname{Vect}((\mathcal{O}_{K})_{\Prism}, (\mathcal{O}_{\Prism}[\frac{1}{p}])_{\mathcal{I}}^{\wedge})$ is uniquely determined by a pari $(M,\varepsilon)$ where $M$ is a finite free $\bdr(\mathfrak{S})$-modules and $\varepsilon$ is a  stratification $$\varepsilon(\underline{e})=\underline e\cdot \sum_{m\geq0}(\sum_{n\geq 0}A_{m,n}X^{[n]})t^m.$$
On the other hand, thanks to \cref{Kummer Sen operator} and \cref{decompletion for de Rham crystals}, we can associate to $\mathcal{M}$ a pair $(D_{\mathrm{Sen},K_{\infty}[[t]]}(V(\mathcal{M})),N_{\nabla})$, where $D_{\mathrm{Sen},K_{\infty}[[t]]}(V(\mathcal{M}))=M\otimes_{K[[t]]} K_{\infty}[[t]]$ is a finite dimensional $K_{\infty}[[t]]$-space and  $N_{\nabla}$ is the Kummer-Sen operator on $D_{\mathrm{Sen},K_{\infty}[[t]]}(V(\mathcal{M}))$. Motivated by Gao's proof that one can extract the stratification information of a rational Hodge-Tate crystal from the Sen operator, we would like to read off the stratification of a de Rham crystal from the Kummer-Sen operator $N_{\nabla}$. The main result is the following:
\begin{theorem}\label{operator and stratification}
    Assume as before, then 
    \[N_{\nabla}(\underline e)=-\underline e \lambda_1 u(\sum_{m=0}^{\infty}A_{m,1}t^m),\]
    where $\lambda_1=\frac{\lambda}{E(u)}\in K[[t]]^{\times}$
\end{theorem}
\begin{remark}
By construction, our $N_{\nabla}$ is just the monodromy operator on $(D_{\mathrm{Sen},K_{\infty}}(V(\mathbb{M}))$  after modulo $t$, where $V(\mathbb{M})=\mathbb{M}(\mathfrak{S}\otimes_{K} \mathbb{C}_p)\in \Rep_{\mathbb{C}_p}(G_K)$ is associated to the rational Hodge Tate crystal $\mathbb{M}=\mathcal{M}/\mathcal{I}$. Then one can check that this coincides with the calculation in \cite[Theorem 4.3.3]{Gao22}.
\end{remark}
\begin{proof}[Proof of \cref{operator and stratification}]
By \cref{the image of t}, $\theta(\mathfrak{t})=\theta(\frac{\log([\varepsilon])}{p\lambda})$ is a unit, hence $\lambda_1$ is a unit.

It suffices to show that for $\forall k\in \mathbb{N}$
\[N_{\nabla,k}(\underline e)=\underline e \theta_{k}(-\lambda_1 u)(\sum_{m=0}^{k-1}A_{m,1}t^m).\]
For this, by \cref{image of X under theta}
$$\theta_k(X(g))=\theta_{k}(\frac{[\pi^{\flat}](1-[\varepsilon])}{E([\pi^{\flat}])})(\sum_{i=1}^{k}\frac{(-1)^{i-1}c(g)\cdots (c(g)-i+1)}{i!}(1-[\varepsilon])^{i-1}).$$
Hence utilizing that $c(\tau^{p^n})=p^n$ and by calculating the coefficient of $c(g)$ in $\theta_k(X(g))$, we have that
\begin{equation*}
    \begin{split}
        \nabla_{\tau,k}(\underline e)&=\lim_{n\rightarrow \infty}\frac{\tau^{p^n}(\underline e)-\underline e}{p^n}=\lim_{n\rightarrow \infty}\frac{\underline e U(\tau^{p^n})-\underline e}{p^n}
        \\&=\lim_{n\rightarrow \infty}\frac{\underline e (\sum_{m=0}^{k-1}\sum_{n=0}^{\infty}(A_{m,n}\theta_k(X_1(\tau^{p^n}))^{[n]})t^m)-\underline e}{p^n}
        \\&=\underline e (\sum_{m=0}^{k-1}A_{m,1}t^m)(\sum_{i=1}^{k}\frac{(1-[\varepsilon])^{i-1}}{i}) \theta_{k}(\frac{[\pi^{\flat}](1-[\varepsilon])}{E([\pi^{\flat}])}).
    \end{split}
\end{equation*}
The second to last equality follows as $A_{0,1}=I$ and $A_{m,0}=0$ for $m>0$ by \cref{Main theorem I}.

Hence 
\begin{equation*}
    \begin{split}
        N_{\nabla,k}(\underline e)&=\frac{1}{p\theta_k(\mathfrak{t})}\nabla_{\tau,k}(\underline e)=\frac{\theta_k(E(u)\lambda_1)}{\theta_k(\log([\varepsilon]))}\nabla_{\tau,k}(\underline e)
        \\&=\underline e (\sum_{m=0}^{k-1}A_{m,1}t^m) \frac{\theta_k(u\lambda_1)}{\theta_k(\log([\varepsilon]))}(\sum_{i=1}^{k}\frac{(1-[\varepsilon])^{i}}{i})
        \\&=\underline e (\sum_{m=0}^{k-1}A_{m,1}t^m)\theta_k(-u\lambda_1).
    \end{split}
\end{equation*}
This is precisely the desired result.
\end{proof}
Motivated by this calculation, we see that $V(\mathcal{M})$ needs to be strong nearly de Rham (see \cref{strong}). Actually, we have the following result:

\begin{prop}\label{weakly essential surjectivity}
Consider the composition of the de Rham realization functor and the restriction functor, which we still denoted as $V$ by abuse of notation, then we get a fully faithful functor 
\begin{equation*}
     V: \operatorname{Vect}((\mathcal{O}_{K})_{\Prism}, (\mathcal{O}_{\Prism}[\frac{1}{p}])_{\mathcal{I}}^{\wedge})\longrightarrow \Rep_{B_{\dR}^+}^{\fp,SndR}(G_K).
\end{equation*}
Moreover, assume that \cref{conjecture e=1} holds, then $ V$ is essentially surjective, in particular, it induces an equivalence of categories between $\operatorname{Vect}((\mathcal{O}_{K})_{\Prism}, (\mathcal{O}_{\Prism}[\frac{1}{p}])_{\mathcal{I}}^{\wedge})$ and $\Rep_{B_{\dR}^+}^{\fp,SndR}(G_K)$.
\end{prop}
\begin{proof}
A priori is that $ V$ defines a fully faithful functor $\operatorname{Vect}((\mathcal{O}_{K})_{\Prism}, (\mathcal{O}_{\Prism}[\frac{1}{p}])_{\mathcal{I}}^{\wedge})\longrightarrow \Rep_{B_{\dR}^+}^{\fp}(G_K)$ by \cref{de Rham realization} and \cref{fully faithful of the restriction functor}. Utilizing \cref{specializing Hodge Tate} and \cref{Gao theorem}, we see the image of $ V$ lands into $\Rep_{B_{\dR}^+}^{\fp,ndR}(G_K)$. Further thanks to \cref{decompletion for de Rham crystals} and \cref{operator and stratification}, 
we see the target of $ V$ is in $\Rep_{B_{\dR}^+}^{\fp,SndR}(G_K)$.

Next we show that $ V$ is essentially surjective assuming that \cref{conjecture e=1} is true. Given $W\in \Rep_{B_{\dR}^+}^{\fp,ndR}(G_K)$ of rank $l$, by definition there exists a $N_{\nabla}$ stable $K[[t]]$-module $\tilde{M}$ inside $D_{\mathrm{Sen},K_{\infty}[[t]]}(W)$, fix a basis $\underline{\tilde{e}}$ of $\tilde{M}$, as $-\lambda_1 u$ is a unit in $K[[t]]$, we can find $\{B_{m,1}\} \in M_{l\times l}$ such that
\begin{equation}\label{weakly essential surjectivity I}
    N_{\nabla}(\underline{\tilde{e}})=-\underline{\tilde{e}} \lambda_1 u(\sum_{m=0}^{\infty}B_{m,1}t^m).
\end{equation}
 As a result, $N_{\nabla}$ on $$D_{\mathrm{Sen},K_{\infty}}(W/tW)=D_{\mathrm{Sen},K_{\infty}[[t]]}(W)/t=\tilde{M}/t\tilde{M}\otimes_K K_{\infty}$$ is given by $$N_{\nabla}(\underline{\tilde{e}})=-\underline{\tilde{e}} \theta(\lambda_1 u)B_{0,1}.$$
By \cite[Theorem 4.3.3]{Gao22}, the (normalized) Kummer-Sen operator on $D_{\mathrm{Sen},K_{\infty}}(W/tW)$ is just $\frac{1}{\theta(u\lambda^{\prime})}N_{\nabla}$, where $\lambda^{\prime}$ is the derivative of $\lambda$ with respect to $u$. Hence the Kummer-Sen operator on $D_{\mathrm{Sen},K_{\infty}}(W/tW)$ acts on $\underline{\tilde{e}}$ via
\begin{equation*}
   - \frac{\theta(\lambda_1 u)}{\theta(u\lambda^{\prime})}B_{0,1}=- \frac{1}{\theta(\lambda^{\prime})/\theta(\lambda_1)}B_{0,1}=-\frac{B_{0,1}}{\beta}.
\end{equation*}
The last equality holds as $\lambda=\lambda_1 E(u)$, hence $\lambda^{\prime}=\lambda_1 E^{\prime}(u)+\lambda_1^{\prime}E(u)$, hence $\theta(\lambda^{\prime})=\theta(\lambda_1)E^{\prime}(\pi)=\theta(\lambda_1)\beta$.

On the other hand, as $W$ is nearly de Rham, hence $W/tW$ is nearly Hodge-Tate, i.e. the eigenvalues of the Sen operator are all in the subset $\mathbb{Z}+\beta^{-1}\mathfrak{m}_{\mathcal{O}_{\overline{K}}}$, from which we see all the eigenvalues of $B_{0,1}$ are in the set $\beta\mathbb{Z}+\mathfrak{m}_{\mathcal{O}_{\overline{K}}}$, which implies that $\lim_{n\to+\infty}\prod_{i=0}^n(iE'(\pi)+B_{0,1}) = 0$ by the proof of \cite[Corollary 4.1.6]{Gao22}. 

Now we have that $\{B_{m,1}\}$ satisfies the assumption in \cref{conjecture e=1}. As we assume that \cref{conjecture e=1} is true, hence we can construct a de Rham crystal $\mathcal{M}$ whose corresponding stratification $\varepsilon$ is determined by $\{B_{m,1}\}$ as in \cref{conjecture e=1}. Now let $M=\mathcal{M}(\mathfrak{S}, E)$ and $\underline e$ be a basis of $M$ on which $\varepsilon$ acts as in \cref{conjecture e=1}.
Then by \cref{operator and stratification},
\begin{equation}\label{weakly essential surjectivity II}
    N_{\nabla}(\underline e)=-\underline e \lambda_1 u(\sum_{m=0}^{\infty}B_{m,1}t^m).
\end{equation}
In this way, we can construct a $K[[t]]$-linear isomorphism $f: \tilde{M} \rightarrow M$ by sending $\tilde{e}_i$ to $e_i$, which is compatible with $N_{\nabla}$-structure by \cref{weakly essential surjectivity I} and \cref{weakly essential surjectivity II}. This can be further $K_{\infty}[[t]]$-linearly extended to $f\otimes \Id: \tilde{M}\otimes_{K[[t]]} K_{\infty}[[t]]\rightarrow M\otimes_{K[[t]]} K_{\infty}[[t]]$. Then notice that
\begin{equation*}
    \begin{split}
        \tilde{M}\otimes_{K[[t]]}K_{\infty}[[t]]&=D_{\mathrm{Sen},K_{\infty}[[t]]}(W),
        \\M\otimes_{K[[t]]} K_{\infty}[[t]]&=D_{\mathrm{Sen},K_{\infty}[[t]]}(V(\mathcal{M})),
    \end{split}
\end{equation*}
where the first equality follows from the definition of strong nearly de Rham crystals , while the second identity holds thanks to \cref{decompletion for de Rham crystals}.
In this way, $f\otimes \Id$ defines an isomorphism in
\begin{equation*}
    \Hom_{\Mod_{K_{\infty}[[t]]}^{N_{\nabla}}}((D_{\mathrm{Sen},K_{\infty}[[t]]}(W_1), N_{\nabla}),(D_{\mathrm{Sen},K_{\infty}[[t]]}(W_2), N_{\nabla})).
\end{equation*}
This implies $W$ is isomorphic to $V(\mathcal{M})$ in $\rb$ thanks to \cref{detect isomorphic class}. Since $W$ is arbitrarily chosen, we see that $V$ is essentially surjective.
\end{proof}

\bibliographystyle{amsalpha}
\bibliography{template}

\begin{thebibliography}{{Sta}20}

\bibitem[BC09]{BC09}
Olivier Brinon and Brian Conrad, \emph{CMI Summer School Notes on $p$-adic Hodge theory}, available at \url{https://math.stanford.edu/~conrad/papers/notes.pdf}.
\bibitem[BC16]{BC16}
Laurent Berger and Pierre Colmez, \emph{Théorie de Sen et vecteurs localement analytiques} Ann. Sci. Éc. Norm. Supér.(4) 49, no. 4 (2016): 947-970.

\bibitem[Ber16]{Ber16}
Laurent Berger, \emph{Multivariable $(\varphi,\Gamma)$-modules and locally analytic vectors} Duke Mathematical Journal 165, no. 18 (2016): 3567-3595.


\bibitem[BS15]{BS15}
Bhargav Bhatt and Peter Scholze, \emph{The pro-\'etale topology for schemes} arXiv preprint arXiv:1309.1198 (2013). available at \url{https://arxiv.org/pdf/1309.1198.pdf}.



\bibitem[BS19]{BS19}
Bhargav Bhatt, Peter Scholze \emph{Prisms and prismatic cohomology}, 2019, arXiv:1905.08229,
  available at \url{https://arxiv.org/abs/1905.08229}.
  
\bibitem[BS21]{BS21}
Bhargav Bhatt and Peter Scholze, \emph{Prismatic $ F $-crystals and crystalline Galois representations} arXiv preprint arXiv:2106.14735 (2021). available at \url{https://arxiv.org/pdf/2106.14735.pdf}.

\bibitem[DL21]{DL21}
Heng Du and Tong Liu, \emph{A prismatic approach to $(\varphi,\hat G)$-modules and $F $-crystals} arXiv preprint arXiv:2107.12240 (2021). available at \url{https://arxiv.org/pdf/2107.12240.pdf}.

\bibitem[DLMS22]{DLMS22}
Heng Du, Tong Liu, Yong Suk Moon, and Koji Shimizu. \emph{Completed prismatic $F $-crystals and crystalline $\mathbf{Z}_p $-local systems} arXiv preprint arXiv:2203.03444 (2022). available at \url{https://arxiv.org/pdf/2203.03444.pdf}.

\bibitem[Fon04]{Fon04}
Fontaine, Jean-Marc. \emph{Arithmétique des représentations galoisiennes p-adiques} Astérisque 295 (2004): 1-115.

\bibitem[Gao22]{Gao22}
Hui Gao, \emph{Hodge-Tate prismatic crystals and Sen theory} arXiv preprint arXiv:2201.10136 (2022). available at \url{https://arxiv.org/pdf/2201.10136.pdf}.


\bibitem[GP21]{GP21}
Gao, Hui, and Léo Poyeton, \emph{Locally analytic vectors and overconvergent $(\varphi,\tau)$-modules}, Journal of the Institute of Mathematics of Jussieu 20.1 (2021): 137-185.
\bibitem[GR22]{GR22}
Haoyang Guo and Emanuel Reinecke, \emph{Prismatic $ F $-crystals and crystalline local systems} arXiv preprint arXiv:2203.09490 (2022). available at \url{https://arxiv.org/pdf/2203.09490.pdf}.


\bibitem[KL16]{KL16}
Kiran S. Kedlaya and Ruochuan Liu. \emph{Relative p-adic Hodge theory, II: Imperfect period rings} arXiv preprint arXiv:1602.06899 (2016). available at \url{https://arxiv.org/pdf/1602.06899.pdf}.

\bibitem[KP21]{KP21}
Dmitry Kubrak, Artem Prikhodko, \emph{p-adic Hodge theory for Artin stacks}, 2021, arXiv:2105.05319 (2021).,
  available at \url{https://arxiv.org/abs/2105.05319}.
 
 
 
\bibitem[MW21a]{MW21a}
Yu Min and Yupeng Wang, \emph{Relative ($\phi$, $\Gamma$)-modules and prismatic F-crystals} arXiv preprint arXiv:2110.06076 (2021). available at \url{https://arxiv.org/pdf/2110.06076.pdf}.

\bibitem[MW21]{MW21}
Yu Min and Yupeng Wang, \emph{On the Hodge--Tate crystals over $\mathcal{O}_K$} arXiv preprint arXiv:2112.10140 (2021). available at \url{https://arxiv.org/pdf/2112.10140.pdf}.

\bibitem[Sch13]{Sch13}
Peter Scholze, \emph{$p$-adic Hodge theory for rigid-analytic varieties} In Forum of Mathematics, Pi, vol. 1. Cambridge University Press, 2013.

\bibitem[SP22]{SP22}
The Stacks Project Authors, Stacks Project, \url{https://stacks.math.columbia.edu}, 2022.



\bibitem[Tian21]{Tian21}
Yichao Tian, \emph{Finiteness and duality for the cohomology of prismatic crystals} arXiv preprint arXiv:2109.00801 (2021). available at \url{https://arxiv.org/pdf/2109.00801.pdf}.
\bibitem[Wu21]{Wu21}
Zhiyou Wu, \emph{Galois representations, ($\phi$, $\Gamma$) modules and prismatic F-crystals} arXiv preprint arXiv:2104.12105 (2021),
  available at \url{https://arxiv.org/pdf/2104.12105.pdf}.

\end{thebibliography}

\end{document}